\documentclass{amsart}

\usepackage{graphicx}
\usepackage{subfigure}
\usepackage[english]{babel}
\numberwithin{equation}{section}

\newtheorem{theorem}{Theorem}[section]
\newtheorem{lemma}[theorem]{Lemma}

\theoremstyle{definition}
\newtheorem{definition}[theorem]{Definition}

\theoremstyle{remark}
\newtheorem{remark}[theorem]{Remark}

\numberwithin{equation}{section}
\newcommand{\ds}{\displaystyle}

\begin{document}

\title{VARIABLE ORDER MIXED $\boldsymbol{h}$-FINITE ELEMENT METHOD
FOR LINEAR ELASTICITY WITH WEAKLY IMPOSED SYMMETRY.
II. AFFINE AND CURVILINEAR ELEMENTS IN 2D}

\author{Weifeng Qiu}
\address{The Institute for Computational Engineering and Sciences, 
the University of Texas at Austin, Austin, Texas 78712}
\email{qiuw@ices.utexas.edu}

\author{Leszek Demkowicz}
\address{The Institute for Computational Engineering and Sciences, 
the University of Texas at Austin, Austin, Texas 78712}
\email{leszek@ices.utexas.edu}

\begin{abstract}
We continue our study on variable order Arnold-Falk-Winther elements for 2D elasticity
in context of both affine and parametric curvilinear elements. We present an $h$-stability 
result for affine elements, and an asymptotic stability result for curvilinear
elements. Both theoretical results are confirmed with numerical experiments.
\end{abstract}



\maketitle

\pagestyle{myheadings}
\thispagestyle{plain}
\markboth{WEIFENG QIU AND LESZEK DEMKOWICZ}{MIXED FEM FOR 2D LINEAR ELASTICITY}


\section{\bigskip Introduction}

\subsection{Elasticity problem}
Linear elasticity is a classical subject, and it has been studied
for a long time. The presented research  is motivated with a class
of time-harmonic problems formulated as follows.
Let $\Omega \subset \mathbb{R}^n,\: n=2,3$,
denote a bounded domain occupied by the elastic body. 
Assume that the boundary of $\Omega$, $\Gamma = \partial \Omega$ has been
split into two disjoint subsets $\Gamma_1$ and $\Gamma_2$,
\begin{equation*}
\Gamma = \overline{\Gamma_1} \cup \overline{\Gamma_2},\quad
\Gamma_1 \cap \Gamma_2 = \emptyset
\end{equation*}
Subsets $\Gamma_1,\Gamma_2$ are assumed to be (relatively) open
in $\Gamma$.
We seek:
\begin{itemize}
  \item displacement vector $u_i(x),\:  x \in \Omega$,
  \item linearized strain tensor $\epsilon_{ij}(x),\:  x \in \Omega$, and
  \item stress tensor $\sigma_{ij}(x),\: x \in \Omega$
\end{itemize}
that satisfy the following system of equations and boundary conditions.
\begin{itemize}
  \item Cauchy's geometrical relation between the displacement and strain,
\begin{equation}
\epsilon_{ij} = \dfrac{1}{2} (u_{i,j} + u_{j,i}),\quad x \in \Omega
\label{eq:geometrical_relations}
\end{equation}
  \item Equations of motion resulting from the principle of linear
momentum,
\begin{equation}
 - \sigma_{ij,j} - \rho(x) \omega^2 u_i = f_i(x),\quad x \in \Omega
\label{eq:momentum_equations}
\end{equation}
  \item Symmetry of the stress tensor being a consequence of the
principle of angular momentum,
\begin{equation*}
\sigma_{ij} = \sigma_{ji},\quad x \in \Omega
\end{equation*}
  \item Constitutive equations for linear elasticity,
\begin{equation*}
\sigma_{ij} = E_{ijkl}(x) \epsilon_{kl},\quad x \in \Omega
\end{equation*}
  \item Kinematic boundary conditions,
\begin{equation*}
u_i = u^0_i,\quad x \in \Gamma_1
\end{equation*}
  \item Traction boundary conditions,
\begin{equation*}
\sigma_{ij} n_j = g_i(x), \quad x \in \Gamma_2
\end{equation*}
\end{itemize}
Here:
\begin{itemize}
  \item $\rho$ is the density of the body,
  \item $f_i$ are volume forces prescribed within the body,
  \item $g_i$ are tractions prescribed on $\Gamma_2$ part of the
boundary,
  \item $u^0_i$ are displacements prescribed on $\Gamma_1$ part of the
boundary,
  \item $n_j$ is the unit outward normal vector for boundary $\Gamma$, and
  \item $E_{ijkl}$ is the tensor of elasticities.
\end{itemize}
As usual, commas denote partial derivatives, and we use the Einstein's
summation convention. Vector $t_i := \sigma_{ij}n_j$ is known
as the traction or stress vector.

The geometrical relations imply the symmetry of the strain tensor.
The symmetry of strain and stress tensors imply then the usual
(minor) symmetry conditions for the elasticities,
\begin{equation}
E_{ijkl} = E_{jikl} = E_{ijlk}
\label{eq:minor_symmetries}
\end{equation}
The laws of thermodynamics imply an additional (major) symmetry condition,
\begin{equation*}
E_{ijkl} = E_{klij}
\end{equation*}
and positive definiteness condition,
\begin{equation*}
E_{ijkl} \xi_{ij} \xi_{kl} > 0,\quad \forall \xi_{ij} = \xi_{ij}
\end{equation*}
The last condition implies that the constitutive equation can be
inverted to represent strains in terms of stresses,
\begin{equation}
\epsilon_{kl} = C_{klij} \sigma_{ij}
\label{eq:constitutive_equation_inverted}
\end{equation}
where $C_{klij} = E_{ijkl}^{-1}$ is known as the compliance tensor
and satisfies analogous symmetry and positive-definiteness properties.

For isotropic materials, the elasticities tensor is expressed in terms
of two independent Lame's constants $\mu$ and $\lambda$,
\begin{equation*}
E_{ijkl} = \mu (\delta_{ik}\delta_{jl} + \delta_{il}\delta_{jk})
+ \lambda \delta_{ij} \delta_{kl}
\end{equation*}
and the constitutive equation reduces to,
\begin{equation*}
\sigma_{ij} = 2 \mu \epsilon_{ij} + \lambda \epsilon_{kk} \delta_{ij}
\label{Hooke_law}
\end{equation*}
where $\delta_{ij}$ denotes the Kronecker's delta. 
Inverting, we obtain
\begin{equation*}
\epsilon_{ij} = A \sigma_{ij},\quad \epsilon_{ij}  = \frac{1}{2\mu} \sigma_{ij} 
- \frac{\lambda}{2\mu (2\mu + n \lambda)} \sigma_{kk} \delta_{ij}
\end{equation*}
Of particular interest is the case of nearly incompressible material
corresponding to $\lambda \to \infty$. Notice that the norm of the
elasticities blows up then to infinity but the norm of the
compliance tensor remains uniformly bounded. This suggests that
formulations based on the compliance relation have a chance to remain
uniformly stable for nearly incompressible materials. Another reason 
of using formulations based on compliance relation is for (visco)elastic vibration 
problems for structures with large material contrast. For a more detailed discussion on
these problems, we refer to the introduction in \cite{QD:2009:MMEW}.

\vspace*{1cm}
\subsection{Variational formulations for elasticity}

\subsubsection*{Dual-Mixed Formulation}

We eliminate the strain tensor and combine 
geometrical relations~(\ref{eq:geometrical_relations})
and constitutive equation~(\ref{eq:constitutive_equation_inverted})
in a single equation,
\begin{equation*}
C_{klij}\sigma_{ij} = \frac{1}{2}(u_{k,l} + u_{l,k})
\end{equation*}
Classical formulation for elasticity is based on satisfying the
equation above in a strong sense while relaxing the momentum equations.
The idea behind the dual-mixed formulation for elasticity is exactly opposite,
the equation above is relaxed whereas momentum equations~(\ref{eq:momentum_equations})
are satisfied in a strong sense.

The final formulation includes building in kinematic boundary conditions
and it reads as follows.
\begin{equation}
\left\{
\begin{array}{lllll}
\ds \sigma \in H(\text{div},\Omega,\mathbb{S}) \: : \: \sigma_{ij}n_j & =
  g_i  \text{ on } 
\Gamma_2,\: 
 u \in L^2(\Omega,\mathbb{V})\\[8pt]
\ds  \int_{\Omega} C_{ijkl} \sigma_{kl} \tau_{ij}
 + \int_{\Omega} u_i \tau_{ij,j} & \ds = \int_{\Gamma_1} u^0_i \tau_{ij} n_j \quad
\forall \tau \in H(\text{div},\Omega,\mathbb{S}) 
\: : \: \tau_{ij}n_j = 0 \text{ on } \Gamma_2 \\[8pt]
\ds  - \int_{\Omega} \sigma_{ij,j} v_i 
- \omega^2 \int_{\Omega} \rho u_i v_i
&   \ds = \int_{\Omega} f_i v_i \quad
\forall v \in L^2(\Omega,\mathbb{V}) \\[8pt]
\end{array}
\right.
\label{complete_mixed_formulation_strong_symmetry}
\end{equation}
Above, $\mathbb{V} = \mathbb{R}^3$ and $\mathbb{S}$ denotes the space of symmetric
tensor. $L^2(\Omega,\mathbb{V})$ denotes the space of square integrable functions
with values in $\mathbb{V}$, and $H(\text{div},\Omega,\mathbb{S})$ denotes
the space of square-integrable functions with values in $\mathbb{S}$, whose
row-wise divergence lives in $L^2(\Omega,\mathbb{V})$.
The traction conditions 
are satisfied in the sense of traces for functions
from $H({\rm div},\Omega)$ (they live in $H^{-1/2}(\Gamma)$).

The corresponding static case can be derived formally by considering the so-called
Hellinger-Reissner functional, and the variational formulation is frequently
identified as the Hellinger-Reissner variational principle.

\subsubsection*{Dual-Mixed Formulation with Weakly Imposed Symmetry}

The symmetry condition is difficult to enforce on the discrete level.
This has led to the idea of relaxing the symmetric function and satisfying
it in a weaker, integral form. This is obtained by introducing
tensor-valued test functions $q$ with values in the space of {\em antisymmetric} 
tensors $\mathbb{K} := \{ q_{ij} \: : \: q_{ij} = - q_{ji}\}$, and replacing
the symmetry condition with an integral condition,
\begin{equation*}
\int_{\Omega} \sigma_{ij} q_{ij} = 0,\quad \forall q \in L^2(\Omega,\mathbb{K})
\end{equation*}
On the continuous level, the integral condition implies the pointwise
condition (understood in the $L^2$ sense), but on the discrete level,
with an appropriate choice of spaces, the integral condition does not necessary
imply the symmetry condition pointwise. 

The extra condition leads to an extra unknown. The derivation of the weak form of the
constitutive equation has to be revisited. We start by introducing the
{\em tensor of infinitesimal rotations},
\begin{equation*}
p_{ij} = \dfrac{1}{2}(u_{i,j} - u_{j,i})
\end{equation*}
Upon eliminating the strain tensor, the constitutive equation in the 
compliance form is now rewritten as,
\begin{equation*}
C_{ijkl} \sigma_{kl} = \dfrac{1}{2} (u_{i,j} + u_{j,i})
= u_{i,j} - p_{ij}
\end{equation*}
Multiplication with a test function $\tau$ (now, not necessarily symmetric),
integration over $\Omega$, and integration by parts, leads to a new
relaxed version of the equation,
\begin{equation*}
\int_{\Omega} C_{ijkl} \sigma_{kl} \tau_{ij}
 = - \int_{\Omega} u_i \tau_{ij,j} + \int_{\Gamma} u_i \tau_{ij} n_j
- \int_{\Omega} p_{ij} \tau_{ij}
\end{equation*}
We obtain a new variational
formulation in the form:
\begin{equation}
\left\{
\begin{array}{lllll}
\ds \sigma \in H({\rm div},\Omega,\mathbb{M}) \: : \: \sigma_{ij}n_j =
  g_i  \mbox{ on } 
\Gamma_2,\: 
 u \in L^2(\Omega,\mathbb{V}),\:
p \in L^2(\Omega,\mathbb{K})
\\[10pt]
\ds \int_{\Omega} C_{ijkl} \sigma_{kl} \tau_{ij}+u_i \tau_{ij,j}+p_{ij} \tau_{ij}
= \int_{\Gamma_1} u^0_i \tau_{ij} n_j 
\quad\forall \tau \in H({\rm div},\Omega,\mathbb{M}) 
: \tau_{ij}n_j = 0 \mbox{ on } \Gamma_2 \\[10pt]
\ds - \int_{\Omega} \sigma_{ij,j} v_i 
- \omega^2 \int_{\Omega} \rho u_i v_i
 = \int_{\Omega} f_i v_i \quad 
\forall v \in L^2(\Omega,\mathbb{V}) \\[10pt]
\ds  \int_{\Omega} \sigma_{ij} q_{ij} = 0\quad \forall q \in L^2(\Omega,\mathbb{K})
\end{array}
\right.
\label{complete_mixed_formulation_weak_symmetry}
\end{equation}
In the above, $H({\rm div},\Omega,\mathbb{M})$ denotes the space
of tensor-valued square integrable fields with a square integrable
divergence.
The formulation for the static case can again be derived formally by looking 
for a stationary point
of the so-called {\em Generalized Hellinger-Reissner} functional,
and it is frequently identified as the 
{\em Generalized Hellinger-Reissner Variational Principle}.

In the analysis presented in this paper, 
we consider the static case only and, for the sake of simplicity 
we assume that the body is fixed on the entire boundary $\partial\Omega$,
i.e. $\Gamma_2 = \emptyset$.

To connect with the notations used in \cite{AFW:2007:MMEW},\cite{AAW:2008:MMES}, 
we denote by $A$ the compliance operator.
\begin{equation*}
A:\mathbb{M}\ni\varepsilon_{ij}\longrightarrow\sigma_{ij}=C_{ijkl}\varepsilon_{kl}\in\mathbb{M}.
\end{equation*}

The operator $A$ maps tensors to tensors, is bounded, symmetric and,
for piecewise constant material properties, uniformly 
positive definite. Symmetry properties~(\ref{eq:minor_symmetries})
imply that $A$ maps $\mathbb{S}$
into itself.
In the isotropic case, $A$ has the form
\[
A\sigma=\frac{1}{2\mu}\left(  \sigma-\frac{\lambda}{2\mu+n\lambda}Tr\left(
\sigma\right)  I\right)  ,
\]
where $\lambda\left(  x\right)  ,\mu\left(  x\right)  $ are the Lam\'{e} coefficients.

We can rewrite then formulation (\ref{complete_mixed_formulation_strong_symmetry})
with a more compact, index free notation.

Find $\sigma\in H\left(  \text{div},\Omega;\mathbb{S}\right)  $,
and $u\in L^{2}\left(  \Omega;\mathbb{V}\right)  $, satisfying
\begin{gather}
\int_{\Omega}\left(  A\sigma:\tau+\text{div}\tau\cdot u\right)  \: dx=0,\text{
\ }\tau\in H\left(  \text{div},\Omega;\mathbb{S}\right)
,\label{strong_symm_formula_continuous}\\
\int_{\Omega}\text{div}\sigma\cdot v \: dx=\int_{\Omega}f\cdot v\: dx,\text{ \ }v\in
L^{2}\left(  \Omega;\mathbb{V}\right)  .\nonumber
\end{gather}

Similarly, the formulation (\ref{complete_mixed_formulation_weak_symmetry}) 
becomes to seek $\sigma\in
H\left(  \text{div},\Omega;\mathbb{M}\right)  $, $u\in L^{2}\left(
\Omega;\mathbb{V}\right)  $, and $p\in L^{2}\left(  \Omega;\mathbb{K}\right)
$ satisfying
\begin{gather}
\int_{\Omega}\left(  A\sigma:\tau+\text{div}\tau\cdot u+\tau:p\right)\: 
dx=0,\text{ \ }\tau\in H\left(  \text{div},\Omega;\mathbb{M}\right)
,\label{weak_symm_formula_continuous}\\
\int_{\Omega}\text{div}\sigma\cdot v\: dx=\int_{\Omega}f\cdot v\: dx,\text{ \ }v\in
L^{2}\left(  \Omega;\mathbb{V}\right)  ,\nonumber\\
\int_{\Omega}\sigma:q\: dx=0,\text{\ \ }q\in L^{2}\left(  \Omega;\mathbb{K}
\right)  .\nonumber
\end{gather}

\subsubsection*{Dual-Mixed Formulation with Weakly Imposed Symmetry for Plane Elasticity}

In two space dimensions, the skew-symmetric tensors involve a single non-zero
component $p$,
\begin{equation*}
\left(
\begin{array}{cc}
0 & p \\
-p & 0
\end{array}
\right)
\end{equation*}
The formulation (\ref{weak_symm_formula_continuous})  
reduces to seek $\sigma\in
H(\text{div},\Omega;\mathbb{M})$, $u\in L^{2}(\Omega;\mathbb{V})$, 
and $p\in L^{2}(\Omega)$ satisfying
\begin{gather}
\int_{\Omega}(A\sigma:\tau+\text{div}\tau\cdot u- S_{1}\tau \: p)\:
dx=0,\text{ \ }\tau\in H(\text{div},\Omega;\mathbb{M})
,\label{weak_symm_formula_continuous_plane}\\
\int_{\Omega}\text{div}\sigma\cdot v\: dx=\int_{\Omega}f\cdot v\: dx,\text{ \ }v\in
L^{2}(\Omega;\mathbb{V})  ,\nonumber\\
\int_{\Omega}S_{1}\sigma\:  q\: dx=0,\text{\ \ }q\in L^{2}(\Omega)  .\nonumber
\end{gather}
where operator $S_{1}$ maps a real $2\times 2$ matrix to a real number.
For any $\sigma\in\mathbb{R}^{2\times 2}$, 
\begin{equation}
\label{operator_S1}
S_{1}\sigma = \sigma_{12} - \sigma_{21}.
\end{equation}

\subsection{Review of the existing work}
Besides the Hellinger-Reissner and generalized Hellinger-Reissner principles discussed 
in the previous sections, there are other variational formulations for elasticity,
see \cite{ODENREDDY:1976:VMTM}. 

Stable finite element discretizations based on
(\ref{strong_symm_formula_continuous}) are difficult to construct. A
detailed description of related challenges can be founded in \cite{Falk:2008:FME}. 
A stable discretization based on the strong enforcement of symmetry condition
was developed in \cite{AAW:2008:MMES}. For the lowest order element, the number 
of degrees of freedom for stress tensor in 3D is $162$.

Meanwhile, a number of authors have developed approximation schemes based on 
(\ref{weak_symm_formula_continuous}). See \cite{AFW:2006:ECH,AFW:2007:MMEW,
Falk:2008:FME,AT:1979:EFEE,ABD:1984:MFEPE,AF:1988:NMFEE,FarhloulFortin:1997:DHES,
FdV:1975:SFA,Morley:1989:MFEE,SteinRolfes:1990:SOMFEPE,
Stenberg:1986:COMME,Stenberg:1988:FMME,Stenberg:1988:TLOMME}.
For a brief description of these methods, we refer to the 
introduction in \cite{AFW:2007:MMEW}. Recently, Cockburn, Gopalakrishnan 
and Guzman have developed a new mixed method for linear elasticity using a hybridized 
formulation based on (\ref{weak_symm_formula_continuous}) in \cite{CGG:2009:NEEMW}. 


The work presented in this paper is based on mixed finite element methods
developed by Arnold, Falk and Winther for the formulation
(\ref{weak_symm_formula_continuous}) in \cite{AFW:2006:ECH,Falk:2008:FME,AFW:2007:MMEW}. 
We restrict our analysis to plane linear elasticity problem only.

At the first glance, a generalization to elements of variable order seems to be 
straightforward. And this is indeed the case from the implementation point
of view. The stability analysis, however, is much more difficult. 
The canonical projection operators defined in \cite{AFW:2006:ECH} do not commute with
exterior derivative operators for variable order elements, a property essential in
the proof of discrete stability. In \cite{QD:2009:MMEW}, we resolved the problem 
by invoking projection based (PB) interpolation operators. Unfortunately, the PB operators 
do not commute with the operator $S_{n-2}$, invalidating construction
of operator $\tilde{W}_h$ introduced in the stability analysis presented 
in \cite{AFW:2006:ECH}
We resolved the problem in \cite{QD:2009:MMEW} by designing a new operator $\tilde{W}_{h}$. 
However, in our first contribution 
we were only able to prove the well-definedness 
of $\tilde{W}_{h}$ for $n=2$ with orders varying from $0$ to $3$.
In this paper, we propose PB operators of new type and a new operator 
$\tilde{W}_{h}$ that enable the $h$-stability analysis for 2D elements of variable
order\footnote{With an upper bound on the polynomial order.}.
Additionally, we analyze curvilinear elements and establish conditions for their
asymptotic $h$-stability.

\subsection{Scope of the paper}
An outline of the paper is as follows. Section 2 introduces notations.
Section 3 reviews definitions of finite element spaces on affine meshes.
In Section 4, we return to the mixed formulation for plane elasticity 
with weakly imposed symmetry, and recall Brezzi's conditions for stability.
In Section 5, we establish necessary results for proving the stability for 
affine meshes. We design new PB operators and a new operator $\tilde{W}_{h}$. 
In Section 6 we prove that the Brezzi conditions are satisfied for affine meshes.
The main result for affine meshes is Theorem~\ref{thm:main_theorem_affine}.
In Section 7, we discuss curvilinear meshes and geometry assumptions.
Section 8 reviews definitions of finite element spaces on a reference triangle, 
on a (physical) curved triangle, and on curved meshes.
In Section 9, we establish necessary results for proving the asymptotic stability for 
curvilinear meshes. This part of work is much more technical than that of Section 5 
because of the curvilinear meshes. We explain briefly the substantial difference 
between the analysis for curvilinear meshes and that for affine meshes at the beginning of 
Section~\ref{preliminary_curve}. In Section 10 we prove that the Brezzi conditions are 
satisfied asymptotically for curvilinear meshes,
i.e. for meshes that are sufficiently fine.
We conclude the paper with numerical experiments that illustrate and verify the 
presented theoretical results.

\section{Notations}

We denote by $T$ an arbitrary triangle in $\mathbb{R}^{2}$. And let $\hat{T}$ 
be the reference triangle.
We denote the set of subsimplexes of dimension $k$ of $T$ by $\triangle_{k}(T),
\: k=0,1,2$, 
and the set of all subsimplexes of $T$ by $\triangle (T)$. 
$\triangle_{0}(T)$ consists of all vertices of $T$, $\triangle_{1}(T)$ 
consists of all edges of $T$, and $\triangle_{2}(T)=\{T\}$. 

For $U$, a bounded open subset in $\mathbb{R}^{n}$, we define:
$$
C^{k}(\overline{U})= \{u\in C^{k}(U):D^{\alpha}\text{ is uniformly 
continuous on }U,\forall\vert\alpha\vert\leq k\}.
$$
We also define 
$$
C^{k,1}(\overline{U})=\{u\in C^{k}(\overline{U}):D^{\alpha}\text{ is Lipschitz on }
U,\forall\vert\alpha\vert =k\}.
$$
$\mathbb{M}$ will denote the space of $2\times 2$ real matrices. For any 
vector space $\mathbf{X}$, we denote by $L^{2}(\Omega;\mathbf{X})$
the space of square-integrable vector fields on $\Omega$ with values in $\mathbf{X}$.  
In the paper, $\mathbf{X}$ will be $\mathbb{R}$, $\mathbb{R}^{2}$, or $\mathbb{M}$.
When $\mathbf{X}=\mathbb{R}$, we will write $L^{2}(\Omega)$. The corresponding
norms will be denoted with the same symbol $L^{2}(\Omega; \mathbb{M})$.
The corresponding Sobolev space of order 
$m$, denoted $H^{m}(\Omega;\mathbf{X})$,
is a subspace of $L^{2}(\Omega;\mathbf{X})$ consisting of functions with 
all partial derivatives of order less than or equal to $m$ in $L^{2}(\Omega;\mathbf{X})$.
The corresponding norm will be denoted by $\Vert\cdot\Vert_{H^{m}(\Omega)}$.
The space $H(\text{div},\Omega;\mathbb{M})$ 
is defined by
\begin{equation*}
H(\text{div},\Omega;\mathbb{M})=\{\sigma\in L^{2}(\Omega;\mathbb{M}):\text{div}
\sigma\in L^{2}(\Omega;\mathbb{R}^{2})\},
\end{equation*}
where divergence of a matrix field is the vector field obtained by applying operator 
$\text{div}$ row-wise, i.e., 
\begin{equation*}
\text{div}\sigma=(\dfrac{\partial\sigma_{11}}{\partial x_{1}}
+\dfrac{\partial\sigma_{12}}{\partial x_{2}},\quad \dfrac{\partial\sigma_{21}}{\partial x_{1}}
+\dfrac{\partial\sigma_{22}}{\partial x_{2}})^{\top}.
\end{equation*}
We introduce also a special space 
$$
H(\Omega)=\{\omega\in H(\text{div},\Omega):\omega
|_{\partial\Omega}\in L^{2}(\partial\Omega;\mathbb{R}^{2})\}
$$ 
with the norm,
$$
\Vert\omega\Vert_{H(\Omega)}=\Vert\omega\Vert_{H(\text{div},\Omega)}
+\Vert\omega\Vert_{L^{2}(\partial\Omega)},\forall\omega\in H(\Omega).
$$
For any scalar function $u$ and any vector function $\omega$ with values 
in $\mathbb{R}^{2}$, we denote 
$$
\text{curl }u = (\dfrac{\partial u}{\partial x_{2}},
-\dfrac{\partial u}{\partial x_{1}})^{\top},\quad
\text{curl }\omega =\left[\begin{array}{cc}\dfrac{\partial\omega_{1}}{\partial x_{2}}, & 
-\dfrac{\partial\omega_{1}}{\partial x_{1}}\\ \dfrac{\partial\omega_{2}}{\partial x_{2}}, & 
-\dfrac{\partial\omega_{2}}{\partial x_{1}}\end{array}\right].
$$
Finally, by $\Vert\cdot\Vert$ we denote the standard $2$-norm for vectors and matrices.

\section{Finite element spaces}

We begin by introducing the relevant finite element spaces on a single triangle $T$. 
Then, we define the finite element spaces on a whole affine triangular mesh $\mathcal{T}_{h}$ 
by ``gluing" the finite element spaces on triangles.

\subsection{Finite element spaces on a single triangle}

For any $r\in\mathbb{Z}_{+}:=\{n\in\mathbb{Z}:n\geq 0\}$, we introduce
\begin{align}
\label{FEM_spaces_uniform_order_triangle}
&\mathcal{P}_{r}(T):=\{\text{space of polynomials of order } r \text{ on }T\},\\
& \mathcal{P}_{r}\Lambda^{0}(T) = \mathcal{P}_{r}\Lambda^{2}(T)
:=\mathcal{P}_{r}(T),
\mathcal{P}_{r}\Lambda^{1}(T) := [\mathcal{P}_{r}(T)]^2, \nonumber \\
& \mathcal{\mathring{P}}_{r}\Lambda^{0}(T) = \mathcal{\mathring{P}}_{r}(T) :=\{w\in\mathcal{P}_{r}
\Lambda^{0}(T):w|_{e}=0, \forall e\in\triangle_{1}(T)\},\nonumber \\
& \mathcal{P}_{r}^{-}\Lambda^{1}(T):=[\mathcal{P}_{r-1}(T)]^{2}
+(x_{1},x_{2})^{\top}\mathcal{P}_{r-1}(T),\nonumber \\
& \mathcal{P}_{r}\Lambda^{0}(T;\mathbb{R}^{2})=\mathcal{P}_{r}\Lambda^{2}(T;\mathbb{R}^{2})
:=[\mathcal{P}_{r}(T)]^{2},\nonumber\\
& \mathcal{P}_{r}\Lambda^{1}(T;\mathbb{R}^{2}) := \left\{\left[ \begin{array}{cc}\sigma_{11} & \sigma_{12}\\
\sigma_{21} & \sigma_{22} \end{array}\right]:(\sigma_{11},\sigma_{12})^{\top}, (\sigma_{21},\sigma_{22})^{\top}
\in \mathcal{P}_{r}\Lambda^{1}(T)\right\}.\nonumber
\end{align}

In \cite{AFW:2006:ECH,Falk:2008:FME}, spaces in (\ref{FEM_spaces_uniform_order_triangle}) 
are defined in the language of exterior calculus. Here we just rewrite them 
using the language of calculus. 
We refer to \cite{AFW:2006:ECH} and \cite{Falk:2008:FME} for
a detailed correspondence before the exterior and classical calculus notations.

\bigskip
We denote by $\tilde{r}$ a mapping from $\triangle (T)$ 
to $\mathbb{Z}_{+}$ such that if $e,f\in\Delta(T)$ and $e\subset f$ then
$\tilde{r}(e)  \leq\tilde{r}(f)$.
We introduce now formally the FE spaces of variable order.

\begin{definition}
\label{FEM_spaces_single_triangle_variable_order}
\[
\mathcal{P}_{\tilde{r}}\Lambda^{0}(T):=
\{u\in\mathcal{P}_{\tilde{r}(T)}\Lambda^{0}(T):
\forall e\in\triangle_{1}(T) \text{, } u|_{e}
\in\mathcal{P}_{\tilde{r}(e)}(e)\},
\]
\[
\mathcal{P}_{\tilde{r}}\Lambda^{2}(T):=
\mathcal{P}_{\tilde{r}(T)}\Lambda^{2}(T)
=\mathcal{P}_{\tilde{r}(T)}(T),
\]
\[
\mathcal{P}_{\tilde{r}}\Lambda^{1}(T):=
\{\omega\in\mathcal{P}_{\tilde{r}(T)}\Lambda^{1}(T):
\forall e\in\triangle_{1}(T) \text{, } \omega\cdot \mathbf{n}|_{e}
\in\mathcal{P}_{\tilde{r}(e)}(e)\},
\]
\[
\mathcal{P}_{\tilde{r}}^{-}\Lambda^{1}(T):=
\{\omega\in\mathcal{P}_{\tilde{r}(T)}^{-}\Lambda^{1}(T):
\forall e\in\triangle_{1}(T) \text{, } \omega\cdot \mathbf{n}|_{e}
\in\mathcal{P}_{\tilde{r}(e)-1}(e)\},
\]
\[
\mathcal{P}_{\tilde{r}}\Lambda^{0}(T;\mathbb{R}^{2}):=
[\mathcal{P}_{\tilde{r}}\Lambda^{0}(T)]^{2},\quad
\mathcal{P}_{\tilde{r}}\Lambda^{2}
(T;\mathbb{R}^{2}):=[\mathcal{P}_{\tilde{r}}\Lambda^{2}(T)]^{2},
\]
\[
\mathcal{P}_{\tilde{r}}\Lambda^{1}(T;\mathbb{R}^{2}):=\left\{\left[ \begin{array}{cc}\sigma_{11} & \sigma_{12}\\
\sigma_{21} & \sigma_{22} \end{array}\right]:(\sigma_{11},\sigma_{12})^{\top}, (\sigma_{21},\sigma_{22})^{\top}\in
\mathcal{P}_{\tilde{r}}\Lambda^{1}(T)\right\}.
\]
Here $\mathbf{n}$ is the outward normal unit vector along $\partial T$.
\end{definition}

\begin{remark}
According to \cite{AFW:2006:ECH}, for any $e\in\triangle_{1}(T)$, we have 
\begin{equation*}
\mathcal{P}_{\tilde{r}}\Lambda^{0}(T)|_{e} =\mathcal{P}_{\tilde{r}(e)}(e), 
\mathcal{P}_{\tilde{r}}\Lambda^{1}(T)|_{e}\cdot\mathbf{n} =
\mathcal{P}_{\tilde{r}(e)}(e), 
\mathcal{P}_{\tilde{r}}^{-}\Lambda^{1}(T)|_{e}\cdot\mathbf{n} =
\mathcal{P}_{\tilde{r}(e)-1}(e).
\end{equation*}

Actually, spaces in Definition~\ref{FEM_spaces_single_triangle_variable_order} 
have been given in \cite{QD:2009:MMEW} with the language of exterior calculus.
\end{remark}



\subsection{Finite element spaces on an affine triangular mesh}

Let $\mathcal{T}_{h}$ be an affine triangular mesh. We extend the $\tilde{r}$ 
mapping to a global map defined on  $\triangle (\mathcal{T}_{h})$ with values in 
$\mathbb{Z}_{+}$ such that if $e\subset f$, then $\tilde{r}(e)\leq \tilde{r}(f)$. 

\begin{definition}
\label{FE_space_affine_triangulation}
We put $\Omega_{h}:=\bigcup_{T\in\mathcal{T}_{h}}T$.
\begin{align*}
C\Lambda^{0}(\mathcal{T}_{h}) & :=\{u\in H^{1}(\Omega_{h}): u 
\text{ is piece-wise smooth with respect to }\mathcal{T}_{h}\},\\
C\Lambda^{1}(\mathcal{T}_{h}) & :=\{\omega\in H(\text{div},\Omega_{h}): \omega 
\text{ is piece-wise smooth with respect to }\mathcal{T}_{h}\},\\
C\Lambda^{2}(\mathcal{T}_{h}) & :=\{u\in L^{2}(\Omega_{h}): u 
\text{ is piece-wise smooth with respect to }\mathcal{T}_{h}\}.
\end{align*}

We define
\begin{align*}
\mathcal{P}_{\tilde{r}}\Lambda^{0}(\mathcal{T}_{h}) & :=\{u\in C\Lambda^{0}(\mathcal{T}_{h}):
u|_{T}\in\mathcal{P}_{\tilde{r}}\Lambda^{0}(T), \forall T\in\mathcal{T}_{h}\},\\
\mathcal{P}_{\tilde{r}}\Lambda^{1}(\mathcal{T}_{h}) & :=\{\omega\in C\Lambda^{1}(\mathcal{T}_{h}):
\omega |_{T}\in\mathcal{P}_{\tilde{r}}\Lambda^{1}(T), \forall T\in\mathcal{T}_{h}\},\\
\mathcal{P}_{\tilde{r}}^{-}\Lambda^{1}(\mathcal{T}_{h}) & :=\{\omega\in C\Lambda^{1}(\mathcal{T}_{h}):
\omega |_{T}\in\mathcal{P}_{\tilde{r}}^{-}\Lambda^{1}(T), \forall T\in\mathcal{T}_{h}\},\\
\mathcal{P}_{\tilde{r}}\Lambda^{2}(\mathcal{T}_{h}) & :=\{u\in C\Lambda^{2}(\mathcal{T}_{h}):
u|_{T}\in\mathcal{P}_{\tilde{r}}\Lambda^{2}(T), \forall T\in\mathcal{T}_{h}\},\\
\mathcal{P}_{\tilde{r}}\Lambda^{0}(\mathcal{T}_{h};\mathbb{R}^{2}) & :=[\mathcal{P}_{\tilde{r}}
\Lambda^{0}(\mathcal{T}_{h})]^{2}, \mathcal{P}_{\tilde{r}}\Lambda^{2}(\mathcal{T}_{h};\mathbb{R}^{2})
:=[\mathcal{P}_{\tilde{r}}\Lambda^{2}(\mathcal{T}_{h})]^{2},\\
\mathcal{P}_{\tilde{r}}\Lambda^{1}(\mathcal{T}_{h};\mathbb{R}^{2}) & := \left\{\left[ \begin{array}{cc}
\sigma_{11} & \sigma_{12}\\\sigma_{21} & \sigma_{22} \end{array}\right]:(\sigma_{11},\sigma_{12})^{\top}, 
(\sigma_{21},\sigma_{22})^{\top}\in\mathcal{P}_{\tilde{r}}\Lambda^{1}(\mathcal{T}_{h},\mathbb{R}^{2})\right\}.
\end{align*}
\end{definition}

\begin{remark}
Spaces $\mathcal{P}_{\tilde{r}}\Lambda^{0}(\mathcal{T}_{h}),\mathcal{P}_{\tilde{r}}\Lambda^{1}(\mathcal{T}_{h}),
\mathcal{P}_{\tilde{r}}^{-}\Lambda^{1}(\mathcal{T}_{h}),\mathcal{P}_{\tilde{r}}\Lambda^{2}(\mathcal{T}_{h})$ 
coincide with those analyzed in \cite{QD:2009:MMEW}. 
\end{remark}

\begin{lemma}
\label{embedding_affine_triangulation}
\begin{align*}
& \mathcal{P}_{\tilde{r}}\Lambda^{1}(\mathcal{T}_{h})\subset\mathcal{P}_{\tilde{r}+1}^{-}\Lambda^{1}
(\mathcal{T}_{h})\subset\mathcal{P}_{\tilde{r}+1}\Lambda^{1}(\mathcal{T}_{h}), \\
& \text{div}\mathcal{P}_{\tilde{r}+1}\Lambda^{1}(\mathcal{T}_{h})\subset
\mathcal{P}_{\tilde{r}}\Lambda^{2}(\mathcal{T}_{h}),
\text{curl}\mathcal{P}_{\tilde{r}+1}\Lambda^{0}(\mathcal{T}_{h})\subset
\mathcal{P}_{\tilde{r}}\Lambda^{1}(\mathcal{T}_{h}).
\end{align*} 
\end{lemma}

\begin{proof}
This is straightforward.
\end{proof}

\section{Mixed formulation for elasticity with weakly imposed
symmetry}

We assume that there is $r_{\max}\in\mathbb{N}$ such that,
 for any $h>0$ and $f\in\Delta (\mathcal{T}_{h})  $,
$\tilde{r}(f)  \leq r_{\max}$.

We recall the mixed formulation~(\ref{weak_symm_formula_continuous}): Find $(\sigma,u,p)  \in
H(\text{div},\Omega;\mathbb{M}) \times L^{2}(\Omega;\mathbb{R}^{2})  
\times L^{2}(\Omega)$ such that
\begin{align}
\left\langle A\sigma,\tau\right\rangle +\left\langle \text{div}\tau,u\right\rangle
-\left\langle S_{1}\tau,p\right\rangle  &  =0,\text{ \ }\tau\in H(\text{div},\Omega;\mathbb{M})
,\label{weak_symmetry_differential_form}\\
\left\langle \text{div}\sigma,v\right\rangle  &  =\left\langle f,v\right\rangle ,\text{
\ }v\in L^{2}(\Omega;\mathbb{R}^{2})  ,\nonumber\\
\left\langle S_{1}\sigma,q\right\rangle  &  =0,\text{ \ }q\in L^{2}(\Omega)  .\nonumber
\end{align}
Here $\left\langle \cdot,\cdot\right\rangle$ is the standard $L^{2}$ inner product on $\Omega$.
This problem is well-posed.
See \cite{AFW:2006:ECH} and \cite{Falk:2008:FME} for the proof.

We consider now a finite element discretization of
(\ref{weak_symmetry_differential_form}). For this, we choose families of
finite-dimensional subspaces
\[
\Lambda_{h}^{1}(\mathbb{M})  \subset H(\text{div},\Omega;\mathbb{M}), 
\Lambda_{h}^{2}(\mathbb{R}^{2})  \subset
L^{2}(\Omega;\mathbb{R}^{2})  ,\Lambda_{h}^{2}
\subset L^{2}(\Omega),
\]
indexed by $h$, and seek a discrete solution $(\sigma_{h},u_{h}
,p_{h})  \in\Lambda_{h}^{1}(\mathbb{M})  \times
\Lambda_{h}^{2}(\mathbb{R}^{2})  \times\Lambda_{h}^{2}$ such that
\begin{align}
\left\langle A\sigma_{h},\tau\right\rangle +\left\langle \text{div}\tau,u_{h}
\right\rangle -\left\langle S_{1}\tau,p_{h}\right\rangle  &  =0,\text{ \ }\tau
\in\Lambda_{h}^{1}(\mathbb{M})
,\label{weak_symmetry_differential_form_discrete}\\
\left\langle \text{div}\sigma_{h},v\right\rangle  &  =\left\langle f,v\right\rangle
,\text{ \ }v\in\Lambda_{h}^{2}(\mathbb{R}^{2})  ,\nonumber\\
\left\langle S_{1}\sigma_{h},q\right\rangle  &  =0,\text{ \ }q\in\Lambda_{h}^{2}  .\nonumber
\end{align}
The stability of (\ref{weak_symmetry_differential_form_discrete}) will be
ensured by the Brezzi stability conditions:
\begin{gather}
\text{(S1) }\left\Vert \tau\right\Vert _{H(\text{div},\Omega;\mathbb{M})}^{2}\leq c_{1}
(A\tau,\tau)  \text{ whenever }\tau\in\Lambda_{h}^{1}
(\mathbb{R}^{2})  \text{ satisfies }\left\langle \text{div}\tau,v\right\rangle
=0\label{S1_condition}\\
\forall v\in\Lambda_{h}^{2}(\mathbb{R}^{2})  \text{ and }\left\langle
S_{1}\tau,q\right\rangle =0\text{ }\forall q\in\Lambda_{h}^{2}  ,\nonumber
\end{gather}
\begin{gather}
\text{(S2) for all nonzero }\left(  v,q\right)  \in\Lambda_{h}^{2}
(\mathbb{R}^{2})  \times\Lambda_{h}^{2}  \text{,
there exists nonzero}\label{S2_condition}\\
\tau\in\Lambda_{h}^{1}(\mathbb{R}^{2})  \text{ with }\left\langle
\text{div}\tau,v\right\rangle -\left\langle S_{1}\tau,q\right\rangle \geq 
c_{2}\left\Vert\tau\right\Vert_{H(\text{div},\Omega;\mathbb{M})}
(\Vert v\Vert_{L^{2}(\Omega;\mathbb{R}^{2})} +\Vert q\Vert_{L^{2}(\Omega)} )  ,\nonumber
\end{gather}
where constants $c_{1}$ and $c_{2}$ are independent of $h$.

For meshes of arbitrary but uniform order, conditions (\ref{S1_condition})
and (\ref{S2_condition}) have been proved in \cite{AFW:2006:ECH} and
\cite{Falk:2008:FME}. In what follows, we will demonstrate that they are also
satisfied for (2D) meshes with elements of variable (but limited) order.

Before presenting our proof, we would like to comment shortly on 
difficulties encountered in proving stability for generalizing 
AFW elements with variable order. 
As we have shown in Section 3, it is rather straightforward to generalize 
AFW elements to the variable order case.
The following 
commuting diagrams are essential in the stability proof from 
\cite{AFW:2007:MMEW,AFW:2006:ECH,Falk:2008:FME},

{\scriptsize
\begin{align}
\label{three_diagrams}
\begin{tabular}
[c]{lllllllll}
$H^{1}(\Omega;\mathbb{R}^{2})$ & $\overset{\text{div}}
{\longrightarrow}$ & $L^{2}(\Omega)$ & $H^{1}(\Omega;\mathbb{R}^{2})$
& $\overset{\text{div}}{\longrightarrow}$ & $L^{2}(\Omega)$ 
& $H^{1}(\Omega;\mathbb{R}^{2})$ & $\overset{Id}{\longrightarrow}$ 
& $H^{1}(\Omega;\mathbb{R}^{2})$\\
 & & & & & & & &\\
$\Pi_{h}^{1,-}\downarrow$ & & $\Pi_{h}^{2}\downarrow$ & 
$\Pi_{h}^{1}\downarrow$ & & $\Pi_{h}^{2}\downarrow$ & 
$\Pi_{h}^{0}\downarrow$ & & $\Pi_{h}^{1,-}\downarrow$\\
 & & & & & & & &\\
 $\mathcal{P}_{r+1}^{-}\Lambda^{1}(\mathcal{T}_{h})$ & 
$\overset{\text{div}}{\longrightarrow}$ & 
$\mathcal{P}_{r}\Lambda^{2}(\mathcal{T}_{h})$ & 
$\mathcal{P}_{r+1}\Lambda^{1}(\mathcal{T}_{h})$ & 
$\overset{\text{div}}{\longrightarrow}$ & 
$\mathcal{P}_{r}\Lambda^{2}(\mathcal{T}_{h})$ &
$\mathcal{P}_{r+2}\Lambda^{0}(\mathcal{T}_{h};\mathbb{R}^{2})$ & 
$\overset{\Pi_{h}^{1,-}\circ Id}{\longrightarrow}$ &
$\mathcal{P}_{r+1}^{-}\Lambda^{1}(\mathcal{T}_{h})$
\end{tabular}
\end{align}
}

When uniform order $r$ is replaced by variable order $\tilde{r}$, then the left and middle 
diagrams do not commute if $\Pi_{h}^{1,-}$ and $\Pi_{h}^{1}$ are natural generalizations of 
canonical projection operators introduced in \cite{AFW:2007:MMEW,AFW:2006:ECH,Falk:2008:FME}. 
A counterexample is given in the appendix of \cite{QD:2009:MMEW}. 
To our best knowledge, the only operators 
which make these two diagrams commute for meshes with variable order,
are projection based interpolation 
operators introduced in \cite{QD:2009:MMEW}. 
But then we need to construct an operator $\Pi_{h}^{0}$ that makes the right diagram commute 
if $\Pi_{h}^{1,-}$ is the projection based interpolation operator. 
We replaced the operator $\Pi_{h}^{0}$ with a new operator
$W_{h}$, constructed in \cite{QD:2009:MMEW}, but we were able to prove its well-definedness 
only for $0\leq \tilde{r}(T) \leq 3$. 
In the following sections we pursue a different strategy, defining {\em new}
projection 
based interpolation $\Pi_{h}^{1,-}, \Pi_{h}^{1}$ operators, and a new operator $W_{h}$ 
in such a way 
that we can make all three 
diagrams 
commute.


\section{Preliminaries for the proof of stability\label{preliminary_affine}}

We begin by recalling our assumptions on the domain and meshes: $\Omega$ is a
polygon and it is meshed with a family $(\mathcal{T}_{h})_{h}$ of 
affine triangular meshes satisfying assumption of regularity. For any mesh $\mathcal{T}_{h}$, 
mapping $\tilde{r} \: : \: \triangle (\mathcal{T}_{h}) \rightarrow \mathbb{Z}_{+}$ 
defines a locally variable order of discretization that satisfies the minimum rule,
$\tilde{r}(e) \leq\tilde{r}(f)$ for $e\subset f,\: e,f\in\triangle_{T}$. 
The maximum order is limited, i.e. $\sup_{h}\sup_{T\in\mathcal{T}_{h}}\tilde{r}(T)<\infty$. 

\begin{definition}
For any $T\in\mathcal{T}_{h}$, we define a linear operator 
$\Pi_{\tilde{r},T}^{2}$ as the standard orthogonal projection operator 
from $L^{2}(T)$ onto $\mathcal{P}_{\tilde{r}}\Lambda^{2}(T)$.
\end{definition}

\subsection{Projection Based Interpolation onto $\mathcal{P}_{\tilde{r}+1}\Lambda^{1}(T)$}

\begin{definition}
For any $T\in\mathcal{T}_{h}$, we define a linear operator 
$\Pi_{\tilde{r}+1,T}^{1}:H^{1}(T;\mathbb{R}^{2})
\longrightarrow \mathcal{P}_{\tilde{r}+1}\Lambda^{1}(T)$
by the relations
\begin{equation}
\int_{T}\text{div}(\Pi_{\tilde{r}+1,T}^{1}\omega-\omega)(\mathbf{x})
\psi (\mathbf{x}) 
d\mathbf{x}=0\quad
\forall \psi
\in\mathcal{P}_{\tilde{r}(T)}(T)/\mathbb{R}
\label{interior_div}
\end{equation}
\begin{equation}
\int_{T}(\Pi_{\tilde{r}+1,T}^{1}\omega(\mathbf{x})-\omega(\mathbf{x}))^{\top}
\text{curl}
\varphi (\mathbf{x}) d\mathbf{x}=0 \quad
\forall \varphi
\in\mathcal{\mathring{P}}_{\tilde{r}(T)+2}(T) 
\label{interior_aux}
\end{equation}
\begin{equation}
\int_{e}[(\Pi_{\tilde{r}+1,T}^{1}\omega-\omega)
\cdot\mathbf{n}]\eta(s) ds = 0 \quad
\forall \eta\in\mathcal{P}_{\tilde{r}(e)+1}(e)\: , \forall e\in\triangle_{1}(T)
\label{H_div_edge}
\end{equation}
Here $\mathbf{n}$ is a unit outward normal vector along $e$.
\end{definition}

The operator $\Pi_{\tilde{r}+1,T}^{1}$ is the Projection Based Interpolation operator onto 
$\mathcal{P}_{\tilde{r}+1}\Lambda^{1}(T)$ defined in \cite{QD:2009:MMEW}. We immediately 
have the following two lemmas.

\begin{lemma}
\label{commuting_diagram1}
For any $T\in\mathcal{T}_{h}$, and any $\omega\in [H^{1}(T)]^2$, we have
\begin{equation*}
\Pi_{\tilde{r},T}^{2}\text{div}\omega = \text{div}\Pi_{\tilde{r}+1,T}^{1}\omega.
\end{equation*}
\end{lemma}

\begin{lemma}
\label{H_div_inequality1}
There exists $C>0$ such that
\begin{equation*}
\Vert \Pi_{\tilde{r},T}^{1}\omega\Vert_{L^{2}(T)}\leq C\Vert \omega\Vert_{H^{1}(T)} \quad
\forall T\in\mathcal{T}_{h},\omega\in H^{1}(T;\mathbb{R}^{2}).
\end{equation*}
\end{lemma}

\subsection{Modified Projection Based Interpolation onto 
$\mathcal{P}_{\tilde{r}+1}^{-}\Lambda^{1}(T)$
and modified operator $W$ onto $\mathcal{P}_{\tilde{r}+2}\Lambda^{0}(T)$}

In order to prove the stability of the mixed FE method, we need to make the left and right 
diagrams in (\ref{three_diagrams}) commute. In \cite{QD:2009:MMEW}, 
the definition of Projection Based (PB) interpolation 
operator onto $\mathcal{P}_{\tilde{r}+1}^{-}\Lambda^{1}(\hat{T})$ was very similar to 
$\mathcal{P}_{\tilde{r}+1}\Lambda^{1}(\hat{T})$.
From the proof of Lemma~10 in \cite{QD:2009:MMEW}, 
we can see that only conditions~(\ref{interior_div})
and~(\ref{H_div_edge}) are used to prove the commutativity 
of the middle diagram in (\ref{three_diagrams}).
This implies that we may be able to change condition~(\ref{interior_aux}) 
for the PB interpolation 
operator onto $\mathcal{P}_{\tilde{r}+1}^{-}\Lambda^{1}(\hat{T})$ 
in such a way that we can
design a new operator $W_{h}$ which makes the right diagram in~(\ref{three_diagrams}) commute,
and can be proved to be well-defined, for an arbitrary order of discretization. 

\begin{definition}
Let $\tilde{r}\: : \: \triangle(\hat{T}) \rightarrow \mathbb{Z}_{+}$ 
be a mapping that prescribes the local order of discretization and satisfies
the minimum rule, i.e. if 
$\hat{e},\hat{f}\in\triangle(\hat{T})$ and $\hat{e}\subset\hat{f}$ then $\tilde{r}(\hat{e})
\leq\tilde{r}(\hat{f})$. We put $k_{\tilde{r}}=\dim\text{curl}_{\hat{\mathbf{x}}}
\mathcal{\mathring{P}}_{\tilde{r}(\hat{T})+1}(\hat{T})$.
Let $\{\hat{\mathbf{f}}_{\tilde{r},1}(\hat{\mathbf{x}}),\cdots, \hat{\mathbf{f}}_{\tilde{r},k_{\tilde{r}}}
(\hat{\mathbf{x}})\}$ be a basis of  $\text{curl}_{\hat{\mathbf{x}}}
\mathcal{\mathring{P}}_{\tilde{r}(\hat{T})+1}(\hat{T})$.
Let $\{\hat{\mathbf{g}}_{\tilde{r},1}(\hat{\mathbf{x}}),\cdots, \hat{\mathbf{g}}_{\tilde{r},k_{\tilde{r}}}
(\hat{\mathbf{x}})\}$ be a linearly independent subset of $\mathcal{P}_{\tilde{r}(\hat{T})-1}(\hat{T};
\mathbb{R}^{2})$ such that 
$$
\mbox{span} \{\hat{\mathbf{g}}_{\tilde{r},1}(\hat{\mathbf{x}}),\cdots, \hat{\mathbf{g}}_{\tilde{r},
k_{\tilde{r}}}(\hat{\mathbf{x}})\}\oplus \text{grad}_{\hat{\mathbf{x}}}\mathcal{P}_{\tilde{r}(\hat{T})}
(\hat{T})=[\mathcal{P}_{\tilde{r}(\hat{T})-1}(\hat{T})]^{2}
$$
For $t \in [0,1]$, we define $\hat{\mathbf{h}}_{\tilde{r},i}(\hat{\mathbf{x}},t)=(1-t)\hat{\mathbf{f}}_{\tilde{r},1}(\hat{\mathbf{x}})
+t\hat{\mathbf{g}}_{\tilde{r},1}(\hat{\mathbf{x}}), \: 1\leq i\leq k_{\tilde{r}}$.
\end{definition}

\begin{remark}
It is easy to check that $k_{\tilde{r}}=\dim\mathcal{P}_{\tilde{r}(\hat{T})-1}(\hat{T};\mathbb{R}^{2})-
\dim\text{grad}_{\hat{\mathbf{x}}}\mathcal{P}_{\tilde{r}(\hat{T})}(\hat{T})$.
\end{remark}

\begin{remark}
\label{jacobian_affine}
For any $T\in\mathcal{T}_{h}$, there exists an affine mapping from the reference triangle $\hat{T}$ 
onto $T$, which can be written as $\mathbf{x}=B_{T}\hat{\mathbf{x}}+\mathbf{b}$. Here $B_{T}$ is a constant 
nonsingular matrix in $\mathbb{R}^{2\times 2}$. We denote by $\tilde{r}$ a mapping from $\triangle(T)$
to $\mathbb{Z}_{+}$ such that if $e,f\in\triangle (T)$ and $e\subset f$ then $\tilde{r}(e)\leq \tilde{r}(f)$. 
With the same symbol $\tilde{r}$ we denote the corresponding mapping from $\triangle (\hat{T})$ to $\mathbb{Z}_{+}$, 
$\tilde{r}(\hat{f}):=\tilde{r}(f)$ for any $\hat{f}\in\triangle (\hat{T})$, where $f$ is the image of $\hat{f}$ 
under the affine mapping mentioned above.
\end{remark}

\begin{definition}
\label{PB_minus_reference}(One-parameter family of PB interpolation operators onto 
$\mathcal{P}_{\tilde{r}+1}^{-}\Lambda^{1}(\hat{T})$)
For any $t \in [0,1]$, we define a linear operator $\Pi_{\tilde{r}+1,\hat{T},t}^{1,-}
:H^{1}(\hat{T};\mathbb{R}^{2})\longrightarrow \mathcal{P}_{\tilde{r}+1}^{-}\Lambda^{1}(\hat{T})$
by the relations
\begin{equation}
\int_{\hat{T}}\text{div}_{\hat{\mathbf{x}}}
(\Pi_{\tilde{r}+1,\hat{T},t}^{1,-}\hat{\omega}-\hat{\omega})
(\hat{\mathbf{x}})\hat{\psi}(\hat{\mathbf{x}}) d\hat{\mathbf{x}}=0 \quad 
\forall \hat{\psi}
\in\mathcal{P}_{\tilde{r}(\hat{T})}(\hat{T})/\mathbb{R}
\label{interior_div_minus_reference}
\end{equation}
\begin{equation}
\int_{\hat{T}}(\Pi_{\tilde{r}+1,\hat{T},t}^{1,-}\hat{\omega}(\hat{\mathbf{x}})
-\hat{\omega}(\hat{\mathbf{x}}))^{\top}
\hat{\mathbf{h}}_{i}(\hat{\mathbf{x}},t) d\hat{\mathbf{x}}=0 \quad
1\leq i\leq k_{\tilde{r}}
\label{interior_aux_minus_reference}
\end{equation}
\begin{equation}
\int_{\hat{e}}[(\Pi_{\tilde{r}+1,\hat{T},t}^{1,-}\hat{\omega}-\hat{\omega})
\cdot\hat{\mathbf{n}}]\hat{\eta} d\hat{s} = 0 \quad 
\forall \hat{\eta}\in\mathcal{P}_{\tilde{r}(\hat{e})}(\hat{e}),
\forall \hat{e}\in\triangle_{1}(\hat{T})
\label{H_div_edge_minus_reference}
\end{equation}
\end{definition}

\begin{definition}
\label{PB_minus}(One-parameter family of PB interpolation operators onto 
$\mathcal{P}_{\tilde{r}+1}^{-}\Lambda^{1}(T)$)
For $t \in [0,1]$, and 
for any $T\in\mathcal{T}_{h}$, we define a linear operator $\Pi_{\tilde{r}+1,T,t}^{1,-}
:H^{1}(T;\mathbb{R}^{2})\longrightarrow \mathcal{P}_{\tilde{r}+1}^{-}\Lambda^{1}(T)$
by the relations
\begin{equation}
\int_{T}\text{div}(\Pi_{\tilde{r}+1,T,t}^{1,-}\omega-\omega)(\mathbf{x})
\hat{\psi}(\hat{\mathbf{x}}(\mathbf{x})) \: d\mathbf{x}=0 \quad 
\forall \hat{\psi}
\in\mathcal{P}_{\tilde{r}(T)}(\hat{T})/\mathbb{R}
\label{interior_div_minus}
\end{equation}
\begin{equation}
\int_{T}(\Pi_{\tilde{r}+1,T,t}^{1,-}\omega(\mathbf{x})-\omega(\mathbf{x}))^{\top}
B_{T}^{-\top}\hat{\mathbf{h}}_{i}(\hat{\mathbf{x}}(\mathbf{x}),t) d\mathbf{x}=0 \quad
1\leq i\leq k_{\tilde{r}}
\label{interior_aux_minus}
\end{equation}
\begin{equation}
\int_{e}[(\Pi_{\tilde{r}+1,T,t}^{1,-}\omega-\omega)
\cdot\mathbf{n}]\eta (s) ds = 0 \quad
\forall \eta\in\mathcal{P}_{\tilde{r}(e)}(e),
\forall e\in\triangle_{1}(T)
\label{H_div_edge_minus}
\end{equation}
In the above, $\hat{\mathbf{x}} = \hat{\mathbf{x}}(\mathbf{x})$ denotes the
inverse of the affine mapping introduced in Remark \ref{jacobian_affine}. 
The matrix $B_{T}$ is defined in Remark \ref{jacobian_affine} as well. 
And according 
to Remark~\ref{jacobian_affine}, $\tilde{r}(f)=\tilde{r}(\hat{f})$ for any $\hat{f}\in
\triangle (\hat{T})$, where $f$ is the image of $\hat{f}$ under the affine mapping 
from $\hat{T}$ to $T$. 
\end{definition}

\begin{definition}
In the sequel, the phrase ``for almost all'' (parameters) will mean 
``for all except for a finite number'' (of parameters). For example, a sequence
$x_n$ in a topological space converges to $x$ if, for every neighborhood of $x$,
almost all values of $x_n$ belong to the neighborhood.
\end{definition}
\begin{lemma}
\label{PB_modifiled_well_defined}
$\Pi_{\tilde{r}+1,\hat{T},t}^{1,-}:H^{1}(\hat{T};\mathbb{R}^{2})\longrightarrow \mathcal{P}_{\tilde{r}+1}^{-}
\Lambda^{1}(\hat{T})$ is a well-defined linear operator for almost all $t\in [0,1]$.
Moreover,
\begin{equation*}
\text{div}_{\hat{\mathbf{x}}}\Pi_{\tilde{r}+1,\hat{T},t}^{1,-}\hat{\omega}
=\Pi_{\tilde{r},\hat{T}}^{2}\text{div}_{\hat{\mathbf{x}}}\hat{\omega}
\label{com_div_minus_reference}
\end{equation*}
for any $\hat{\omega}\in H^{1}(\hat{T};\mathbb{R}^{2})$.
\end{lemma}

\begin{proof}
The linearity of the operator is obvious, for any $t\in\mathbb{R}$. 
The point is to show that the operator is well-defined.
For $t=0$, the operator $\Pi_{\tilde{r}+1,\hat{T},0}^{1,-}$ reduces to the PB interpolation 
defined in \cite{QD:2009:MMEW} and, 
according to Lemma~$9$ and Lemma~$20$ from \cite{QD:2009:MMEW}, is well-defined.
Moreover,
\begin{itemize}
  \item $\int_{\hat{T}}\text{div}_{\hat{\mathbf{x}}}\hat{\omega}
(\hat{\mathbf{x}})\hat{\psi}(\hat{\mathbf{x}}) d\hat{x}_{1}d\hat{x}_{2}$ 
is a continuous functional 
of $\hat{\omega}$, for any 
$\hat{\psi}
\in\mathcal{P}_{\tilde{r}(\hat{T})}(\hat{T})/\mathbb{R}$,
\item
$\int_{\hat{T}}\hat{\omega}(\hat{\mathbf{x}})^{\top}
\hat{\mathbf{h}}_{i}(\hat{\mathbf{x}},t) d\hat{x}_{1}d\hat{x}_{2}$ 
is a continuous functional of $\hat{\omega}$, 
for any $1\leq i\leq k_{\tilde{r}}$ and any $t\in\mathbb{R}$, and 
\item
$\int_{\hat{e}}[\hat{\omega}\cdot\hat{\mathbf{n}}]\hat{\eta} d\hat{s}$ 
is a continuous functional of $\hat{\omega}$, 
for any $\hat{\eta}\in\mathcal{P}_{\tilde{r}(\hat{e})}(\hat{e})$,
and any $\hat{e}\in\triangle_{1}(\hat{T})$.
\end{itemize}
Therefore, in order to demonstrate that $\Pi_{\tilde{r}+1,\hat{T},t}^{1,-}$ 
is well-defined, 
it is sufficient to show that $\hat{\omega}=0$ 
if $\hat{\omega}\in \mathcal{P}_{\tilde{r}+1}^{-}\Lambda^{1}(\hat{T})$ 
and $\Pi_{\tilde{r}+1,\hat{T},t}^{1,-}\hat{\omega} = 0$.

We take an arbitrary $\hat{\omega}\in\mathcal{P}_{\tilde{r}+1}^{-}\Lambda^{1}(\hat{T})$ 
such that 
$\Pi_{\tilde{r}+1,\hat{T},t}^{1,-}\hat{\omega} = 0$. 
According to the definition of $\mathcal{P}_{\tilde{r}+1}^{-}\Lambda^{1}(\hat{T})$ 
and (\ref{H_div_edge_minus_reference}), 
we know that 
$\hat{\omega}\in\mathring{\mathcal{P}}_{\tilde{r}+1}^{-}\Lambda^{1}(\hat{T})$.
Set $r=\tilde{r}(\hat{T})$.
We denote by $C(t,r)$ the matrix associated with conditions (\ref{interior_div_minus_reference}) 
and (\ref{interior_aux_minus_reference}), corresponding to a particular basis for
 $\mathring{\mathcal{P}}_{\tilde{r}+1}^{-}\Lambda^{1}(\hat{T})$ 
(the solution space), and a particular basis for $\mathcal{P}_{\tilde{r}(\hat{T})}(\hat{T})/\mathbb{R}$. 
We argue that
matrix $C(t,r)$ is non-singular for almost all $t\in [0,1]$.
Notice that, for any $r$, $\det(C(t,r))$ is a polynomial in $t$. 
Since $\Pi_{\tilde{r}+1,\hat{T},0}^{1,-}$ is well-defined, $\det(C(0,r))\neq 0$.
So $\det(C(t,r))$ is a non-zero polynomial.
By the fundamental theorem of algebra, the polynomial has a finite number of real roots.
We can conclude thus that $\Pi_{\tilde{r}+1,\hat{T},t}^{1,-}:H^{1}(\hat{T};\mathbb{R}^{2})\longrightarrow 
\mathcal{P}_{\tilde{r}+1}^{-}\Lambda^{1}(\hat{T})$ 
is well-defined for any $t\in [0,1]$ 
except for the roots of $\det(C(t,r))$. The number of roots is independent of the
choice of the bases and depends 
only upon $r=\tilde{r}(\hat{T})$. 

Since $\text{div}_{\hat{\mathbf{x}}}\Pi_{\tilde{r}+1,\hat{T},t}^{1,-}\hat{\omega}\in\mathcal{P}_{\tilde{r}}
\Lambda^{2}(\hat{T})$,
 for any $\hat{\omega}\in H^{1}(\hat{T};\mathbb{R}^{2})$,\: 
then $\Pi_{\tilde{r},\hat{T}}^{2}\text{div}_{\hat{\mathbf{x}}}\Pi_{\tilde{r}+1,\hat{T},t}^{1,-}
\hat{\omega} =\text{div}_{\hat{\mathbf{x}}}\Pi_{\tilde{r}+1,\hat{T},t}^{1,-}
\hat{\omega}$.
In order to show that $\text{div}_{\hat{\mathbf{x}}}\Pi_{\tilde{r}+1,\hat{T},t}^{1,-}\hat{\omega}
=\Pi_{\tilde{r},\hat{T}}^{2}\text{div}_{\hat{\mathbf{x}}}\hat{\omega}$, it is sufficient to show that 
$\Pi_{\tilde{r},\hat{T}}^{2}\text{div}_{\hat{\mathbf{x}}}\hat{\omega}=0$,
 for any $\hat{\omega}$ 
such that $\Pi_{\tilde{r}+1,\hat{T},t}^{1,-}\hat{\omega}=0$. We have
\begin{align*}
\int_{\hat{T}}\text{div}_{\hat{\mathbf{x}}}\hat{\omega}\hat{\psi}d\hat{\mathbf{x}} = 
\int_{\hat{T}}\text{div}_{\hat{\mathbf{x}}}\hat{\omega}(\hat{\chi}+c)d\hat{\mathbf{x}}
= \int_{\hat{T}}\text{div}_{\hat{\mathbf{x}}}\hat{\omega}\hat{\chi}d\hat{\mathbf{x}}
+c\int_{\partial\hat{T}}\hat{\omega}\cdot\hat{\mathbf{n}}d\hat{s}=0.
\end{align*}
Here $\Pi_{\tilde{r}+1,\hat{T},t}^{1,-}\hat{\omega}=0$, $\hat{\psi}\in\mathcal{P}_{\tilde{r}(\hat{T})}(\hat{T})$, 
$c=\int_{\hat{T}}\hat{\psi}d\hat{\mathbf{x}}$. The last equality holds due to the definition of $\Pi_{\tilde{r}+1,\hat{T},t}^{1,-}$ 
and the fact that $\Pi_{\tilde{r}+1,\hat{T},t}^{1,-}\hat{\omega}=0$. 
We have thus $\text{div}_{\hat{\mathbf{x}}}\Pi_{\tilde{r}+1,\hat{T},t}^{1,-}\hat{\omega}=\Pi_{\tilde{r},
\hat{T}}^{2}\text{div}_{\hat{\mathbf{x}}}\hat{\omega}$ for any $\hat{\omega}\in H^{1}(\hat{T};\mathbb{R}^{2})$. 
\end{proof}

\begin{theorem}
\label{H_div_transform_minus}
Let $t\in [0,1]$ be any value 
for which  $\Pi_{\tilde{r}+1,\hat{T},t}^{1,-}$ is well-defined.
Then the operator 
$\Pi_{\tilde{r}+1,T,t}^{1,-}:H^{1}(T;\mathbb{R}^{2})
\longrightarrow \mathcal{P}_{\tilde{r}+1}^{-}\Lambda^{1}(T)$ 
is well-defined as well, and we have the following result.

There exist $C>0$ such that,for $T\in\mathcal{T}_{h}$, and 
$\omega\in H^{1}(T;\mathbb{R}^{2})$, the
corresponding function $\hat{\omega}$ defined by 
\begin{equation}
\omega(\mathbf{x}(\hat{\mathbf{x}}))=\dfrac{B_{T}}{\det(B_{T})}
\hat{\omega}(\hat{\mathbf{x}})
\label{trans1}
\end{equation}
belongs to $H^{1}(\hat{T};\mathbb{R}^{2})$, and 
\begin{equation}
\Pi_{\tilde{r}+1,T,t}^{1,-}\omega(\mathbf{x}(\hat{\mathbf{x}})) = \dfrac{B_{T}}
{\det(B_{T})}\Pi_{\tilde{r}+1,\hat{T},t}^{1,-}\hat{\omega}(\hat{\mathbf{x}})
\label{trans2}
\end{equation}
Additionally,
\begin{equation}
\text{div}\Pi_{\tilde{r}+1,T,t}^{1,-}\omega=\Pi_{\tilde{r},T}^{2}\text{div}\omega
\label{com_div_minus}
\end{equation}
and
\begin{equation}
\Vert\Pi_{\tilde{r}+1,T,t}^{1,-}\omega\Vert_{L^{2}(T)}\leq C\Vert\omega\Vert_{H^{1}(T)}
\label{bound_pi_one_minus}
\end{equation}
\end{theorem}

\begin{proof}
It is easy to see that $\hat{\omega}\in H^{1}(\hat{T};\mathbb{R}^{2})$.
We define a linear isomorphism $A_{T}$ from $H^{1}(\hat{T};\mathbb{R}^{2})$ to $H^{1}(T;\mathbb{R}^{2})$ 
by $(A_{T}\hat{\omega})(\mathbf{x}(\hat{\mathbf{x}}))=\dfrac{B_{T}}{\det(B_{T})}
\hat{\omega}(\hat{\mathbf{x}})$. It is easy to see that $A_{T}$ is a linear isomorphism from 
$\mathcal{P}_{\tilde{r}+1}^{-}\Lambda^{1}(\hat{T})$ to $\mathcal{P}_{\tilde{r}+1}^{-}\Lambda^{1}(T)$.

By pulling back to $\hat{T}$ by $A_{T}$ (applied to both $\omega$ and $\Pi_{\tilde{r}+1,T,t}^{1,-}\omega)$, 
we can see that (\ref{interior_div_minus}) is the same as (\ref{interior_div_minus_reference}), 
(\ref{interior_aux_minus}) is the same as (\ref{interior_aux_minus_reference}), and 
(\ref{H_div_edge_minus}) is the same as (\ref{H_div_edge_minus_reference}). So we have (\ref{trans2}). 
As a consequence of (\ref{com_div_minus_reference}), (\ref{trans1}) and Lemma~\ref{trans2}, we have (\ref{com_div_minus}).
Inequality (\ref{bound_pi_one_minus}) is proved by using standard scaling techniques.
\end{proof}

\begin{definition}
\label{C_W_operator_reference}
We define a linear operator $C_{t}:[H^{1}(\hat{T})]^2\longrightarrow 
\mathcal{P}_{\tilde{r}+2}\Lambda^{0}(\hat{T};\mathbb{R}^{2})$
by the following relations
\begin{equation}
\int_{\hat{T}}\text{div}_{\hat{\mathbf{x}}}C_{t}\hat{\omega}
(\hat{\mathbf{x}})\hat{\psi}(\hat{\mathbf{x}}) d\hat{\mathbf{x}}=
\int_{\hat{T}}\text{div}_{\hat{\mathbf{x}}}\hat{\omega}
(\hat{\mathbf{x}})\hat{\psi}(\hat{\mathbf{x}}) d\hat{\mathbf{x}}\quad 
\forall \hat{\psi}
\in\mathcal{P}_{\tilde{r}(\hat{T})}(\hat{T})/\mathbb{R}
\label{C_interior_div}
\end{equation}
\begin{equation}
\int_{\hat{T}}(C_{t}\hat{\omega}(\hat{\mathbf{x}}))^{\top}
\hat{\mathbf{h}}_{i}(\hat{\mathbf{x}},t) d\hat{\mathbf{x}}=
\int_{\hat{T}}(\hat{\omega}(\hat{\mathbf{x}}))^{\top}
\hat{\mathbf{h}}_{i}(\hat{\mathbf{x}},t) d\hat{\mathbf{x}}\quad
1\leq i\leq k_{\tilde{r}}
\label{C_interior_aux}
\end{equation}
\begin{align}
\label{C_edge_normal}
\int_{\hat{e}}[(C_{t}\hat{\omega})\cdot 
\hat{\mathbf{n}}]\hat{\eta}d\hat{s} = 
\int_{\hat{e}}[\hat{\omega}\cdot\hat{\mathbf{n}}]\hat{\eta}d\hat{s}\quad
\forall \hat{\eta}\in\mathcal{P}_{\tilde{r}(\hat{e})}(\hat{e}), \forall \hat{e}\in\triangle_{1}(\hat{T})
\end{align}
\begin{align}
\label{C_edge_tangent}
 \int_{\hat{e}}[(C_{t}\hat{\omega})\cdot\hat{\mathbf{t}}]\hat{\eta}d\hat{s}= 
\int_{\hat{e}}[\hat{\omega}\cdot\hat{\mathbf{t}}]\hat{\eta}d\hat{s}\quad
\forall \hat{\eta}\in\mathcal{P}_{\tilde{r}(\hat{e})}(\hat{e}), \forall \hat{e}\in\triangle_{1}(\hat{T})
\end{align}
\begin{equation}
C_{t}\hat{\omega} = 0 \quad \text{ at all vertices of }\hat{T}
\label{C_vertex}
\end{equation}
Here $\hat{\mathbf{n}},\hat{\mathbf{t}}$ denote the normal and tangent unit vectors along $\partial\hat{T}$.
\end{definition}

\begin{definition}
For any $T\in\mathcal{T}_{h}$, we define a linear operator $W_{T,t}
:H^{1}(T;\mathbb{R}^{2})\longrightarrow \mathcal{P}_{\tilde{r}+2}\Lambda^{0}(T;\mathbb{R}^{2})$
by the following relations
\begin{equation}
\int_{T}\text{div}(W_{T,t}\omega-\omega)(\mathbf{x})\hat{\psi}(\hat{\mathbf{x}}(\mathbf{x}))d\mathbf{x}=0
\quad 
\forall \hat{\psi}
\in\mathcal{P}_{\tilde{r}(T)}(\hat{T})/\mathbb{R}
\label{W_interior_div}
\end{equation}
\begin{equation}
\int_{T}(W_{T,t}\omega(\mathbf{x})-\omega(\mathbf{x}))^{\top}B_{T}^{-\top}
\hat{\mathbf{h}}_{i}(\hat{\mathbf{x}}(\mathbf{x}),t) d\mathbf{x}=0\quad
1\leq i\leq k_{\tilde{r}}
\label{W_interior_aux}
\end{equation}
\begin{equation}
\int_{e}[(W_{T,t}\omega-\omega)
\cdot\mathbf{n}]\eta ds = 0\quad 
\forall \eta\in\mathcal{P}_{\tilde{r}(e)}(e),\: 
\forall e\in\triangle_{1}(T)
\label{W_edge_normal}
\end{equation}
\begin{equation}
\int_{e}[(W_{T,t}\omega-\omega)
\cdot\mathbf{t}]\eta ds = 0\quad
\forall \eta\in\mathcal{P}_{\tilde{r}(e)}(e),\: \forall e\in\triangle_{1}(T)
\label{W_edge_tangent}
\end{equation}
\begin{equation}
W_{T,t}\omega = 0 \text{ at all vertices of }T
\label{W_vertex}
\end{equation}
Here $\mathbf{n},\mathbf{t}$ denote the normal and tangent unit vectors along $\partial T$.
As before, $\hat{\mathbf{x}} = \hat{\mathbf{x}}(\mathbf{x})$ stands for the
inverse of the affine mapping from the master to the physical element, defined along
with matrix $B_{T}$ in Remark \ref{jacobian_affine}. Also, according 
to Remark~\ref{jacobian_affine}, $\tilde{r}(f)=\tilde{r}(\hat{f})$ for any $\hat{f}\in
\triangle (\hat{T})$, where $f$ is the image of $\hat{f}$ under the affine mapping 
from $\hat{T}$ to $T$. 
\end{definition}

\begin{lemma}
\label{C_well_defined}
Operator $C_{t}:H^{1}(\hat{T};\mathbb{R}^{2})\longrightarrow 
\mathcal{P}_{\tilde{r}+2}\Lambda^{0}(\hat{T};
\mathbb{R}^{2})$ is well-defined for almost all $t\in [0,1]$.
\end{lemma}

\begin{proof}
Take an arbitrary $t\in [0,1]$. Since $\hat{\omega}\in H^{1}(\hat{T};\mathbb{R}^{2})$, then 
all right hand sides of (\ref{C_interior_div}),(\ref{C_interior_aux}),(\ref{C_edge_normal}),
(\ref{C_edge_tangent}) 
are continuous functionals with respect to $\hat{\omega}$.
In order to show the well-definedness of $C_{t}$, 
it is sufficient to show that $\hat{\omega}=0$ if 
$\hat{\omega}\in\mathcal{P}_{\tilde{r}+2}\Lambda^{0}(\hat{T};\mathbb{R}^{2})$,
$\hat{\omega}=0$ at all vertices of $\hat{T}$, and $C_{t}\hat{\omega}=0$.
According to conditions (\ref{C_edge_normal}) and~(\ref{C_edge_tangent}), 
$\hat{\omega}\in\mathring{\mathcal{P}}_{\tilde{r}+2}\Lambda^{0}(\hat{T};\mathbb{R}^{2})$.
Define now $\mathring{C}_{t}:H^{1}(\hat{T};\mathbb{R}^{2})\longrightarrow \mathring{\mathcal{P}}_{\tilde{r}+2}
\Lambda^{0}(\hat{T};\mathbb{R}^{2})$ by the relations
\begin{equation}
\int_{\hat{T}}\text{div}_{\hat{\mathbf{x}}}\mathring{C}_{t}\hat{\omega}
(\hat{\mathbf{x}})\hat{\psi}(\hat{\mathbf{x}}) d\hat{\mathbf{x}}=
\int_{\hat{T}}\text{div}_{\hat{\mathbf{x}}}\hat{\omega}
(\hat{\mathbf{x}})\hat{\psi}(\hat{\mathbf{x}}) d\hat{\mathbf{x}} \quad
\forall \hat{\psi}(\hat{\mathbf{x}})\in\mathcal{P}_{\tilde{r}(\hat{T})}(\hat{T})/\mathbb{R}
\label{C0_interior_div}
\end{equation}
\begin{equation}
\int_{\hat{T}}(\mathring{C}_{t}\hat{\omega}(\hat{\mathbf{x}}))^{\top}
\hat{\mathbf{h}}_{i}(\hat{\mathbf{x}},t) d\hat{\mathbf{x}}=
\int_{\hat{T}}(\hat{\omega}(\hat{\mathbf{x}}))^{\top}
\hat{\mathbf{h}}_{i}(\hat{\mathbf{x}},t) d\hat{\mathbf{x}}\quad
1\leq i\leq k_{\tilde{r}}
\label{C0_interior_aux}
\end{equation}
It is sufficient to show that operator 
$\mathring{C}_{T,t}:H^{1}(\hat{T};\mathbb{R}^{2})\longrightarrow 
\mathring{\mathcal{P}}_{\tilde{r}+2}\Lambda^{0}(\hat{T};\mathbb{R}^{2})$ is well-defined.

Obviously, the right-hand side of conditions~(\ref{C0_interior_div}) 
and~(\ref{C0_interior_aux}) are continuous functionals with respect to $\hat{\omega}$. 
Set $r=\tilde{r}(\hat{T})$ and
denote by $\mathring{C}(t,r)$ the matrix associated with the left-hand side of 
conditions~(\ref{C0_interior_div}) 
and~(\ref{C0_interior_aux}) 
corresponding to some basis of 
$\mathring{\mathcal{P}}_{\tilde{r}+2}^{-}\Lambda^{0}(\hat{T};\mathbb{R}^{2})$ 
(the solution space), and some basis of $\mathcal{P}_{\tilde{r}(\hat{T})}(\hat{T})/\mathbb{R}$. 
We need to show that the
$\det(\mathring{C}(t,r))\neq 0$ for almost all $t\in [0,1]$.

Since $\hat{\omega}$ vanishes on the boundary,
we can integrate by parts (\ref{C0_interior_div}) without getting any boundary terms.
So (\ref{C0_interior_div}) is the same as
 \begin{equation}
\int_{\hat{T}}(\mathring{C}_{t}(\hat{\omega})(\hat{\mathbf{x}}))^{\top}\cdot 
\text{grad}_{\hat{\mathbf{x}}}\hat{\psi}(\hat{\mathbf{x}}) d\hat{\mathbf{x}}=
\int_{\hat{T}}(\hat{\omega}(\hat{\mathbf{x}}))^{\top}\cdot \text{grad}_{\hat{\mathbf{x}}}
\hat{\psi}(\hat{\mathbf{x}}) d\hat{\mathbf{x}} \quad 
\forall \hat{\psi}
\in\mathcal{P}_{\tilde{r}(\hat{T})}(\hat{T})/\mathbb{R}.
\label{C0_interior_div_by_part}
\end{equation}

Notice that $\mathring{\mathcal{P}}_{\tilde{r}+2}\Lambda^{0}(\hat{T};\mathbb{R}^{2})=\hat{\chi}(\hat{\mathbf{x}})
[\mathcal{P}_{\tilde{r}(\hat{T})-1}(\hat{T})]^{2}$ where 
$\hat{\chi}$ is a third order polynomial vanishing along
boundary $\partial\hat{T}$ and positive in the interior of $\hat{T}$ (a ``bubble''). 
And notice that  
\begin{align*}
& \{\mathbf{h}_{\tilde{r},1}(\hat{\mathbf{x}},1),\cdots, \mathbf{h}_{\tilde{r},k_{\tilde{r}}}(\hat{\mathbf{x}},1)\}\oplus \text{grad}_{\hat{\mathbf{x}}}\mathcal{P}_{\tilde{r}(\hat{T})}(\hat{T})/\mathbb{R}\\
=& \{\mathbf{g}_{\tilde{r},1}(\hat{\mathbf{x}}),\cdots, \mathbf{g}_{\tilde{r},k_{\tilde{r}}}(\hat{\mathbf{x}})\}\oplus \text{grad}_{\hat{\mathbf{x}}}\mathcal{P}_{\tilde{r}(\hat{T})}(\hat{T})=[\mathcal{P}_{\tilde{r}(\hat{T})-1}(\hat{T})]^{2}.
\end{align*}
Consequently, $\det(\mathring{C}(1,r))\neq 0$. Since 
$\det(\mathring{C}(t,r))$ is a polynomial in $t$, 
by the fundamental theorem of algebra argument again,
$\det(\mathring{C}(1,r))$ vanishes at a finite number of roots only.

We can conclude thus that operator
$C_{t}:H^{1}(\hat{T};\mathbb{R}^{2})\longrightarrow 
\mathcal{P}_{\tilde{r}+2}\Lambda^{0}(\hat{T};
\mathbb{R}^{2})$ is well-defined for almost all $t\in [0,1]$.
Notice that the number of roots is independent of the choice of the bases, and depends
upon $\tilde{r}(\hat{T})$ only. 
\end{proof}

\begin{theorem}
\label{W_inequality}
For any $t\in [0,1]$ such that $C_{t}:H^{1}(\hat{T};\mathbb{R}^{2})\longrightarrow \mathcal{P}_{\tilde{r}+2}
\Lambda^{0}(\hat{T};\mathbb{R}^{2})$ is well-defined,
there exists $C>0$ such that, for any $T\in\mathcal{T}_{h}$, 
operator $W_{T,t}:H^{1}(T;\mathbb{R}^{2})\longrightarrow 
\mathcal{P}_{\tilde{r}+2}\Lambda^{0}(T;\mathbb{R}^{2})$
is well-defined as well and,
\begin{equation}
\label{bound_WT}
\Vert \text{curl}W_{T,t}\omega\Vert_{L^{2}(T)} \leq C (h_{T}^{-1}\Vert\omega\Vert_{L^{2}(T)} 
+ \Vert \omega \Vert_{H^{1}(T)})\quad
\forall \omega\in H^{1}(T;\mathbb{R}^{2}).
\end{equation}
Here $h_{T}$ is the diameter of $T$.
\end{theorem}

\begin{proof}
For any $T\in\mathcal{T}_{h}$,  we define a linear isomorphism $A_{T}$ from 
$H^{1}(\hat{T};\mathbb{R}^{2})$ to $H^{1}(T;\mathbb{R}^{2})$ by 
$(A_{T}\hat{\omega})(\mathbf{x}(\hat{\mathbf{x}}))=\dfrac{B_{T}\hat{\omega}(\hat{\mathbf{x}})}{\det(B_{T})}$.
It is easy to see that $A_{T}$ is a linear isomorphism from 
$\mathcal{P}_{\tilde{r}+2}\Lambda^{0}(\hat{T};\mathbb{R}^{2})$ to 
$\mathcal{P}_{\tilde{r}+2}\Lambda^{0}(T;\mathbb{R}^{2})$.

By pulling back to $\hat{T}$ by $A_{T}$ (applied to both $\omega$ and $W_{T,t}\omega$),  
we can see that (\ref{W_interior_div}) is the same as (\ref{C_interior_div}), (\ref{W_interior_aux}) 
is the same as (\ref{C_interior_aux}), (\ref{W_edge_normal}) is the same as (\ref{C_edge_normal}) 
and (\ref{W_vertex}) is the same as (\ref{C_vertex}). (\ref{W_edge_tangent}) becomes 
\begin{equation}
\label{W_edge_tangent_tranformed}
 \int_{\hat{e}}[(A_{T}^{-1}(W_{T,t}\omega))^{\top}B_{T}^{\top}B_{T}\hat{\mathbf{t}}]\hat{\eta}d\hat{s}= 
\int_{\hat{e}}[\hat{\omega}^{T}B_{T}^{\top}B_{T}\hat{\mathbf{t}}]\hat{\eta}d\hat{s}\quad
\forall \hat{\eta}\in\mathcal{P}_{\tilde{r}(\hat{e})}(\hat{e}), \forall \hat{e}\in\triangle_{1}(\hat{T}).
\end{equation}
Notice that $\hat{\mathbf{n}}$ and $B_{T}^{\top}B_{T}\hat{\mathbf{t}}$ are not parallel to each other. 
Since $\hat{\mathbf{t}}^{\top}B_{T}^{\top}B_{T}\hat{\mathbf{t}}\neq 0$, $\hat{\mathbf{t}}$ is not perpendicular
to $B_{T}^{\top}B_{T}\hat{\mathbf{t}}$. Since $\hat{\mathbf{n}}\bot\hat{\mathbf{t}}$, then $\hat{\mathbf{n}}$
is not parallel to $B_{T}^{\top}B_{T}\hat{\mathbf{t}}$.
This means that conditions (\ref{W_edge_normal}) and (\ref{W_edge_tangent}) imply conditions (\ref{C_edge_normal}) and (\ref{C_edge_tangent}) 
by the definition of $A_{T}$. So we can conclude that $W_{T,t}$ is well-defined when $C_{t}$ is well-defined, and
\begin{equation}
\label{transform_AT}
W_{T,t}\omega(\mathbf{x}(\hat{\mathbf{x}})) = A_{T}C_{t}\hat{\omega}(\hat{\mathbf{x}}). 
\end{equation}
Finally, inequality (\ref{bound_WT}) results from standard scaling techniques.
\end{proof}

\begin{definition}
\label{parameter_t_order}
Let $\tilde{r}:\triangle (\hat{T})\rightarrow\mathbb{Z}_{+}$ be a locally variable order of 
discretization that satisfies the minimum rule. 
According to Lemma~\ref{PB_modifiled_well_defined} and Lemma~\ref{C_well_defined}, 
there exists $t_{\tilde{r}(\hat{T})}\in [0,1]$ such that both $\Pi_{\tilde{r}+1,\hat{T},t_{\tilde{r}(\hat{T})}}^{1,-}$ 
and $C_{t_{\tilde{r}(\hat{T})}}$ are well-defined. And the value of $t_{\tilde{r}(\hat{T})}$ depends on $\tilde{r}(\hat{T})$
only.
\end{definition}

\subsection{Projection operators on the whole mesh}

According to Lemma~\ref{PB_modifiled_well_defined}, Theorem~\ref{H_div_transform_minus}, 
Lemma~\ref{C_well_defined}, and Theorem~\ref{W_inequality}, there exist 
$\{t_{1},t_{2},\cdots \}\subset\mathbb{Z}_{+}$ such that, for any $T\in\mathcal{T}_{h}$,
$\Pi_{\tilde{r}+1,T,t_{\tilde{r}(T)}}^{1,-}$ and $W_{T,t_{\tilde{r}(T)}}$ are well-defined linear operators. From now on, 
we rename $\Pi_{\tilde{r}+1,T,t_{\tilde{r}(T)}}^{1,-}$ by $\Pi_{\tilde{r}+1,T}^{1,-}$, 
and $W_{T,t_{\tilde{r}(T)}}$ 
by $W_{T}$.

\begin{definition}
\label{projections_whole}
We define the following global interpolation operators,
$$
\Pi_{\tilde{r},\mathcal{T}_{h}}^{2}:L^{2}(\Omega)\longrightarrow \mathcal{P}_{\tilde{r}}\Lambda^{2}(\mathcal{T}_{h}),\quad 
(\Pi_{\tilde{r},\mathcal{T}_{h}}^{2}u)|_{T} = \Pi_{\tilde{r},T}^{2}(u|_{T})
$$ 
$$
\tilde{\Pi}_{\tilde{r},\mathcal{T}_{h}}^{2}:L^{2}(\Omega;\mathbb{R}^{2})\longrightarrow \mathcal{P}_{\tilde{r}}
\Lambda^{2}(\mathcal{T}_{h};\mathbb{R}^{2}),\quad 
(\tilde{\Pi}_{\tilde{r},\mathcal{T}_{h}}^{2}(u_{1},u_{2})^{\top})|_{T} = (\Pi_{\tilde{r},T}^{2}(u_{1}|_{T}),
\Pi_{\tilde{r},T}^{2}(u_{2}|_{T}))^{\top}
$$ 
$$
\Pi_{\tilde{r}+1,\mathcal{T}_{h}}^{1,-}:H^{1}(\Omega;\mathbb{R}^{2})\longrightarrow \mathcal{P}_{\tilde{r}+1}^{-1}
\Lambda^{1}(\mathcal{T}_{h}),\quad
(\Pi_{\tilde{r}+1,\mathcal{T}_{h}}^{1,-}\omega)|_{T} = \Pi_{\tilde{r}+1,T}^{1,-}(\omega|_{T})
$$
$$
\tilde{\Pi}_{\tilde{r}+1,\mathcal{T}_{h}}^{1}:H^{1}(\Omega;\mathbb{M})\longrightarrow \mathcal{P}_{\tilde{r}+1}
\Lambda^{1}(\mathcal{T}_{h};\mathbb{R}^{2}),\quad
(\tilde{\Pi}_{\tilde{r}+1,\mathcal{T}_{h}}^{1}\sigma)|_{T} = \left[\begin{array}{cc}\tau_{11} & \tau_{12}\\ 
\tau_{21} & \tau_{22}\end{array} \right]
$$
where 
$$
\left[ \begin{array}{c} \tau_{11} \\ \tau_{12}\end{array} 
\right]=\Pi_{\tilde{r}+1,\mathcal{T}_{h}}^{1}\left[ \begin{array}{c} 
\sigma_{11}|_{T} \\ \sigma_{12}|_{T} \end{array}\right],\quad
\left[ \begin{array}{c} \tau_{21} \\ \tau_{22}\end{array} \right]=
\Pi_{\tilde{r}+1,\mathcal{T}_{h}}^{1}\left[ \begin{array}{c} \sigma_{21}|_{T} 
\\ \sigma_{22}|_{T} \end{array}\right], \quad
\sigma=\left[ \begin{array}{cc}\sigma_{11} & \sigma_{12}\\\sigma_{21} & \sigma_{22}
\end{array}\right]
$$
$$
W_{\mathcal{T}_{h}}:H^{1}(\Omega;\mathbb{R}^{2})\longrightarrow \mathcal{P}_{\tilde{r}+2}
\Lambda^{0}(\mathcal{T}_{h};\mathbb{R}^{2}),\quad
(W_{\mathcal{T}_{h}}\omega)|_{T} = W_{T}(\omega|_{T})
$$
for all $T\in\mathcal{T}_{h}$.
\end{definition}

\begin{theorem}
\label{H_div_whole}
There exists $C>0$ such that
\begin{align*}
\Vert\tilde{\Pi}_{\tilde{r}+1,\mathcal{T}_{h}}^{1}\sigma\Vert_{L^{2}(\Omega)}\leq C\Vert\sigma\Vert_{H^{1}(\Omega)},\quad 
\Vert\Pi_{\tilde{r}+1,\mathcal{T}_{h}}^{1,-}\omega\Vert_{L^{2}(\Omega)}\leq C\Vert\omega\Vert_{H^{1}(\Omega)},
\end{align*}
for any $\sigma\in H^{1}(\Omega;\mathbb{M})$ and $\omega\in H^{1}(\Omega;\mathbb{R}^{2})$.
Moreover,
$$
\text{div}\tilde{\Pi}_{\tilde{r}+1,\mathcal{T}_{h}}^{1}\sigma=\tilde{\Pi}_{\tilde{r},\mathcal{T}_{h}}^{2}\text{div}\sigma, \quad
\text{div}\Pi_{\tilde{r}+1,\mathcal{T}_{h}}^{1,-}\omega=\Pi_{\tilde{r},\mathcal{T}_{h}}^{2}\text{div}\omega
$$
\end{theorem}

\begin{proof}
This is an immediate result of Lemma~\ref{H_div_inequality1} and Theorem~\ref{H_div_transform_minus}.
\end{proof}

\begin{definition}
Let $R_{h}$ denote the generalized Clement interpolant operator from Theorem 5.1 in
 \cite{CB:1989:OIC}, 
mapping $H^{1}(\Omega;\mathbb{R}^{2})$ into 
$\mathcal{P}_{1}\Lambda^{0}(\mathcal{T}_{h};\mathbb{R}^{2})$. 
We define 
$$
\tilde{W}_{h}=W_{h}(I-R_{h})+R_{h}
$$
\end{definition}

\begin{theorem}
\label{bounded_W}
There exists $C>0$ such that
$$
\Vert\text{curl}\tilde{W}_{h}\omega\Vert_{L^{2}(\Omega)}
\leq C\Vert\omega\Vert_{H^{1}(\Omega)}\quad
\forall \omega\in H^{1}(\Omega;\mathbb{R}^{2})
$$
Operator $\tilde{W}_{h}$ maps $H^{1}(\Omega;\mathbb{R}^{2})$ 
into $\mathcal{P}_{\tilde{r}+2}\Lambda^{0}(T;\mathbb{R}^{2})$ and satisfies the
condition
$$
\Pi_{\tilde{r}+1,\mathcal{T}_{h}}^{1,-}\omega 
= \Pi_{\tilde{r}+1,\mathcal{T}_{h}}^{1,-}\tilde{W}_{h}\omega\quad
\forall \omega\in H^{1}(\Omega;\mathbb{R}^{2})
$$
\end{theorem}

\begin{proof}
We utilize Example~1 from \cite{CB:1989:OIC} (in this case, $R_{h}$ is the same as standard Clement operator) 
with uniform order equal $1$ to construct operator $R_{h}$. Operator $R_{h}$ maps $H^{1}(\Omega;\mathbb{R}^{2})$ 
into $\mathcal{P}_{1}\Lambda^{0}(T;\mathbb{R}^{2})
\subset\mathcal{P}_{\tilde{r}+2}\Lambda^{0}(T;\mathbb{R}^{2})$. 
According to Theorem~5.1 from \cite{CB:1989:OIC}, there exists a constant $c>0$ such that, 
for any $T\in\mathcal{T}_{h}$,
\begin{equation}
\label{clement_interpolant_inequality}
\Vert\omega-R_{h}\omega\Vert_{L^{2}(T)}
+h_{T}\Vert\omega-R_{h}\omega\Vert_{H^{1}(T)}\leq 
c h_{T}\Vert\omega\Vert_{H^{1}(K_{T})}\quad
\forall \omega\in H^{1}(\Omega;\mathbb{R}^{2})
\end{equation}
where $K_{T}$ stands for the patch of elements adjacent to $T$,
$K_{T}=\bigcup_{T^{\prime}\in\mathcal{T}_{h},T^{\prime}\cap T\neq\emptyset}T^{\prime}$.

As $(\mathcal{T}_{h})_{h}$ is regular, 
$\sup_{h}\sup_{T\in\mathcal{T}_{h}}\# \{T^{\prime}\in\mathcal{T}_{h}:
T^{\prime}\subset K_{T}\}<\infty$.
We have
\begin{align*}
\Vert\text{curl}\tilde{W}_{h}\omega\Vert_{L^{2}(\Omega)}  \leq &
\Vert\text{curl}W_{h}(I-R_{h})\omega\Vert_{L^{2}(\Omega)}+\Vert\text{curl}R_{h}\omega\Vert_{L^{2}(\Omega)} \\
\leq & c(h_{T}^{-1}\Vert (I-R_{h})\omega\Vert_{L^{2}(\Omega)}+\Vert (I-R_{h})\omega\Vert_{H^{1}(\Omega)}
+\Vert\text{curl}R_{h}\omega\Vert_{L^{2}(\Omega)}) \\
\leq & C\Vert\omega\Vert_{H^{1}(\Omega)}\quad \forall \omega\in H^{1}(\Omega;\mathbb{R}^{2}).
\end{align*}
The second inequality above holds by Theorem~\ref{W_inequality} and the third one
by~(\ref{clement_interpolant_inequality}).

According to the definition of $\Pi_{\tilde{r}+1,\mathcal{T}_{h}}^{1,-}$,
 and the definition of $W_{h}$, $\Pi_{\tilde{r}+1,\mathcal{T}_{h}}^{1,-}\omega=\Pi_{\tilde{r}+1,\mathcal{T}_{h}}^{1,-}W_{h}\omega$,
 for any $\omega\in H^{1}(\Omega;\mathbb{R}^{2})$. Notice that $(I-R_{h})\omega\in \mathcal{P}_{1}\Lambda^{0}(\mathcal{T}_{h})\subset H^{1}(\Omega)$.
This implies
$$
\Pi_{\tilde{r}+1,\mathcal{T}_{h}}^{1,-}(I-\tilde{W}_{h})\omega
=\Pi_{\tilde{r}+1,\mathcal{T}_{h}}^{1,-}
(I-W_{h})(I-R_{h})\omega=0
$$
So $\Pi_{\tilde{r}+1,\mathcal{T}_{h}}^{1,-}\omega = \Pi_{\tilde{r}+1,\mathcal{T}_{h}}^{1,-}\tilde{W}_{h}\omega$.
\end{proof}

\section{Stability of the finite element discretization}

We need the following well-known result from partial differential
equations, see e.g. \cite{GR:1986:FENS}.

\begin{lemma}
\label{PDE_lemma}Let $\Omega$ be a bounded domain in $\mathbb{R}^{2}$ with a Lipschitz boundary. 
Then, for all $\mu\in L^{2}(\Omega)$, there exists $\eta\in H^{1}(\Omega;\mathbb{R}^{2})$ satisfying 
$\text{div}\eta=\mu$. If, in addition, $\int_{\Omega}\mu=0$, then we can choose $\eta\in\mathring{H}^{1}(\Omega)$.
\end{lemma}

\bigskip The main result of this paper for affine meshes is Theorem~\ref{thm:main_theorem_affine} below. 
Its proof follows the lines of Theorem~9.1 from \cite{Falk:2008:FME}, Theorem~7.1 from
\cite{AFW:2007:MMEW} and Theorem~11.4 from \cite{AFW:2006:ECH}. The main
difference is the use of operator~$\tilde{W}_{h}$ in place of the operator $\tilde{\Pi}
_{h}^{n-2}$ in \cite{Falk:2008:FME}.

\begin{lemma}
\label{affine_S2}
There exists $c>0$ such that, for any $(\omega,\mu)\in 
\mathcal{P}_{\tilde{r}}\Lambda^{2}(\mathcal{T}_{h})\times 
\mathcal{P}_{\tilde{r}}\Lambda^{2}(\mathcal{T}_{h};\mathbb{R}^{2})$, there exists 
$\sigma\in\mathcal{P}_{\tilde{r}+1}\Lambda^{1}(\mathcal{T}_{h};\mathbb{R}^{2})$ 
such that 
$$
\text{div}\sigma =\mu, \quad -\Pi_{\tilde{r},\mathcal{T}_{h}}^{2}S_{1}\sigma
=\omega
$$
and 
$$
\Vert\sigma\Vert_{H(\text{div},\Omega)}\leq c(\Vert\mu\Vert_{L^{2}(\Omega)}
+\Vert\omega\Vert_{L^{2}(\Omega)})
$$
Here, the constant $c$ depends on 
$\sup_{h}\sup_{T\in\mathcal{T}_{h}}\tilde{r}(T)$. 
\end{lemma}

\begin{proof}
(1) By Lemma~\ref{PDE_lemma}, we can find $\eta\in H^{1}(\Omega;\mathbb{M})$ with 
$\text{div}\eta=\mu$ and $\Vert\eta\Vert_{H^{1}(\Omega)}\leq c\Vert\mu\Vert_{L^{2}(\Omega)}$.

(2) Since $\omega+\Pi_{\tilde{r},h}^{2}S_{1}\tilde{\Pi}_{\tilde{r}+1,h}^{1}\eta\in L^{2}(\Omega)$, we can 
apply Lemma~\ref{PDE_lemma} again to find $\tau\in H^{1}(\Omega;\mathbb{R}^{2})$
with $\text{div}\tau=\omega+\Pi_{\tilde{r},h}^{2}S_{1}\tilde{\Pi}_{\tilde{r}+1,h}^{1}\eta$ 
such that
$$
\Vert \tau\Vert _{H^{1}(\Omega)}\leq c(\Vert \omega\Vert_{L^{2}(\Omega)}
+\Vert \Pi_{\tilde{r},h}^{2}S_{1}\tilde{\Pi}_{\tilde{r}+1,h}^{1}\eta\Vert_{L^{2}(\Omega)})
$$

(3) Define $\sigma=\text{curl}\tilde{W}_{h}\tau+\tilde{\Pi}_{\tilde{r}+1,h}^{1}\eta\in
\mathcal{P}_{\tilde{r}+1}\Lambda^{1}(\mathcal{T}_{h};\mathbb{R}^{2})$.
$\sigma\in \mathcal{P}_{\tilde{r}+1}\Lambda^{1}(\mathcal{T}_{h};\mathbb{R}^{2})$ by
Lemma~\ref{embedding_affine_triangulation}.

(4) From Step (3), Theorem~\ref{H_div_whole}, Step (1), 
and the fact that $\tilde{\Pi}_{\tilde{r},h}^{2}
\tilde{\Pi}_{\tilde{r},h}^{2}=\tilde{\Pi}_{\tilde{r},h}^{2}$, we obtain
\[
\text{div}\sigma=\text{div}\tilde{\Pi}_{\tilde{r}+1,h}^{1}\eta=\tilde{\Pi}_{\tilde{r},h}^{2}\text{div}\eta=
\tilde{\Pi}_{\tilde{r},h}^{2}\mu=\mu.
\]

(5) Notice that $\text{div}\omega=-S_{1}\text{curl}\omega$,
for any $\omega\in H^{1}(\Omega;\mathbb{R}^{2})$. 
Here operator $S_{1}$ is defined in (\ref{operator_S1}).
Then $\Pi_{\tilde{r},h}^{2}\text{div}\Pi_{\tilde{r}+1,h}^{1,-}\omega=-\Pi_{\tilde{r},h}^{2}S_{1}\text{curl}\omega$, 
for any $\omega\in H^{1}(\Omega;\mathbb{R}^{2})$.

(6) Also from Step (3), 
$-\Pi_{\tilde{r},h}^{2}S_{1}\sigma=-\Pi_{\tilde{r},h}^{2}S_{1}\text{curl}\tilde{W}_{h}\tau-
\Pi_{\tilde{r},h}^{2}S_{1}\tilde{\Pi}_{\tilde{r}+1,h}^{1}\eta$.
Applying, in order, Step (5), Theorem~\ref{bounded_W}, Theorem~\ref{H_div_whole}, Step (2), 
and the fact 
that $\Pi_{\tilde{r},h}^{2}\Pi_{\tilde{r},h}^{2}
=\Pi_{\tilde{r},h}^{2}$, we have
\begin{align*}
-\Pi_{\tilde{r},h}^{2}S_{1}\text{curl}\tilde{W}_{h}\tau &  =\Pi_{\tilde{r},h}^{2}\text{div}
\Pi_{\tilde{r}+1,h}^{1,-}\tilde{W}_{h}\tau
=\Pi_{\tilde{r},h}^{2}\text{div}\Pi_{\tilde{r}+1,h}^{1,-}\tau=\Pi_{\tilde{r},h}^{2}\text{div}\tau\\
&  =\Pi_{\tilde{r},h}^{2}(\omega+\Pi_{\tilde{r},h}^{2}S_{1}\tilde{\Pi}_{\tilde{r}+1,h}^{1}\eta)
  =\omega+\Pi_{\tilde{r},h}^{2}S_{1}\tilde{\Pi}_{\tilde{r}+1,h}^{1}\eta.
\end{align*}
Combining, we have $-\Pi_{\tilde{r},h}^{2}S_{1}\sigma=\omega$.

(7) Finally, we prove the norm bound. From the boundedness of $S_{1}$ in
$L^{2}$, Theorem~\ref{H_div_whole}, and Step (1), we obtain
\begin{align*}
\Vert \Pi_{\tilde{r},h}^{2}S_{1}\tilde{\Pi}_{\tilde{r}+1,h}^{1}\eta\Vert_{L^{2}(\Omega)} & \leq
c\Vert S_{1}\tilde{\Pi}_{\tilde{r}+1,h}^{1}\eta\Vert_{L^{2}(\Omega)} \leq c\Vert
\tilde{\Pi}_{\tilde{r}+1,h}^{1}\eta\Vert_{L^{2}(\Omega)}\\ & \leq c\left\Vert \eta\right\Vert _{H^{1}(\Omega)}\leq
c\Vert\mu\Vert_{L^{2}(\Omega)} .
\end{align*}
Combining the result with the bound in Step (2), we get $\Vert
\tau\Vert _{H^{1}(\Omega)}\leq c(\Vert\omega\Vert_{L^{2}(\Omega)} +\Vert\mu\Vert_{L^{2}(\Omega)})$. 
Theorem~\ref{bounded_W} then yields
$$
\Vert\text{curl}\tilde{W}_{h}\tau\Vert_{L^{2}(\Omega)} \leq c\Vert\tau\Vert _{H^{1}(\Omega)}\leq c(\Vert\omega
\Vert_{L^{2}(\Omega)}+\Vert\mu\Vert_{L^{2}(\Omega)}).
$$
From Theorem~\ref{H_div_whole} 
and the bound in Step (1), 
$$
\Vert\tilde{\Pi}_{,\tilde{r}+1,h}^{1}\eta\Vert_{L^{2}(\Omega)} \leq c\Vert
\eta\Vert _{H^{1}(\Omega)}\leq c\Vert\mu\Vert_{L^{2}(\Omega)}.
$$ 
In view of the
definition of $\sigma$, these two last bounds imply that $\Vert
\sigma\Vert_{L^{2}(\Omega)} \leq c(\Vert\omega\Vert_{L^{2}(\Omega)}+\Vert
\mu\Vert_{L^{2}(\Omega)})$, while $\Vert\text{div}\sigma\Vert_{L^{2}(\Omega)}=\Vert
\mu\Vert_{L^{2}(\Omega)}$ by Step (4).
This finishes the proof.
\end{proof}

\begin{theorem}
\label{thm:main_theorem_affine}
There exists $c>0$ such that, 
 for solution $(\sigma,u,p)$ of elasticity system~(\ref{weak_symm_formula_continuous}),
 and corresponding solution $(\sigma_{h},u_{h},p_{h})$ of discrete system~(\ref{weak_symmetry_differential_form}), we have
\begin{align*}
& \Vert\sigma-\sigma_{h}\Vert _{H(\text{div},\Omega)}+\Vert u-u_{h}
\Vert_{L^{2}(\Omega)}+\Vert p-p_{h}\Vert_{L^{2}(\Omega)}\\ \leq & c\inf [\Vert
\sigma-\tau\Vert _{H(\text{div},\Omega)}+\Vert u-v\Vert_{L^{2}(\Omega)}+\Vert
p-q\Vert_{L^{2}(\Omega)}],
\end{align*}
where the infimum is taken over all $\tau\in\mathcal{P}_{\tilde{r}+1}\Lambda^{1}(\mathcal{T}_{h},\mathbb{R}^{2}), 
v\in\mathcal{P}_{\tilde{r}}\Lambda^{2}(\mathcal{T}_{h},\mathbb{R}^{2}),$ and
$q\in\mathcal{P}_{\tilde{r}}\Lambda^{2}(\mathcal{T}_{h})$.
\end{theorem}

\begin{proof}
We need to show that conditions~(\ref{S1_condition}) and~(\ref{S2_condition}) are satisfied.
Condition~(\ref{S1_condition}) follows from the fact that, by construction, 
$\text{div}\mathcal{P}_{\tilde{r}+1}\Lambda^{1}(\mathcal{T}_{h},\mathbb{R}^{2})\subset\mathcal{P}_{\tilde{r}}
\Lambda^{2}(\mathcal{T}_{h};\mathbb{R}^{2})$, and the fact that $A$ is coercive.
Condition~(\ref{S2_condition}) comes from Lemma~\ref{affine_S2}. This finishes the proof.
\end{proof}

\section{\bigskip Curvilinear Meshes}

In practice, meshes generated by CAD software are usually curvilinear. 
In the following sections, we generalize our result for affine meshes -
Theorem~\ref{thm:main_theorem_affine} to curvilinear meshes in the sense of 
asymptotic $h$-stability.

\subsection{Mesh regularity assumptions} 

\begin{definition}
(Curved triangle)
\label{curved_triangle}
A closed set $T\subset\mathbb{R}^2$ is a curved triangle if there exists 
a $C^{1}$-diffeomorphism $G_{T}$ from reference triangle 
$\hat{T}$ onto $T$. This means that $G_{T}$ is a bijection from $\hat{T}$ to $T$ 
such that $G_{T}\in C^{1}(\hat{T})$ and $G_{T}^{-1}\in C^{1}(T)$. We assume additionally that 
$\det(DG_{T}(\hat{\mathbf{x}}))>0$ for any $\hat{\mathbf{x}}\in\hat{T}$.
\end{definition}
We represent $G_{T}$ in the form
\begin{equation}
G_{T} = \tilde{G}_{T} + \Phi_{T},
\end{equation}
where $\tilde{G}_{T}:\hat{\boldsymbol{x}}\rightarrow B_{T}\hat{\boldsymbol{x}}+b_{T}$,
$B_{T}=DG_{T}(\hat{\boldsymbol{p}})$ with $\hat{\boldsymbol{p}}$ being 
the centroid of $\hat{T}$, and $\Phi_{T}$ a $C^{1}$-mapping from $\hat{T}$ 
into $\mathbb{R}^2$. 
The images of edges and vertices of $\hat{T}$ by $G_{T}$ are  
edges and vertices of $T$ respectively. We denote 
$\triangle_{i}(T)=G_{T}(\triangle_{i}(\hat{T})), i=0,1,2$, 
and $\triangle (T)=G_{T}(\triangle (\hat{T}))$.

Definition \ref{curved_triangle} is practically identical with the definition
of curved finite elements 
introduced in \cite{CB:1989:OIC}.

\begin{definition}
A curved triangle $T$ is of class $C^{k}, k\geq 1$, if the mapping $G_{T}\in C^{k}(\hat{T})$.
Similarly, a curved triangle $T$ is of class $C^{k,1}, k\geq 1$, if the mapping $G_{T}\in C^{k,1}(\hat{T})$.
\end{definition}

We define $\mathcal{T}_{h}$ to be a finite set of curved triangles $T$, 
where $h$ denotes the maximal distance between two vertices of $T\in\mathcal{T}_{h}$. 
We define vertices of $\mathcal{T}_{h}$ to be vertices of $T\in\mathcal{T}_{h}$, 
and we define curves of $\mathcal{T}_{h}$ to be edges of $T\in\mathcal{T}_{h}$.
We assume that any edge of $T\in\mathcal{T}_{h}$ is either an edge of another curved
triangle in $\mathcal{T}_{h}$, or part of the boundary of $\mathcal{T}_{h}$.

Each curve of $\mathcal{T}_{h}$ is parametrized with a map from the reference unit
interval into $\mathbb{R}^2$,
\begin{equation*}
[0,1]\ni s\rightarrow \boldsymbol{x}_{e}(s)\in\mathbb{R}^2
\end{equation*}
The parametrization determines the orientation of the curve.

Let $\mathbf{\zeta}(s)$ be the local parametrization for a particular edge of 
a curved triangle $T\in\mathcal{T}_{h}$, occupied by a curve $e$ of $\mathcal{T}_{h}$.
This means that $\mathbf{\zeta}(s)$ is an affine mapping from the reference interval onto 
an edge of the reference triangle $\hat{T}$, whose image under the mapping $G_{T}$ is 
exactly the particular edge of $T$. We can choose $\mathbf{\zeta}(s)$ so that 
$G_{T}(\mathbf{\zeta}(s))$ has the same orientation as $\mathbf{x}_{e}(s)$.

\begin{definition}
($C^{0}$-compatible mesh)
\label{C0_compatible_mesh}
$\mathcal{T}_{h}$ is called a $C^{0}$-compatible mesh if, 
for any curve $e$ and any curved triangle $T$ which contains $e$ as an edge, 
there is a local parametrization $\mathbf{\zeta}(s)$ of $e$ satisfying
\begin{equation*}
G_{T}(\mathbf{\zeta}(s)) = \mathbf{x}_{e}(s).
\end{equation*}
\end{definition}
The concept is illustrated in Fig~\ref{compatibility}.

\begin{figure}[!ht]
\begin{center}
\includegraphics[scale=0.35]{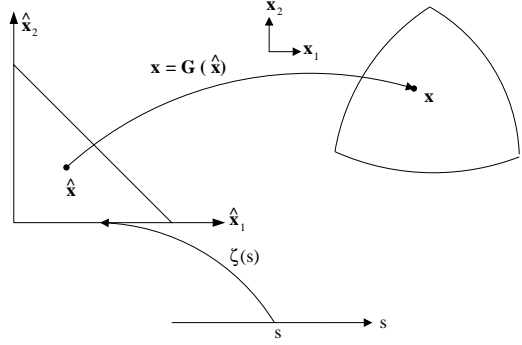}
\end{center}
\caption{Compatibility of edge and triangle parametrizations.}
\label{compatibility}
\end{figure}

\noindent We denote 
\begin{equation}
c_{h}:=\sup_{T\in\mathcal{T}_{h}}((\sup_{\hat{\mathbf{x}}\in\hat{T}}\Vert 
D\Phi_{T}(\hat{\mathbf{x}})\Vert)\Vert B_{T}^{-1}\Vert)
\label{eq:c_h}
\end{equation}
For each $T\in\mathcal{T}_{h}$, we define $\tilde{T}=\tilde{G}_{T}(\hat{T})$. 
We denote by $\tilde{h}_{T}$ the diameter of $\tilde{T}$ and by $\tilde{\rho}_{T}$ 
the diameter of the sphere inscribed in $\tilde{T}$.

We define $\triangle_{i}(\mathcal{T}_{h})=\bigcup_{T\in\mathcal{T}_{h}}\triangle_{i} T$, 
and $\triangle (\mathcal{T}_{h})=\bigcup_{T\in\mathcal{T}_{h}}\triangle T$. 

\begin{remark}
In order to simplify analysis,
compared with \cite{CB:1989:OIC}, our definition of~(\ref{eq:c_h}) replaces
$\Vert D\Phi_{T}\cdot B_{T}^{-1}\Vert$ with the upper bound
$\Vert D\Phi_{T}\Vert \cdot\Vert B_{T}^{-1}\Vert$. 
\end{remark}

\begin{definition}
The family $(\mathcal{T}_{h})_{h}$ of $C^{0}$-compatible meshes is said to be regular if
\begin{equation*} 
\sup_{h}\sup_{T\in\mathcal{T}_{h}}\tilde{h}_{T}/\tilde{\rho}_{T}=\sigma<\infty 
\text{, and } \lim_{h\rightarrow 0}c_{h}=0
\end{equation*}
\end{definition}

We show the construction of $(\mathcal{T}_{h})_{h}$ of $C^{0}$-compatible meshes 
in Appendix~\ref{app_mesh_generation}.

\begin{lemma}
\label{regular_equivalent}
There exist $c_{1},c_{2}>0$ such that, for any triangle $T$,
\begin{equation*}
c_{1}\Vert B_{T}\Vert\cdot\Vert B_{T}^{-1}\Vert\leq \tilde{h}_{T}/\tilde{\rho}_{T} \leq 
c_{2}\Vert B_{T}\Vert\cdot\Vert B_{T}^{-1}\Vert,
\end{equation*}
where $\mathbf{x}=B_{T}\hat{\mathbf{x}}+b_{T}$ is the affine homeomorphism from $\hat{T}$ to $T$.
\end{lemma}

\begin{proof}
$\tilde{h}_{T}/\tilde{\rho}_{T} \leq c_{2}\Vert B_{T}\Vert\cdot\Vert B_{T}^{-1}\Vert$ 
comes from the geometric meaning of singular values of matrix $B_{T}$. 
$c_{1}\Vert B_{T}\Vert\cdot\Vert B_{T}^{-1}\Vert\leq \tilde{h}_{T}/\tilde{\rho}_{T}$ is 
a consequence of Theorem $3.1.3$ in \cite{Ciarlet:2002:FEMelliptic}.
\end{proof}

\begin{lemma}
\label{geometry_property1}
Let family $(\mathcal{T}_{h})_{h}$ be regular. Then, for any indices $i,j,k,l \in \{1,2,3\}$,
 we have 
\begin{equation*}
\lim_{h\rightarrow 0}\sup_{T\in\mathcal{T}_{h}}\sup_{\hat{\mathbf{x}}\in\hat{T}}
\vert \dfrac{(B_{T})_{ij}(D\Phi_{T}(\hat{\mathbf{x}}))_{kl}}{\det(B_{T})} \vert=0
\end{equation*}
\end{lemma}

\begin{proof}
For any $T\in\mathcal{T}_{h}$ and any $\hat{\mathbf{x}}\in\hat{T}$, we have
\begin{equation*}
\vert \dfrac{(B_{T})_{ij}(D\Phi_{T}(\hat{\mathbf{x}}))_{kl}}{\det(B_{T})} \vert 
= \vert (D\Phi_{T}(\hat{\mathbf{x}}))_{kl} \vert \Vert B_{T}^{-1}\Vert\cdot 
\vert (B_{T})_{ij}/\det(B_{T})\vert\dfrac{1}{\Vert B_{T}^{-1}\Vert}.
\end{equation*}

Since $\Vert B_{T}^{-1}\Vert\cdot\Vert B_{T}\Vert\geq 1$, 
$\dfrac{1}{\Vert B_{T}^{-1}\Vert}\leq \Vert B_{T}\Vert$. 
Consequently,
\begin{equation*}
\vert \dfrac{(B_{T})_{ij}(D\Phi_{T}(\hat{\mathbf{x}}))_{kl}}{\det(B_{T})} \vert 
\leq (\Vert D\Phi_{T}(\hat{\mathbf{x}}) \Vert \cdot \Vert B_{T}^{-1}\Vert)
\Vert B_{T}\Vert^{2}/\vert \det(B_{T})\vert.
\end{equation*}

Since $\sup_{h}\sup_{T\in\mathcal{T}_{h}}\tilde{h}_{T}/\tilde{\rho}_{T}=\sigma<\infty$, 
$\Vert B_{T}\Vert^{2}/\vert \det(B_{T})\vert\leq c\sigma^{2}$ with $c>0$.

Since $\lim_{h\rightarrow 0}c_{h}=0$, we have 
$
\lim_{h\rightarrow 0}\sup_{T\in\mathcal{T}_{h}}\sup_{\hat{\mathbf{x}}\in\hat{T}}
\vert \dfrac{(B_{T})_{ij}(D\Phi_{T}(\hat{\mathbf{x}}))_{kl}}{\det(B_{T})} \vert=0.
$
\end{proof}

\begin{lemma}
\label{geometry_property2}
If a family $(\mathcal{T}_{h})_{h}$ is regular, then we have 
\begin{equation*}
\lim_{h\rightarrow 0}\sup_{T\in\mathcal{T}_{h}}\sup_{\hat{\mathbf{x}}\in\hat{T}}
\Vert B_{T}(DG_{T}(\hat{\mathbf{x}}))^{-1} - I\Vert= 
\lim_{h\rightarrow 0}\sup_{T\in\mathcal{T}_{h}}\sup_{\hat{\mathbf{x}}\in\hat{T}}
\Vert (DG_{T}(\hat{\mathbf{x}}))^{-1}B_{T} - I\Vert=0.
\end{equation*}
\end{lemma}

\begin{proof}
For any $T\in\mathcal{T}_{h}$ and any $\hat{\mathbf{x}}\in\hat{T}$, we have
\begin{equation*}
 \Vert B_{T}(DG_{T}(\hat{\mathbf{x}}))^{-1} - I\Vert 
= \Vert B_{T}(B_{T}+D\Phi_{T}(\hat{\mathbf{x}}))^{-1}-I\Vert 
= \Vert (I+D\Phi_{T}(\hat{\mathbf{x}})B_{T}^{-1})^{-1}-I\Vert.
\end{equation*}
Since $c_{h}\rightarrow 0$ as $h\rightarrow 0$, we have
$
\lim_{h\rightarrow 0}\sup_{T\in\mathcal{T}_{h}}\sup_{\hat{\mathbf{x}}\in\hat{T}}
\Vert B_{T}(DG_{T}(\hat{\mathbf{x}}))^{-1} - I\Vert=0.
$
Proof of the second property is fully analogous.
\end{proof} 

\begin{lemma}
\label{geometry_property3}
If a family $(\mathcal{T}_{h})_{h}$ is regular, then we have 
\begin{equation*}
\lim_{h\rightarrow 0}\sup_{T\in\mathcal{T}_{h}}\sup_{\hat{\mathbf{x}}\in\hat{T}}
\vert\det(DG_{T}(\hat{\mathbf{x}})B_{T}^{-1})-1\vert
=\lim_{h\rightarrow 0}\sup_{T\in\mathcal{T}_{h}}\sup_{\hat{\mathbf{x}}\in\hat{T}}
\vert \det(B_{T}(DG_{T}(\hat{\mathbf{x}}))^{-1})-1\vert=0.
\end{equation*}
\end{lemma}

\begin{proof}
Since $\lim_{h\rightarrow 0}c_{h}=0$,  
$
\lim_{h\rightarrow 0}\sup_{T\in\mathcal{T}_{h}}\sup_{\hat{\mathbf{x}}\in\hat{T}}
\vert\det(DG_{T}(\hat{\mathbf{x}})B_{T}^{-1})-1\vert=0.
$
By Lemma~\ref{geometry_property2}, 
$
\lim_{h\rightarrow 0}\sup_{T\in\mathcal{T}_{h}}\sup_{\hat{\mathbf{x}}\in\hat{T}}
\vert \det(B_{T}(DG_{T}(\hat{\mathbf{x}}))^{-1})-1\vert=0.
$
\end{proof}

\section{Finite element spaces on curvilinear meshes}

We begin by introducing the relevant finite element spaces on any curved triangle $T$ 
by the
pull back mappings associated with the inverse of $G_{T}$, where $G_{T}$ maps from 
the reference triangle $\hat{T}$ to $T$, see Definition~\ref{curved_triangle}. Then, we  
define the finite element spaces on a whole mesh $\mathcal{T}_{h}$ by ``gluing" 
the finite element spaces on curved triangles.
\subsection{Finite element spaces on a curved triangle}

Let $T$ be a curved triangle from Definition~\ref{curved_triangle}
with $G_T$ denoting the corresponding 
$C^{1}$-diffeomorphism from $\hat{T}$ to $T$, $\mathbf{x}=G_{T}(\hat{\mathbf{x}})$.
We begin by introducing formally the mapping $\tilde{r}$ 
from $\triangle (T)$ to $\mathbb{Z}_{+}$ specifying the local order of discretization.

\begin{definition}
\label{variable_order_curved_triangle}
We denote by $\tilde{r}$ a mapping from $\triangle (T)$ 
to $\mathbb{Z}_{+}$ such that if $e,f\in\triangle (T)$ and $e\subset f$ then
$\tilde{r}(e)  \leq\tilde{r}(f)$. With the same symbol $\tilde{r}$ we denote
the corresponding mapping from $\triangle (\hat{T})$ to $\mathbb{Z}_{+}$,
$\tilde{r}(\hat{f}):=\tilde{r}(f)$ for 
any $\hat{f}\in\triangle (\hat{T})$, where $f=G_{T}(\hat{f})$. 
\end{definition}
We define now the following FE spaces on $T$:
\begin{definition}
\label{FE_space_curved_triangle}
\begin{align*}
\mathcal{P}_{\tilde{r}}\Lambda^{0}(T) & :=\{u(\mathbf{x}):\hat{u}(\hat{\mathbf{x}})\in
\mathcal{P}_{\tilde{r}}\Lambda^{0}(\hat{T}) \text{ where } u(\mathbf{x})=\hat{u}(\hat{\mathbf{x}})\},\\
\mathcal{P}_{\tilde{r}}\Lambda^{1}(T) & := \{\omega(\mathbf{x}):\hat{\omega}(\hat{\mathbf{x}})\in 
\mathcal{P}_{\tilde{r}}\Lambda^{1}(\hat{T}) \text{ where } \omega (\mathbf{x})=\dfrac{1}{\det(DG_{T}(\hat{\mathbf{x}}))} 
DG_{T}(\hat{\mathbf{x}})\hat{\omega}(\hat{\mathbf{x}})\},\\
\mathcal{P}_{\tilde{r}}^{-}\Lambda^{1}(T) & := \{\omega(\mathbf{x}):\hat{\omega}(\hat{\mathbf{x}})\in 
\mathcal{P}_{\tilde{r}}^{-}\Lambda^{1}(\hat{T}) \text{ where } \omega (\mathbf{x})=\dfrac{1}{\det(DG_{T}
(\hat{\mathbf{x}}))} DG_{T}(\hat{\mathbf{x}})\hat{\omega}(\hat{\mathbf{x}})\},\\
\mathcal{P}_{\tilde{r}}\Lambda^{2}(T) & := \{u(\mathbf{x}):\hat{u}(\hat{\mathbf{x}})\in 
\mathcal{P}_{\tilde{r}}\Lambda^{2}(\hat{T}) \text{ where } u(\mathbf{x})=\dfrac{1}{\det(DG_{T}(\hat{\mathbf{x}}))} 
\hat{u}(\hat{\mathbf{x}})\},\\
\mathcal{P}_{\tilde{r}}\Lambda^{0}(T;\mathbb{R}^{2}) & := \{(u_{1},u_{2}):u_{1},u_{2}\in
\mathcal{P}_{\tilde{r}}\Lambda^{0}(T)\},\\
\mathcal{P}_{\tilde{r}}\Lambda^{2}(T;\mathbb{R}^{2}) & := \{(u_{1},u_{2}):u_{1},u_{2}\in
\mathcal{P}_{\tilde{r}}\Lambda^{2}(T)\},\\
\mathcal{P}_{\tilde{r}}\Lambda^{1}(T;\mathbb{R}^{2}) & := \left\{\left[ \begin{array}{cc}\sigma_{11} & \sigma_{12}\\
\sigma_{21} & \sigma_{22} \end{array}\right]:(\sigma_{11},\sigma_{12})^{\top}, (\sigma_{21},\sigma_{22})^{\top}\in
\mathcal{P}_{\tilde{r}}\Lambda^{1}(T)\right\}.
\end{align*}
\end{definition}

\begin{remark}
Since $G_{T}:\hat{T}\rightarrow T$ is a $C^{1}$-diffeomorphism  
with $\det (DG_{T}(\hat{\mathbf{x}}))\neq 0$, for 
any $\hat{\mathbf{x}}\in\hat{T}$, the formulae in 
Definition~\ref{FE_space_curved_triangle} are well-defined. The
mappings used in the definition are the standard pull back mappings for differential 
forms $\Lambda^{0},\Lambda^{1},\Lambda^{2}$, see e.g. formulas $(2.24),(2.26),(2.27)$
in \cite{Demkowicz:2007:HPAFE2}.
\end{remark}

\begin{lemma}
\label{piola_transform}
For any edge $e\in\triangle_{1}(T)$, let $\zeta (s)$ be the 
local parametrization for $e$ discussed in Definition~\ref{C0_compatible_mesh}, i.e.
the affine mapping from the reference 
interval onto $\hat{e}\in\triangle_{1}(\hat{T})$. 
We have then,
\begin{align*}
\mathcal{P}_{\tilde{r}}\Lambda^{0}(T)|_{e} & = \{u(\mathbf{x}) \text{ where }\mathbf{x}\in e:
u(G_{T}(\zeta(s)))\in\mathcal{P}_{\tilde r(\hat{e})}(\hat{e})\}, \\
\mathcal{P}_{\tilde{r}}\Lambda^{1}(T)\cdot \mathbf{n}|_{e} & = 
\{ u(\mathbf{x})  \text{ where }\mathbf{x}\in e:  u(G_{T}(\zeta(s)))\Vert D(G_{T}\circ \zeta)(s)\Vert\
\in\mathcal{P}_{\tilde{r}(\hat{e})}(\hat{e})\},\\
\mathcal{P}_{\tilde{r}}^{-}\Lambda^{1}(T)\cdot \mathbf{n}|_{e} & = 
\{ u(\mathbf{x})  \text{ where }\mathbf{x}\in e:  u(G_{T}(\zeta(s)))\Vert D(G_{T}\circ \zeta)(s)\Vert\
\in\mathcal{P}_{\tilde{r}(\hat{e})-1}(\hat{e})\}.
\end{align*}
In addition, the above equalities do not depend on the choice of the orientation of the local 
parametrization $\zeta (s)$.
\end{lemma}

\begin{proof}
These are trivial observations on pull back mappings and their restrictions to edges. 
\end{proof}



\subsection{Finite element spaces on a $C^{0}$-compatible mesh}

Let $\mathcal{T}_{h}$ be a $C^{0}$-compatible mesh from Definition~\ref{C0_compatible_mesh}. 
We extend the $\tilde{r}$ mapping to a global map defined on  $\triangle (\mathcal{T}_{h})$ 
with values in 
$\mathbb{Z}_{+}$ such that if $e\subset f$, then $\tilde{r}(e)\leq \tilde{r}(f)$. 

\begin{definition}
\label{FE_space_curved_triangulation}
We put $\Omega_{h}:=\bigcup_{T\in\mathcal{T}_{h}}T$.
\begin{align*}
C\Lambda^{0}(\mathcal{T}_{h}) & :=\{u\in H^{1}(\Omega_{h}): u 
\text{ is piece-wise smooth with respect to }\mathcal{T}_{h}\},\\
C\Lambda^{1}(\mathcal{T}_{h}) & :=\{\omega\in H(\text{div},\Omega_{h}): \omega 
\text{ is piece-wise smooth with respect to }\mathcal{T}_{h}\},\\
C\Lambda^{2}(\mathcal{T}_{h}) & :=\{u\in L^{2}(\Omega_{h}): u 
\text{ is piece-wise smooth with respect to }\mathcal{T}_{h}\}.
\end{align*}

We define
\begin{align*}
\mathcal{P}_{\tilde{r}}\Lambda^{0}(\mathcal{T}_{h}) & :=\{u\in C\Lambda^{0}(\mathcal{T}_{h}):
u|_{T}\in\mathcal{P}_{\tilde{r}}\Lambda^{0}(T), \forall T\in\mathcal{T}_{h}\},\\
\mathcal{P}_{\tilde{r}}\Lambda^{1}(\mathcal{T}_{h}) & :=\{\omega\in C\Lambda^{1}(\mathcal{T}_{h}):
\omega |_{T}\in\mathcal{P}_{\tilde{r}}\Lambda^{1}(T), \forall T\in\mathcal{T}_{h}\},\\
\mathcal{P}_{\tilde{r}}^{-}\Lambda^{1}(\mathcal{T}_{h}) & :=\{\omega\in C\Lambda^{1}(\mathcal{T}_{h}):
\omega |_{T}\in\mathcal{P}_{\tilde{r}}^{-}\Lambda^{1}(T), \forall T\in\mathcal{T}_{h}\},\\
\mathcal{P}_{\tilde{r}}\Lambda^{2}(\mathcal{T}_{h}) & :=\{u\in C\Lambda^{2}(\mathcal{T}_{h}):
u|_{T}\in\mathcal{P}_{\tilde{r}}\Lambda^{2}(T), \forall T\in\mathcal{T}_{h}\},\\
\mathcal{P}_{\tilde{r}}\Lambda^{0}(\mathcal{T}_{h};\mathbb{R}^{2}) & :=[\mathcal{P}_{\tilde{r}}
\Lambda^{0}(\mathcal{T}_{h})]^{2}, \mathcal{P}_{\tilde{r}}\Lambda^{2}(\mathcal{T}_{h};\mathbb{R}^{2})
:=[\mathcal{P}_{\tilde{r}}\Lambda^{2}(\mathcal{T}_{h})]^{2},\\
\mathcal{P}_{\tilde{r}}\Lambda^{1}(\mathcal{T}_{h};\mathbb{R}^{2}) & := \left\{\left[ \begin{array}{cc}
\sigma_{11} & \sigma_{12}\\\sigma_{21} & \sigma_{22} \end{array}\right]:(\sigma_{11},\sigma_{12})^{\top}, 
(\sigma_{21},\sigma_{22})^{\top}\in\mathcal{P}_{\tilde{r}}\Lambda^{1}(\mathcal{T}_{h},\mathbb{R}^{2})\right\}.
\end{align*}
\end{definition}

\begin{remark}
According to Lemma~\ref{piola_transform} 
and the fact that $\mathcal{T}_{h}$ is $C^{0}$-compatible, we 
can conclude that 
$\mathcal{P}_{\tilde{r}}\Lambda^{0}(\mathcal{T}_{h})|_{T}=\mathcal{P}_{\tilde{r}}\Lambda^{0}(T),
\mathcal{P}_{\tilde{r}}\Lambda^{1}(\mathcal{T}_{h})|_{T}=\mathcal{P}_{\tilde{r}}\Lambda^{1}(T),
\mathcal{P}_{\tilde{r}}^{-}\Lambda^{1}(\mathcal{T}_{h})|_{T}=\mathcal{P}_{\tilde{r}}^{-}\Lambda^{1}(T),
\mathcal{P}_{\tilde{r}}\Lambda^{2}(\mathcal{T}_{h})|_{T}=\mathcal{P}_{\tilde{r}}\Lambda^{2}(T)$,
for any $T\in\mathcal{T}_{h}$. For standard (not curved) triangulations 
$\mathcal{T}_{h}$, spaces 
$\mathcal{P}_{\tilde{r}}\Lambda^{0}(\mathcal{T}_{h}),\mathcal{P}_{\tilde{r}}\Lambda^{1}(\mathcal{T}_{h}),
\mathcal{P}_{\tilde{r}}^{-}\Lambda^{1}(\mathcal{T}_{h}),\mathcal{P}_{\tilde{r}}\Lambda^{2}(\mathcal{T}_{h})$ 
coincide with those analyzed in \cite{QD:2009:MMEW}. 
\end{remark}

\begin{lemma}
\label{embedding_curved_triangulation}
\begin{align*}
& \mathcal{P}_{\tilde{r}}\Lambda^{1}(\mathcal{T}_{h})\subset\mathcal{P}_{\tilde{r}+1}^{-}\Lambda^{1}
(\mathcal{T}_{h})\subset\mathcal{P}_{\tilde{r}+1}\Lambda^{1}(\mathcal{T}_{h}), \\
& \text{div}\mathcal{P}_{\tilde{r}+1}\Lambda^{1}(\mathcal{T}_{h})\subset
\mathcal{P}_{\tilde{r}}\Lambda^{2}(\mathcal{T}_{h}),
\text{curl}\mathcal{P}_{\tilde{r}+1}\Lambda^{0}(\mathcal{T}_{h})\subset
\mathcal{P}_{\tilde{r}}\Lambda^{1}(\mathcal{T}_{h}).
\end{align*} 
\end{lemma}

\begin{proof}
The proof is straightforward.
\end{proof}

\section{Preliminaries for the proof of stability for curvilinear meshes\label{preliminary_curve}}

We begin by recalling our assumptions on the domain and meshes: $\Omega$ is a (curvilinear)
polygon and it is meshed with a family $(\mathcal{T}_{h})_{h}$ of 
$C^{0}$-compatible meshes of class $C^{1,1}$. For any mesh $\mathcal{T}_{h}$, 
mapping $\tilde{r} \: : \: \triangle (\mathcal{T}_{h}) \rightarrow \mathbb{Z}_{+}$ 
defines a locally variable order of discretization that satisfies the minimum rule.
The maximum order is limited, i.e. $\sup_{h}\sup_{T\in\mathcal{T}_{h}}\tilde{r}(T)<\infty$. 

As in Section~\ref{preliminary_affine}, we shall design operators 
$\Pi_{\tilde{r},T}^{2}$, $\Pi_{\tilde{r}+1,T}^{1,-}$ and $W_{T}$ in order to 
make the left and the right diagrams in (\ref{three_diagrams}) commute. 
In the case of affine meshes (Section~\ref{preliminary_affine}), 
we define $\Pi_{\tilde{r}+1,T}^{1,-}$ and $W_{T}$ on physical triangles, 
then ``pull them back" to the reference triangle. We designed 
$\Pi_{\tilde{r}+1,T}^{1,-}$ and $W_{T}$ in such a way that their ``pull backs"
are $\Pi_{\tilde{r}+1,\hat{T},t_{\tilde{r}(\hat{T})}}^{1,-}$ and $C_{t_{\tilde{r}(\hat{T})}}$ 
respectively 
(
Definition~\ref{parameter_t_order}).
In the case of curvilinear meshes, the approach to $\Pi_{\tilde{r},T}^{2}$ and $\Pi_{\tilde{r}+1,T}^{1,-}$
is similar to that in Section~\ref{preliminary_affine}. But the treatment to $W_{T}$ is quite different. 
In order to make the right diagram in (\ref{three_diagrams}) commute (we need the commutativity on the physical meshes), 
we incorporate the construction of $\Pi_{\tilde{r}+1,T}^{1,-}$ into the definition 
of $W_{T}$. Then we use a special mapping such that the ``pull back" of $W_{T}$ is a projection 
operator to $\mathcal{P}_{\tilde{r}+2}\Lambda^{0}(\hat{T};\mathbb{R}^{2})$. In general, this kind of 
``pull back" of $W_{T}$ to the reference triangle does not coincide with $C_{t_{\tilde{r}(T)}}$ from 
Definition~\ref{C_W_operator_reference}. This means that, in general, $W_{T}$ is not necessarily well-defined. 
But we manage to prove that the ``pull back" of $W_{T}$  will converge to $C_{t_{\tilde{r}(T)}}$ as $h\rightarrow 0$. 
So $W_{T}$ becomes well-defined when $h$ is small enough. The work presented here illuminates 
the substantial difference between the stability analysis for curvilinear meshes and that for affine meshes. 
This is in particular reflected in the proof of Theorem~\ref{W_inequality_curved}, which 
uses heavily Lemmas~\ref{geometry_property1},\ref{geometry_property2},\ref{geometry_property3}.

In order to make the paper more readable, most proofs of results discussed in  
this section have been moved into Appendix~\ref{app_proofs}.

\begin{definition}
For any $T\in\mathcal{T}_{h}$, we define a linear operator 
$\Pi_{\tilde{r},T}^{2}:L^{2}(T)\longrightarrow \mathcal{P}_{\tilde{r}}\Lambda^{2}(T)$
by the relations
\begin{equation}
\int_{T}(\Pi_{\tilde{r},T}^{2}u(\mathbf{x})-u(\mathbf{x}))
\hat{\psi}(\hat{\mathbf{x}} (\mathbf{x}) )\: d\mathbf{x}=0 \quad  
\forall \hat{\psi} \in\mathcal{P}_{\tilde{r}(T)}(\hat{T})
\label{L2_non_orthogonal}
\end{equation}
\end{definition}
Above, $\hat{\mathbf{x}}(\mathbf{x})$ signifies the inverse
of the element map $\mathbf{x} = G_T(\hat{\mathbf{x}})$.

\begin{definition}
Operator 
$\Pi_{\tilde{r},\hat{T}}^{2}:L^{2}(\hat{T})\longrightarrow \mathcal{P}_{\tilde{r}}
\Lambda^{2}(\hat{T})$ will denote the $L^2$-projection in the reference space,
\begin{equation}
\int_{T}(\Pi_{\tilde{r},\hat{T}}^{2}\hat{u}(\hat{\mathbf{x}})-
\hat{u}(\hat{\mathbf{x}}))\hat{\psi}
(\hat{\mathbf{x}})
\: d\hat{\mathbf{x}}=0 \quad
\forall \hat{\psi}
\in\mathcal{P}_{\tilde{r}(T)}(\hat{T})
\label{L2_non_orthogonal_reference}
\end{equation}
\end{definition}

\begin{remark}
Operator $\Pi_{\tilde{r},T}^{2}$ is a weighted $L^2$-projection in the physical
space. For a regular triangle (affine element map), the jacobian is constant,
and $\Pi_{\tilde{r},T}^{2}$ reduces to the standard $L^2$-projection in the
physical space.
\end{remark}

\begin{lemma}
\label{L2_transform}
For any $T\in\mathcal{T}_{h}$, and arbitrary $u(\mathbf{x})\in L^{2}(T)$, we 
define  $\hat{u}(\hat{\mathbf{x}})$ by the relation:
$$
u(\mathbf{x}(\hat{\mathbf{x}}))=\dfrac{\hat{u}(\hat{\mathbf{x}})}{\det(DG_{T}(\hat{\mathbf{x}}))}
$$
Then $\hat{u}(\hat{\mathbf{x}})\in L^{2}(\hat{T})$, and
$$
\Pi_{\tilde{r},T}^{2}u(\mathbf{x}(\hat{\mathbf{x}}))=
\dfrac{\Pi_{\tilde{r},\hat{T}}^{2}\hat{u}(\hat{\mathbf{x}})}{\det(DG_{T}(\hat{\mathbf{x}})))}
$$
Above, $\mathbf{x}(\hat{\mathbf{x}})$ signifies the element map $\mathbf{x} = G_T(\hat{\mathbf{x}})$.
\end{lemma}

\begin{proof}
Proof follows immediately from Lemma~\ref{mapping3} and the definitions of the two
projections.
\end{proof}

\begin{lemma}
\label{L2_non_standard_projection_inequality}
For any $\varepsilon>0$, there exists $\delta>0$ such that, for any $h<\delta$ and $T\in\mathcal{T}_{h}$, 
\begin{equation*}
\Vert\Pi_{\tilde{r},T}^{2}u-P_{\tilde{r},T}u\Vert_{L^{2}(T)} \leq\varepsilon\Vert u\Vert_{L^{2}(T)},\forall u\in L^{2}(T).
\end{equation*}
Here $P_{\tilde{r},T}$ is the standard $L^2$-projection onto $\mathcal{P}_{\tilde{r}}\Lambda^{2}(T)$.
\end{lemma}
\begin{proof}
Please see Appendix~\ref{app_proofs}.
\end{proof}

\subsection{Projection Based Interpolation onto $\mathcal{P}_{\tilde{r}+1}\Lambda^{1}(\mathcal{T}_{h})$}

\begin{definition}
For any $T\in\mathcal{T}_{h}$, we define a linear operator 
$\Pi_{\tilde{r}+1,T}^{1}:H^{1}(T;\mathbb{R}^{2})
\longrightarrow \mathcal{P}_{\tilde{r}+1}\Lambda^{1}(T)$
by the relations
\begin{equation}
\int_{T}\text{div}(\Pi_{\tilde{r}+1,T}^{1}\omega-\omega)(\mathbf{x})
\hat{\psi}(\hat{\mathbf{x}} (\mathbf{x})) 
d\mathbf{x}=0\quad
\forall \hat{\psi}
\in\mathcal{P}_{\tilde{r}(T)}(\hat{T})/\mathbb{R}
\label{interior_div_curved}
\end{equation}
\begin{equation}
\int_{T}(\Pi_{\tilde{r}+1,T}^{1}\omega(\mathbf{x})-\omega(\mathbf{x}))^{\top}
DG_{T}(\hat{\mathbf{x}})^{-\top}
\text{curl}_{\hat{\mathbf{x}}}
\hat{\varphi}(\hat{\mathbf{x}}(\mathbf{x})) d\mathbf{x}=0 \quad
\forall \hat{\varphi}
\in\mathcal{\mathring{P}}_{\tilde{r}(T)+2}(\hat{T}) 
\label{interior_aux_curved}
\end{equation}
\begin{equation}
\int_{[0,1]}[(\Pi_{\tilde{r}+1,T}^{1}\omega-\omega)(\mathbf{x}_{e}(s))
\cdot\mathbf{n}(\mathbf{x}_{e}(s))]\hat{\eta}(s)\Vert \dot{\mathbf{x}_{e}}(s)\Vert ds = 0 \quad
\forall \hat{\eta}\in\mathcal{P}_{\tilde{r}(e)+1}([0,1])\: \forall e\in\triangle_{1}(T)
\label{H_div_edge_curved}
\end{equation}
Here $\mathbf{x}=G_{T}(\hat{\mathbf{x}})$ for any $\hat{\mathbf{x}}\in\hat{T}$, 
$\mathbf{x}_{e}(s):[0,1]\rightarrow e$ is the parametrization of $e$, and $\mathbf{n}$ is 
a unit normal vector along $e$ (the choice of its direction does not matter).
\end{definition}

\begin{definition}
\label{PB_master_element}(Projection Based Interpolation operator onto 
$\mathcal{P}_{\tilde{r}+1}\Lambda^{1}(\hat{T})$) 
We define a linear operator $\Pi_{\tilde{r}+1,\hat{T}}^{1}:H(\hat{T})
\longrightarrow \mathcal{P}_{\tilde{r}+1}\Lambda^{1}(\hat{T})$
by the relations
\begin{equation}
\int_{\hat{T}}\text{div}_{\hat{\mathbf{x}}}(\Pi_{\tilde{r}+1,\hat{T}}^{1}\hat{\omega}-\hat{\omega})
\hat{\psi} d\hat{\mathbf{x}}=0 \quad 
\forall \hat{\psi}\in\mathcal{P}_{\tilde{r}(\hat{T})}(\hat{T})/\mathbb{R}
\label{interior_div_reference}
\end{equation}
\begin{equation}
\int_{\hat{T}}(\Pi_{\tilde{r}+1,\hat{T}}^{1}\hat{\omega}-\hat{\omega})\cdot
\text{curl}_{\hat{\mathbf{x}}}\hat{\varphi}d\hat{\mathbf{x}}=0 \quad 
\forall \hat{\varphi}\in\mathcal{\mathring{P}}_{\tilde{r}(\hat{T})+2}(\hat{T})
\label{interior_aux_reference}
\end{equation}
\begin{equation}
\int_{\hat{e}}(\Pi_{\tilde{r}+1,\hat{T}}^{1}\hat{\omega}-\hat{\omega})\cdot\mathbf{\hat{n}}
\hat{\eta} d\hat{s} = 0 \quad \forall \hat{\eta}\in\mathcal{P}_{\tilde{r}(\hat{e})+1}(\hat{e}) 
\:
\forall \hat{e}\in\triangle_{1}(\hat{T})
\label{H_div_edge_reference}
\end{equation}
\end{definition}

\begin{remark}
The operator $\Pi_{\tilde{r}+1,\hat{T}}^{1}$ is the Projection Based Interpolation operator onto 
$\mathcal{P}_{\tilde{r}+1}\Lambda^{1}(\hat{T})$ defined in \cite{QD:2009:MMEW}. The 
operator $\Pi_{\tilde{r}+1,T}^{1}$ is defined by the pull back mapping from $\hat{T}$ to $T$.
\end{remark}

\begin{lemma}
\label{H_div_transform_curved}
For any $T\in\mathcal{T}_{h}$, any $\omega\in [H^{1}(T)]^2$, we define 
$\hat{\omega}$ by the relation
$$
\omega(\mathbf{x}(\hat{\mathbf{x}}))=\dfrac{DG_{T}(\hat{\mathbf{x}})}{\det(DG_{T}(\hat{\mathbf{x}}))}
\hat{\omega}(\hat{x})
$$
Then $\hat{\omega}(\hat{\mathbf{x}})\in H(\hat{T})$, and 
$$
\Pi_{\tilde{r}+1,T}^{1}\omega(\mathbf{x}(\hat{\mathbf{x}})) = 
\dfrac{DG_{T}(\hat{\mathbf{x}})}{\det(DG_{T}(\hat{\mathbf{x}}))}
\Pi_{\tilde{r}+1,\hat{T}}^{1}\hat{\omega}(\hat{\mathbf{x}})
$$
\end{lemma}
\begin{proof}
Please see Appendix~\ref{app_proofs}.
\end{proof}

\begin{lemma}
\label{commuting_diagram1_curved}
For any $T\in\mathcal{T}_{h}$, and any $\omega\in [H^{1}(T)]^2$, we have
$\Pi_{\tilde{r},T}^{2}\text{div}\omega = \text{div}\Pi_{\tilde{r}+1,T}^{1}\omega.$
\end{lemma}
\begin{proof}
Please see Appendix~\ref{app_proofs}.
\end{proof}

\begin{lemma}
\label{H_div_inequality1_curved}
There exists $\delta>0$ and $C>0$ such that, for any $h<\delta$, we have
\begin{equation*}
\Vert \Pi_{\tilde{r},T}^{1}\omega\Vert_{L^{2}(T)}\leq C\Vert \omega\Vert_{H^{1}(T)} \quad
\forall T\in\mathcal{T}_{h},\omega\in H^{1}(T;\mathbb{R}^{2})
\end{equation*}
For affine meshes, the inequality above holds for any $h>0$.
\end{lemma}
\begin{proof}
Please see Appendix~\ref{app_proofs}.
\end{proof}

\subsection{Modified Projection Based Interpolation onto 
$\mathcal{P}_{\tilde{r}+1}^{-}\Lambda^{1}(T)$
and modified operator $W$ onto $\mathcal{P}_{\tilde{r}+2}\Lambda^{0}(T)$}


\begin{definition}
\label{PB_minus_curved}(PB interpolation operator onto 
$\mathcal{P}_{\tilde{r}+1}^{-}\Lambda^{1}(T)$)
For $t_{\tilde{r}(\hat{T})}$ given in Definition~\ref{parameter_t_order}, and 
for any $T\in\mathcal{T}_{h}$, we define a linear operator $\Pi_{\tilde{r}+1,T}^{1,-}
:H^{1}(T;\mathbb{R}^{2})\longrightarrow \mathcal{P}_{\tilde{r}+1}^{-}\Lambda^{1}(T)$
by the relations
\begin{equation}
\int_{T}\text{div}(\Pi_{\tilde{r}+1,T}^{1,-}\omega-\omega)(\mathbf{x})
\hat{\psi}(\hat{\mathbf{x}}(\mathbf{x})) \: d\mathbf{x}=0 \quad 
\forall \hat{\psi}
\in\mathcal{P}_{\tilde{r}(T)}(\hat{T})/\mathbb{R}
\label{interior_div_minus_curved}
\end{equation}
\begin{equation}
\int_{T}(\Pi_{\tilde{r}+1,T}^{1,-}\omega(\mathbf{x})-\omega(\mathbf{x}))^{\top}
DG_{T}(\hat{\mathbf{x}}(\mathbf{x}))^{-\top}
\hat{\mathbf{h}}_{i}(\hat{\mathbf{x}}(\mathbf{x}),t_{\tilde{r}(\hat{T})}) d\mathbf{x}=0 \quad
1\leq i\leq k_{\tilde{r}}
\label{interior_aux_minus_curved}
\end{equation}
\begin{equation}
\int_{[0,1]}[(\Pi_{\tilde{r}+1,T}^{1,-}\omega-\omega)(\mathbf{x}_{e}(s))
\cdot\mathbf{n}(\mathbf{x}_{e}(s))]\hat{\eta}(s)\Vert \dot{\mathbf{x}_{e}}(s)\Vert ds = 0 \quad
\forall \hat{\eta}\in\mathcal{P}_{\tilde{r}(e)}([0,1]),
\forall e\in\triangle_{1}(T)
\label{H_div_edge_minus_curved}
\end{equation}
In the above, $\hat{\mathbf{x}} = \hat{\mathbf{x}}(\mathbf{x})$ signifies the
inverse of the element map.
\end{definition}

\begin{theorem}
\label{H_div_transform_minus_curved}
The operator 
$\Pi_{\tilde{r}+1,T}^{1,-}:H^{1}(T;\mathbb{R}^{2})
\longrightarrow \mathcal{P}_{\tilde{r}+1}^{-}\Lambda^{1}(T)$ 
is well-defined, and we have the following result.

$$
\text{div}\Pi_{\tilde{r}+1,T}^{1,-}\omega=\Pi_{\tilde{r},T}^{2}\text{div}\omega\quad
\forall \omega\in H^{1}(T;\mathbb{R}^{2}).
$$ 

There exist $\delta>0$ and $C>0$ such that,
for $h\leq\delta$, $T\in\mathcal{T}_{h}$, 
and $\omega\in H^{1}(T;\mathbb{R}^{2})$, 
$$
\Vert\Pi_{\tilde{r}+1,T}^{1,-}\omega\Vert_{L^{2}(T)}\leq C\Vert\omega\Vert_{H^{1}(T)} 
$$ 
\end{theorem}

\begin{proof}
The proof is analogous to that of Lemma~\ref{H_div_transform_curved}, Lemma~\ref{commuting_diagram1_curved}, 
and Lemma~\ref{H_div_inequality1_curved}.
\end{proof}

\begin{definition}
For $t_{\tilde{r}(\hat{T})}$ given in Definition~\ref{parameter_t_order}, and for any $T\in\mathcal{T}_{h}$, 
we define a linear operator $W_{T}
:H^{1}(T;\mathbb{R}^{2})\longrightarrow \mathcal{P}_{\tilde{r}+2}\Lambda^{0}(T;\mathbb{R}^{2})$
by the following relations
\begin{equation}
\int_{T}\text{div}(W_{T}\omega-\omega)(\mathbf{x})\hat{\psi}(\hat{\mathbf{x}}(\mathbf{x}))d\mathbf{x}=0
\quad 
\forall \hat{\psi}
\in\mathcal{P}_{\tilde{r}(T)}(\hat{T})/\mathbb{R}
\label{W_interior_div_curved}
\end{equation}
\begin{equation}
\int_{T}(W_{T}\omega(\mathbf{x})-\omega(\mathbf{x}))^{\top}
DG_{T}(\hat{\mathbf{x}}(\mathbf{x}))^{-\top}
\hat{\mathbf{h}}_{i}(\hat{\mathbf{x}}(\mathbf{x}),t_{\tilde{r}(\hat{T})}) d\mathbf{x}=0\quad
1\leq i\leq k_{\tilde{r}}
\label{W_interior_aux_curved}
\end{equation}
\begin{equation}
\int_{[0,1]}[(W_{T,t}\omega-\omega)(\mathbf{x}_{e}(s))
\cdot\mathbf{n}(\mathbf{x}_{e}(s))]\hat{\eta}(s)\Vert \dot{\mathbf{x}_{e}}(s)\Vert ds = 0\quad 
\forall \hat{\eta}\in\mathcal{P}_{\tilde{r}(e)}([0,1]),\: 
\forall e\in\triangle_{1}(T)
\label{W_edge_normal_curved}
\end{equation}
\begin{equation}
\int_{[0,1]}[(W_{T,t}\omega-\omega)(\mathbf{x}_{e}(s))
\cdot\mathbf{t}(\mathbf{x}_{e}(s))]\hat{\eta}(s)\Vert \dot{\mathbf{x}_{e}}(s)\Vert ds = 0\quad
\forall \hat{\eta}\in\mathcal{P}_{\tilde{r}(e)}([0,1]),\: \forall e\in\triangle_{1}(T)
\label{W_edge_tangent_curved}
\end{equation}
\begin{equation}
W_{T,t}\omega = 0 \text{ at all vertices of }T
\label{W_vertex_curved}
\end{equation}
Here $\mathbf{n},\mathbf{t}$ denote the normal and tangent unit vectors along $\partial T$.
\end{definition}

\begin{theorem}
\label{W_inequality_curved}
There exist $\delta>0$ and $C>0$ such that, for any $h<\delta$, $T\in\mathcal{T}_{h}$, 
operator $W_{T}:H^{1}(T;\mathbb{R}^{2})\longrightarrow 
\mathcal{P}_{\tilde{r}+2}\Lambda^{0}(T;\mathbb{R}^{2})$
is well-defined and,
\begin{equation*}
\Vert \text{curl}W_{T}\omega\Vert_{L^{2}(T)} \leq C (\tilde{h}_{T}^{-1}\Vert\omega\Vert_{L^{2}(T)} 
+ \Vert \omega \Vert_{H^{1}(T)})\quad
\forall \omega\in H^{1}(T;\mathbb{R}^{2})
\end{equation*}
For affine meshes, $W_{T}$ is well-defined and the 
above inequality holds for any $h>0$.
\end{theorem}
\begin{proof}
Please see Appendix~\ref{app_proofs}.
\end{proof}

\subsection{Projection operators on the whole curvilinear meshes}

\begin{definition}
\label{projections_whole_curved}
Similar to Definition~\ref{projections_whole}, We define the following global interpolation operators,
$$
\Pi_{\tilde{r},\mathcal{T}_{h}}^{2}:L^{2}(\Omega)\longrightarrow \mathcal{P}_{\tilde{r}}\Lambda^{2}(\mathcal{T}_{h}),\quad 
(\Pi_{\tilde{r},\mathcal{T}_{h}}^{2}u)|_{T} = \Pi_{\tilde{r},T}^{2}(u|_{T})
$$ 
$$
\tilde{\Pi}_{\tilde{r},\mathcal{T}_{h}}^{2}:L^{2}(\Omega;\mathbb{R}^{2})\longrightarrow \mathcal{P}_{\tilde{r}}
\Lambda^{2}(\mathcal{T}_{h};\mathbb{R}^{2}),\quad 
(\tilde{\Pi}_{\tilde{r},\mathcal{T}_{h}}^{2}(u_{1},u_{2})^{\top})|_{T} = (\Pi_{\tilde{r},T}^{2}(u_{1}|_{T}),
\Pi_{\tilde{r},T}^{2}(u_{2}|_{T}))^{\top}
$$ 
$$
\Pi_{\tilde{r}+1,\mathcal{T}_{h}}^{1,-}:H^{1}(\Omega;\mathbb{R}^{2})\longrightarrow \mathcal{P}_{\tilde{r}+1}^{-1}
\Lambda^{1}(\mathcal{T}_{h}),\quad
(\Pi_{\tilde{r}+1,\mathcal{T}_{h}}^{1,-}\omega)|_{T} = \Pi_{\tilde{r}+1,T}^{1,-}(\omega|_{T})
$$
$$
\tilde{\Pi}_{\tilde{r}+1,\mathcal{T}_{h}}^{1}:H^{1}(\Omega;\mathbb{M})\longrightarrow \mathcal{P}_{\tilde{r}+1}
\Lambda^{1}(\mathcal{T}_{h};\mathbb{R}^{2}),\quad
(\tilde{\Pi}_{\tilde{r}+1,\mathcal{T}_{h}}^{1}\sigma)|_{T} = \left[\begin{array}{cc}\tau_{11} & \tau_{12}\\ 
\tau_{21} & \tau_{22}\end{array} \right]
$$
where 
$$
\left[ \begin{array}{c} \tau_{11} \\ \tau_{12}\end{array} 
\right]=\Pi_{\tilde{r}+1,\mathcal{T}_{h}}^{1}\left[ \begin{array}{c} 
\sigma_{11}|_{T} \\ \sigma_{12}|_{T} \end{array}\right],\quad
\left[ \begin{array}{c} \tau_{21} \\ \tau_{22}\end{array} \right]=
\Pi_{\tilde{r}+1,\mathcal{T}_{h}}^{1}\left[ \begin{array}{c} \sigma_{21}|_{T} 
\\ \sigma_{22}|_{T} \end{array}\right], \quad
\sigma=\left[ \begin{array}{cc}\sigma_{11} & \sigma_{12}\\\sigma_{21} & \sigma_{22}
\end{array}\right]
$$
$$
W_{\mathcal{T}_{h}}:H^{1}(\Omega;\mathbb{R}^{2})\longrightarrow \mathcal{P}_{\tilde{r}+2}
\Lambda^{0}(\mathcal{T}_{h};\mathbb{R}^{2}),\quad
(W_{\mathcal{T}_{h}}\omega)|_{T} = W_{T}(\omega|_{T})
$$
for all $T\in\mathcal{T}_{h}$.
\end{definition}

\begin{remark}
Since $(\mathcal{T}_{h})_{h}$ is $C^{0}$-compatible, operators $\Pi_{\tilde{r}+1,\mathcal{T}_{h}}^{1,-}$, 
$\tilde{\Pi}_{\tilde{r}+1,\mathcal{T}_{h}}^{1}$ and $W_{\mathcal{T}_{h}}$ are well-defined.
\end{remark}

\begin{theorem}
\label{nonstandard_L2_norm}
For any $\varepsilon>0$, there exists $\delta>0$ such that, for any $h<\delta$, 
\begin{equation*}
\Vert\Pi_{\tilde{r},\mathcal{T}_{h}}^{2}u-P_{\tilde{r},\mathcal{T}_{h}}u\Vert_{L^{2}(\Omega)} \leq\varepsilon
\Vert u\Vert_{L^{2}(\Omega)}\quad
\forall u\in L^{2}(\Omega)
\end{equation*}
Here $P_{\tilde{r},\mathcal{T}_{h}}$ is the standard $L^2$-projection onto 
$\mathcal{P}_{\tilde{r}}\Lambda^{2}(\mathcal{T}_{h})$.
\end{theorem}

\begin{proof}
This is an immediate result of Lemma~\ref{L2_non_standard_projection_inequality}.
\end{proof}

\begin{theorem}
\label{H_div_whole_curved}
There exist $\delta>0$ and $C>0$ such that, for any $h<\delta$, we have
\begin{align*}
\Vert\tilde{\Pi}_{\tilde{r}+1,\mathcal{T}_{h}}^{1}\sigma\Vert_{L^{2}(\Omega)}\leq C\Vert\sigma\Vert_{H^{1}(\Omega)},\quad 
\Vert\Pi_{\tilde{r}+1,\mathcal{T}_{h}}^{1,-}\omega\Vert_{L^{2}(\Omega)}\leq C\Vert\omega\Vert_{H^{1}(\Omega)},
\end{align*}
for any $\sigma\in H^{1}(\Omega;\mathbb{M})$ and $\omega\in H^{1}(\Omega;\mathbb{R}^{2})$.
For affine meshes,
the inequalities above hold for any $h>0$.
Moreover,
$$
\text{div}\tilde{\Pi}_{\tilde{r}+1,\mathcal{T}_{h}}^{1}\sigma=\tilde{\Pi}_{\tilde{r},\mathcal{T}_{h}}^{2}\text{div}\sigma, \quad
\text{div}\Pi_{\tilde{r}+1,\mathcal{T}_{h}}^{1,-}\omega=\Pi_{\tilde{r},\mathcal{T}_{h}}^{2}\text{div}\omega
$$
\end{theorem}

\begin{proof}
This is an immediate result of Lemma~\ref{H_div_inequality1_curved} and Theorem~\ref{H_div_transform_minus_curved}.
\end{proof}

\begin{definition}
Let $R_{h}$ denote the generalized Clement interpolant operator from Theorem 5.1 in
 \cite{CB:1989:OIC}, 
mapping $H^{1}(\Omega;\mathbb{R}^{2})$ into 
$\mathcal{P}_{1}\Lambda^{0}(\mathcal{T}_{h};\mathbb{R}^{2})$. 
We define 
$$
\tilde{W}_{h}=W_{h}(I-R_{h})+R_{h}
$$
\end{definition}

\begin{theorem}
\label{bounded_W_curved}
There exist $\delta>0$ and $C>0$ such that, for any $h<\delta$,
$$
\Vert\text{curl}\tilde{W}_{h}\omega\Vert_{L^{2}(\Omega)}
\leq C\Vert\omega\Vert_{H^{1}(\Omega)}\quad
\forall \omega\in H^{1}(\Omega;\mathbb{R}^{2})
$$
For affine meshes, the inequality holds for any $h>0$.
Operator $\tilde{W}_{h}$ maps $H^{1}(\Omega;\mathbb{R}^{2})$ 
into $\mathcal{P}_{\tilde{r}+2}\Lambda^{0}(T;\mathbb{R}^{2})$ and satisfies the
condition
$$
\Pi_{\tilde{r}+1,\mathcal{T}_{h}}^{1,-}\omega 
= \Pi_{\tilde{r}+1,\mathcal{T}_{h}}^{1,-}\tilde{W}_{h}\omega\quad
\forall \omega\in H^{1}(\Omega;\mathbb{R}^{2})
$$
\end{theorem}

\begin{proof}
We utilize Example~2 from \cite{CB:1989:OIC} with uniform order equal $1$ to construct 
operator $R_{h}$.
Since $(\mathcal{T}_{h})_{h}$ is $C^{0}$-compatible,
 and $c_{h}\rightarrow 0$ as $h\rightarrow 0$, 
operator $R_{h}$ maps $H^{1}(\Omega;\mathbb{R}^{2})$ 
into $\mathcal{P}_{1}\Lambda^{0}(T;\mathbb{R}^{2})
\subset\mathcal{P}_{\tilde{r}+2}\Lambda^{0}(T;\mathbb{R}^{2})$. 
The the proof will be the same as that of Theorem~\ref{bounded_W}.
\end{proof}

\section{Asymptotic stability of the finite element discretization on curvilinear meshes}

\begin{lemma}
\label{asymptotic_S2}
There exist $\delta>0$ and $c>0$ such that, for any $h<\delta$ and any $(\omega,\mu)\in 
\mathcal{P}_{\tilde{r}}\Lambda^{2}(\mathcal{T}_{h})\times 
\mathcal{P}_{\tilde{r}}\Lambda^{2}(\mathcal{T}_{h};\mathbb{R}^{2})$, there exists 
$\sigma\in\mathcal{P}_{\tilde{r}+1}\Lambda^{1}(\mathcal{T}_{h};\mathbb{R}^{2})$ 
such that 
$$
\text{div}\sigma =\mu, \quad -\Pi_{\tilde{r},\mathcal{T}_{h}}^{2}S_{1}\sigma
=\omega
$$
and 
$$
\Vert\sigma\Vert_{H(\text{div},\Omega)}\leq c(\Vert\mu\Vert_{L^{2}(\Omega)}
+\Vert\omega\Vert_{L^{2}(\Omega)})
$$
Here, the constant $c$ depends on 
$\sup_{h}\sup_{T\in\mathcal{T}_{h}}\tilde{r}(T)$. 
For affine meshes, the inequality above holds for any $h>0$.
\end{lemma}

\begin{proof}
The proof is the same as that of Lemma~\ref{affine_S2}.
\end{proof}

\begin{theorem}
\label{thm:main_theorem_curved}
There exist $\delta>0$ and $c>0$ such that, 
 for solution $(\sigma,u,p)$ of elasticity system~(\ref{weak_symm_formula_continuous}),
 and corresponding solution $(\sigma_{h},u_{h},p_{h})$ of discrete system~(\ref{weak_symmetry_differential_form}), we have
\begin{align*}
& \Vert\sigma-\sigma_{h}\Vert _{H(\text{div},\Omega)}+\Vert u-u_{h}
\Vert_{L^{2}(\Omega)}+\Vert p-p_{h}\Vert_{L^{2}(\Omega)}\\ \leq & c\inf [\Vert
\sigma-\tau\Vert _{H(\text{div},\Omega)}+\Vert u-v\Vert_{L^{2}(\Omega)}+\Vert
p-q\Vert_{L^{2}(\Omega)}],
\end{align*}
where the infimum is taken over all $\tau\in\mathcal{P}_{\tilde{r}+1}\Lambda^{1}(\mathcal{T}_{h},\mathbb{R}^{2}), 
v\in\mathcal{P}_{\tilde{r}}\Lambda^{2}(\mathcal{T}_{h},\mathbb{R}^{2}),$ and
$q\in\mathcal{P}_{\tilde{r}}\Lambda^{2}(\mathcal{T}_{h})$, for $h < \delta$.
For affine meshes, the inequality 
holds for any $h>0$.
\end{theorem}

\begin{proof}
We need to show that conditions~(\ref{S1_condition}) and~(\ref{S2_condition}) are satisfied 
asymptotically in $h$.
Condition~(\ref{S1_condition}) follows from the fact that, by construction, 
$\text{div}\mathcal{P}_{\tilde{r}+1}\Lambda^{1}(\mathcal{T}_{h},\mathbb{R}^{2})\subset\mathcal{P}_{\tilde{r}}\Lambda^{2}(\mathcal{T}_{h};\mathbb{R}^{2})$, and the
fact that $A$ is coercive.

We turn now to condition~(\ref{S2_condition}).
According to Lemma~\ref{asymptotic_S2}, there exist $\delta>0$ and $c>0$ such that,
for $h<\delta$ and $(\omega,\mu)\in 
\mathcal{P}_{\tilde{r}}\Lambda^{2}(\mathcal{T}_{h})\times 
\mathcal{P}_{\tilde{r}}\Lambda^{2}(\mathcal{T}_{h};\mathbb{R}^{2})$, 
there exists 
$\sigma\in\mathcal{P}_{\tilde{r}+1}\Lambda^{1}(\mathcal{T}_{h};\mathbb{R}^{2})$ 
such that $\text{div}\sigma =\mu$, $-\Pi_{\tilde{r},\mathcal{T}_{h}}^{2}S_{1}\sigma
=\omega$, and 
$\Vert\sigma\Vert_{H(\text{div},\Omega)}\leq c(\Vert\mu\Vert_{L^{2}(\Omega)}
+\Vert\omega\Vert_{L^{2}(\Omega)})$.
We have then
\begin{align*}
 \langle\text{div}\sigma,\mu\rangle-\langle S_{1}\sigma,\omega\rangle  = &
\langle\text{div}\sigma,\mu\rangle-\langle \Pi_{\tilde{r},\mathcal{T}_{h}}^{2}S_{1}\sigma,\omega\rangle
+\langle(\Pi_{\tilde{r},\mathcal{T}_{h}}^{2}-P_{\tilde{r},\mathcal{T}_{h}})S_{1}\sigma,\omega\rangle\\
\geq & c\Vert\sigma\Vert_{H(\text{div},\Omega)}(\Vert\mu\Vert_{L^{2}(\Omega)}+\Vert\omega\Vert_{L^{2}(\Omega)})
+\langle(\Pi_{\tilde{r},\mathcal{T}_{h}}^{2}-P_{\tilde{r},\mathcal{T}_{h}})S_{1}\sigma,\omega\rangle.
\end{align*}
According to Theorem~\ref{nonstandard_L2_norm}, for sufficiently small $h$,
\begin{equation*}
\vert\langle(\Pi_{\tilde{r},\mathcal{T}_{h}}^{2}-P_{\tilde{r},\mathcal{T}_{h}})S_{1}\sigma,\omega\rangle\vert
\leq \dfrac{c}{2}\Vert\sigma\Vert_{L^{2}(\Omega)}\Vert\omega\Vert_{L^{2}(\Omega)} 
\end{equation*}
So, asymptotically in $h$, we have
\begin{equation*}
\langle\text{div}\sigma,\mu\rangle-\langle S_{1}\sigma,\omega\rangle\geq\dfrac{c}{2}
\Vert\sigma\Vert_{H(\text{div},\Omega)}(\Vert\mu\Vert_{L^{2}(\Omega)}+\Vert\omega\Vert_{L^{2}(\Omega)}).
\end{equation*}
For affine meshes, $\Pi_{\tilde{r},\mathcal{T}_{h}}^{2}$ reduces to the standard
$L^{2}$-projection. The inequality above holds then for any $h>0$. This finishes the proof.
\end{proof}

\section{Numerical experiments}
We continue numerical experiments initiated in~\cite{QD:2009:MMEW}
where we investigated rates of convergence for uniform 
$h$-refinements in presence of 
non-uniform polynomial order $p$, and tried out $p$-adaptivity
(with no underlying theory at presence). The experiments confirmed the optimal
$h$-convergence rates and indicated the $p$-convergence as well. 

Following~\cite{QD:2009:MMEW}, we consider the L-shape domain, and 
use a manufactured solution corresponding to the exact solution
of the homogeneous equilibrium equations (zero volume forces), and
the corresponding unbounded L-shape domain extending to infinity. The manufactured solution 
is designed in such a way that both $p$ and all stress components
have the same singularity at the origin characterized by term
$r^{-0.39596}$ where $r$ is the distance to the origin. 

Experiments presented below focus on $h$-adaptivity for meshes with uniform order $p$,
and the ultimate goal of this research - the $hp$-adaptivity. We investigate both
affine and curvilinear meshes. We use the standard ``greedy algorithm''
for $h$-refinements, and the two-grid $hp$-algorithm for $hp$-refinements, 
see \cite{Demkowicz:2006:HPAFE} for details. The $hp$-algorithm is based
on the standard Projection Based (PB) interpolation.
All convergence plots are displayed on log-log scale, error vs. number of 
degrees-of-freedom (d.o.f.). The error is always measured in terms of percent
of the total norm of the solution.


\subsection*{Uniform $h$-refinements on affine meshes} 

We begin with a verification of stability on uniform meshes. Fig.~\ref{L_domain}
presents the L-shape domain with an initial mesh of six elements.
\begin{figure}[!ht]
\begin{center}
\includegraphics[scale=.6]{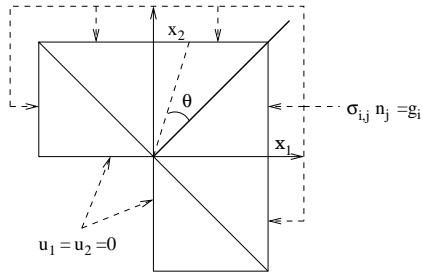}
\end{center}
\caption{The L shape domain with initial mesh.}
\label{L_domain}
\end{figure}
Fig.~\ref{uh_app0} displays the actual approximation error compared with
the best approximation error for a sequence of uniformly refined meshes
of zero order\footnote{We always refer to the order of approximation for
the displacement.}. As expected, the two curves are parallel to each other.
\begin{figure}[!ht]
\begin{center}
\includegraphics[scale=.4]{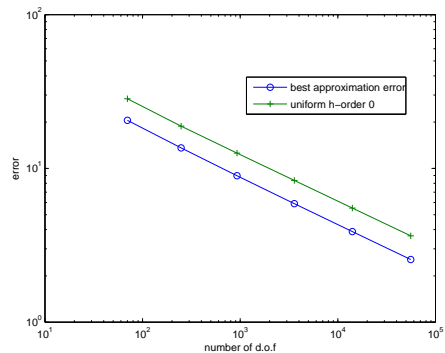}
\end{center}
\caption{Uniform $h$-refinements for $p=0$. FE error vs. best approximation error.}
\label{uh_app0}
\end{figure}
We repeat now the same experiment starting with an initial mesh of elements
with order varying from zero to four, shown in Fig.~\ref{init_mesh}.
\begin{figure}[!ht]
\begin{center}
\includegraphics[scale=.2,angle=90]{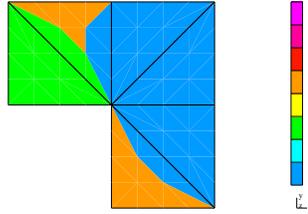}
\end{center}
\caption{The L shape domain with initial mesh.}
\label{init_mesh}
\end{figure}
The corresponding FE error is compared again with the best approximation error
in Fig.~\ref{comp4}. The two lines are again parallel to each other with the
slope determined by the lowest order elements in the mesh (same as in the first
example).
\begin{figure}[!ht]
\begin{center}
\includegraphics[scale=.4]{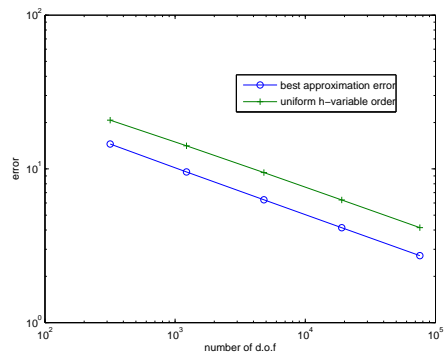}
\end{center}
\caption{Uniform $h$-refinements for a mesh of variable order. 
FE error vs. best approximation error.}
\label{comp4}
\end{figure}

\subsection*{Adaptive $h$- and $hp$-refinements on affine meshes}

We continue now the numerical verification of stability for non-uniform
meshes resulting from $h$- and $hp$-refinements. Unfortunately,
we face a slight discrepancy between the presented theory and the numerical experiments
as the code is using 1-irregular meshes which we have not accounted
for in our theoretical analysis.

Fig.~{\ref{rates}} presents convergence history for adaptive $h$-refinements 
for meshes of order $p=0,1,2$ and adaptive $hp$-refinements 
starting with the mesh of zero order. The approximation error
is compared with the best approximation error computed on the
same meshes (generated by the adaptive algorithm). Results for
$h$-refined meshes of order $p=1,2$, and the $hp$-refinements
confirm the stability result. The result for the lowest order
elements, though, reflects some loss of stability. The only
possible explanation that we have at the moment, is the use
of meshes with hanging nodes. 

\begin{figure}[!htp]
\begin{center}
\subfigure[Adaptive $h$-refinements, $p=0$.]
{\label{rate0}\includegraphics[scale=0.375]{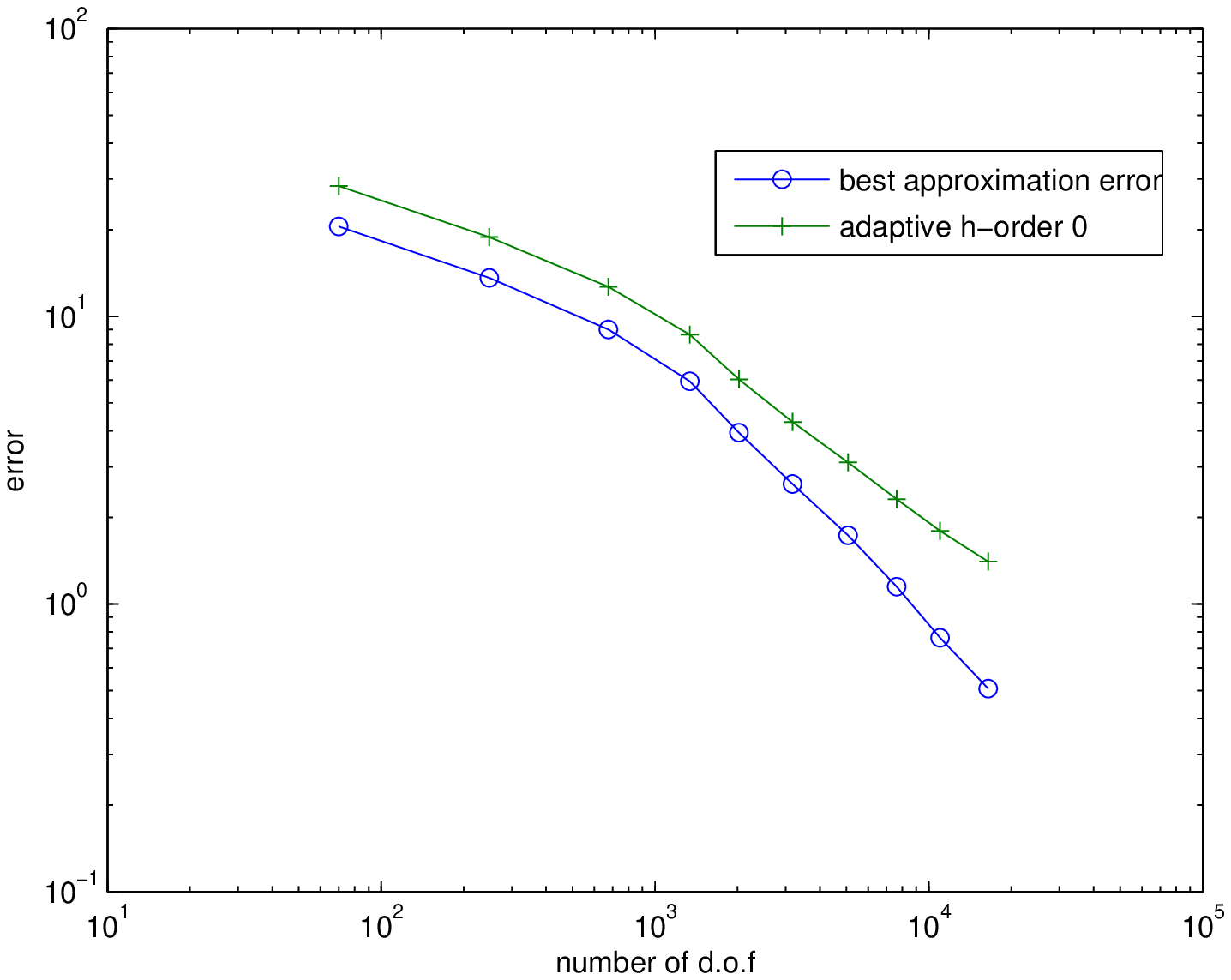}}
\subfigure[Adaptive $h$-refinements, $p=1$.]
{\label{rate1}\includegraphics[scale=0.375]{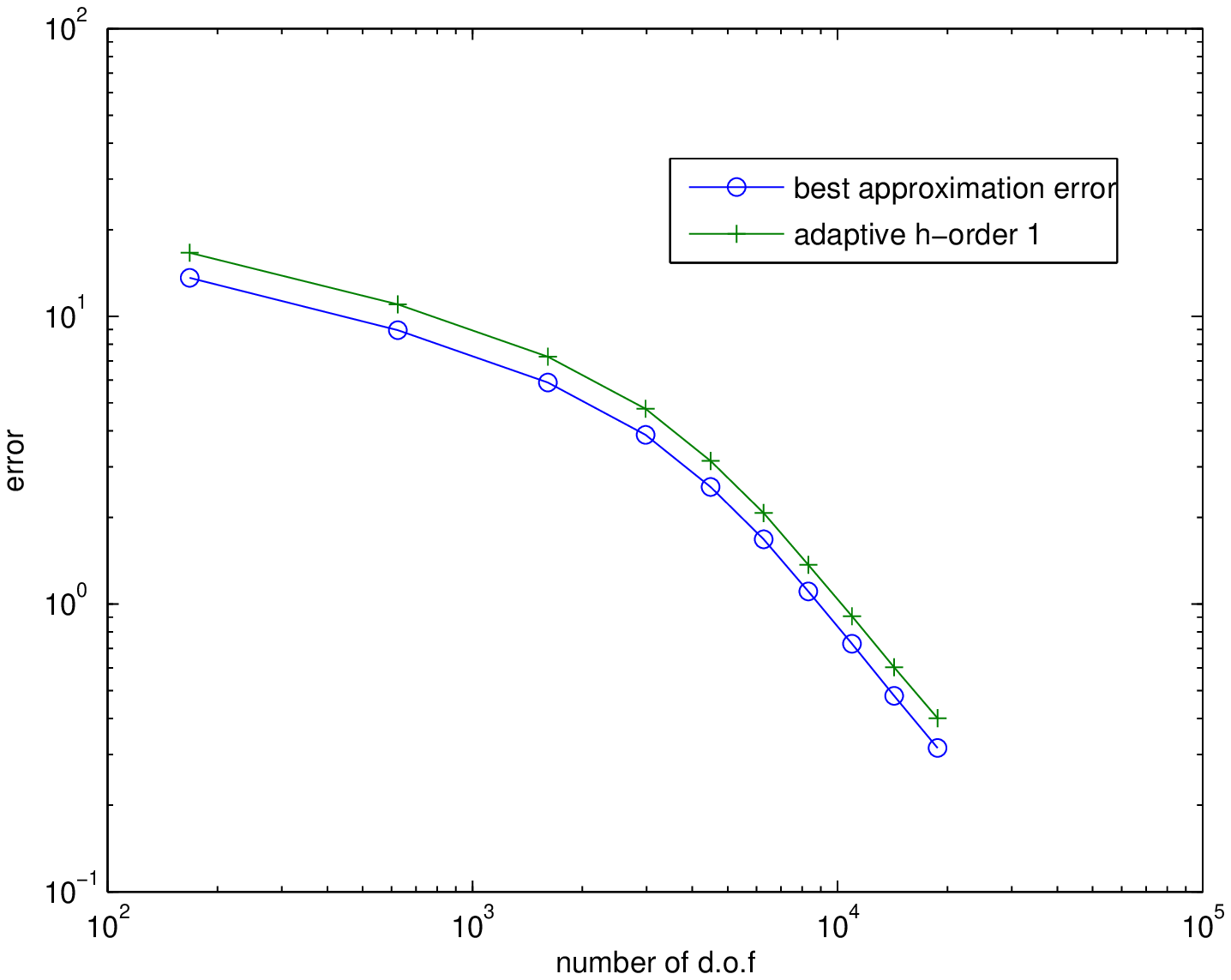}}
\subfigure[Adaptive $h$-refinements, $p=2$.]
{\label{rate2}\includegraphics[scale=0.375]{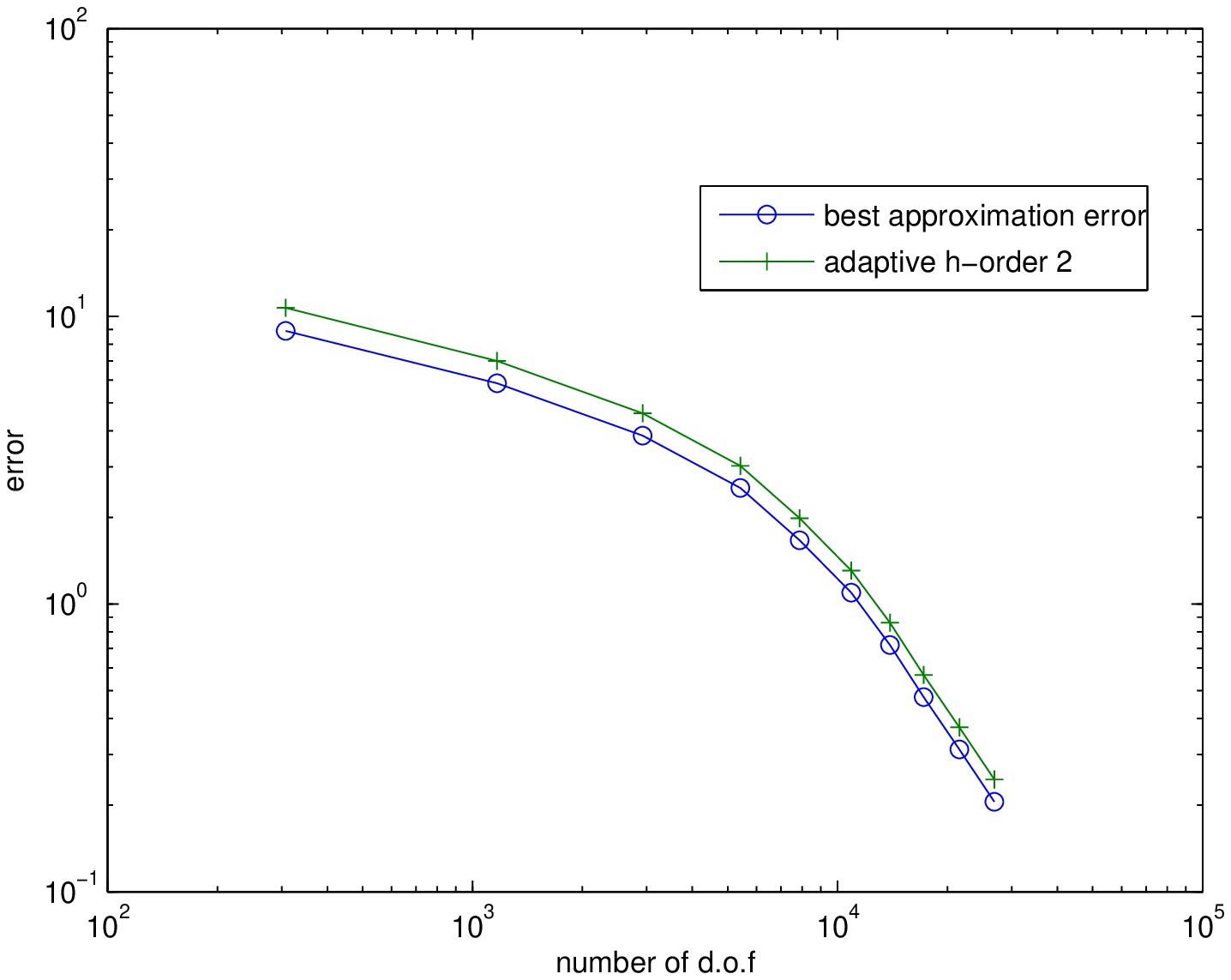}}
\subfigure[Adaptive $hp$-refinements.]
{\label{rateahp}\includegraphics[scale=0.375]{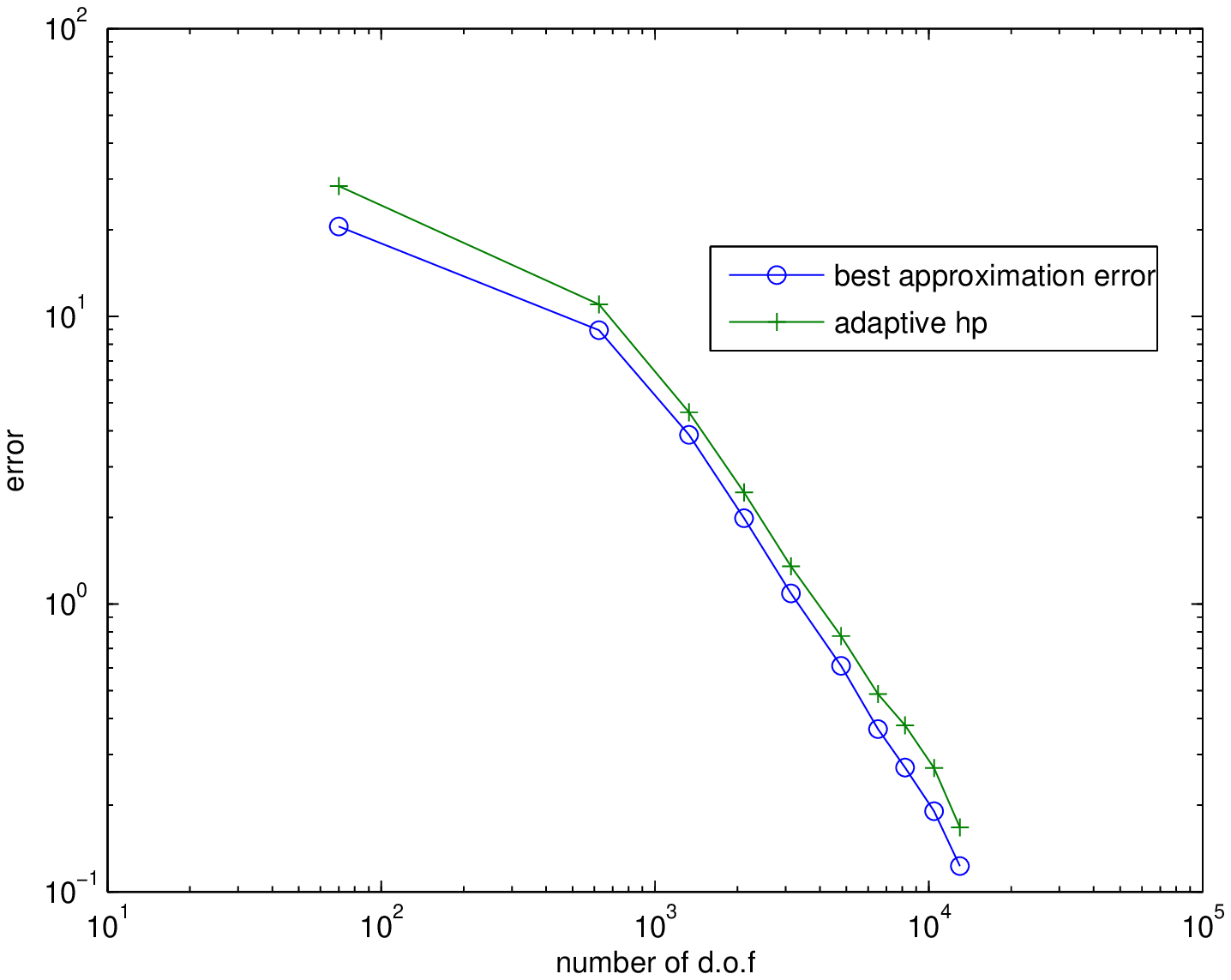}}
\end{center}
\caption{Comparison of FE error for adaptively refined affine meshes
with the best approximation error.}
\label{rates}
\end{figure}
Finally, in Fig.~\ref{comp2}, we compare the convergence history for
all tested refinements.
\begin{figure}[!ht]
\begin{center}
\includegraphics[scale=0.6]{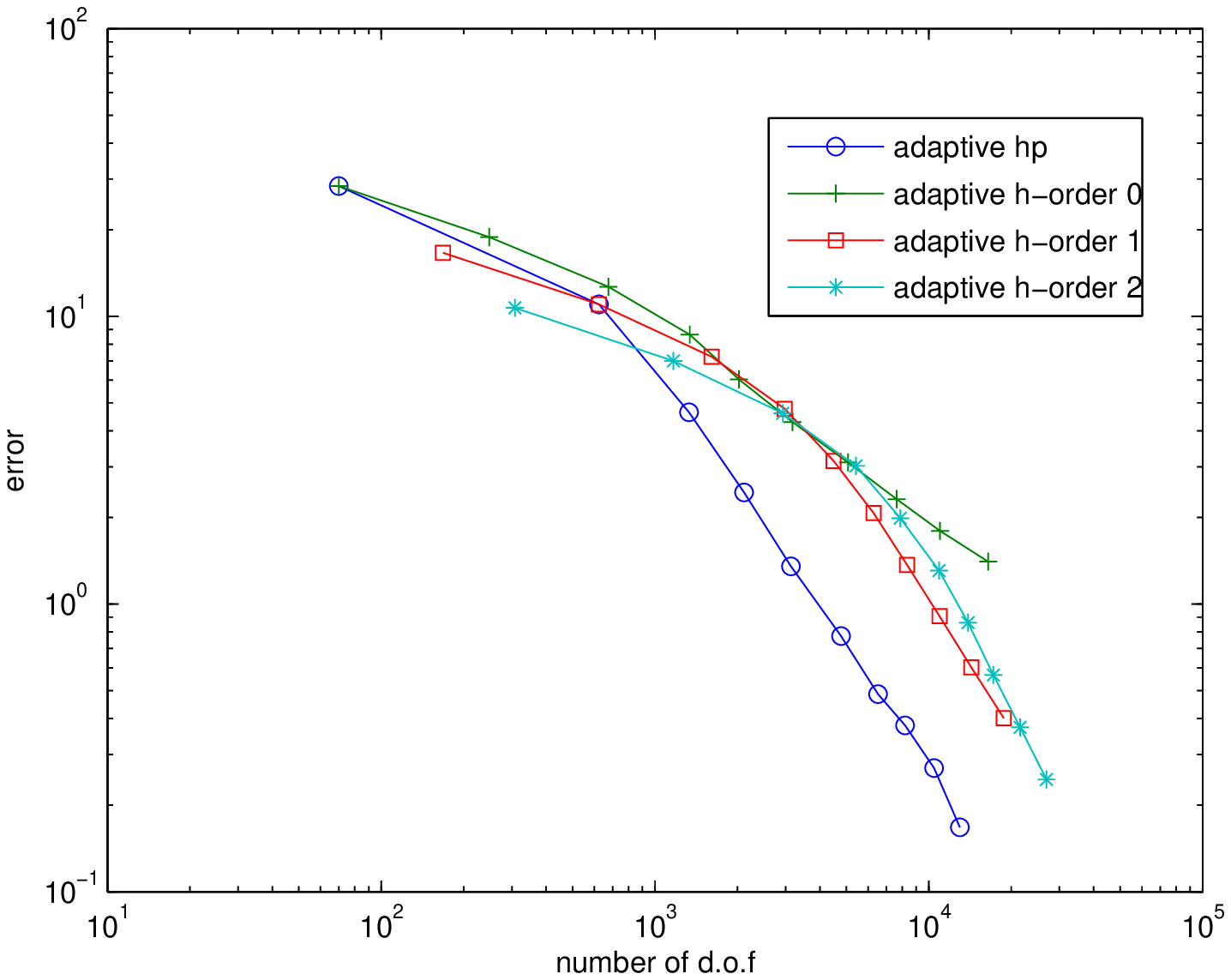}
\end{center}
\caption{Convergence history for adaptive $h$- and
$hp$-refinements on affine meshes.}
\label{comp2}
\end{figure}
The $hp$-adaptivity produces the best results although, in the presented
range, the rate seems to be still only algebraic.

\subsection*{Adaptive $h$- and $hp$-refinements on curvilinear meshes}
We use the same manufactured solution on the circular L-shape domain shown in  
Fig.~\ref{L_domain_curve}. 
\begin{figure}[ht]
\begin{center}
\includegraphics[scale=0.9]{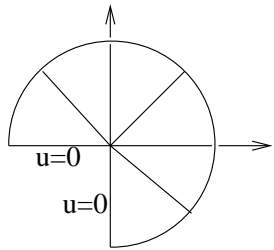}
\end{center}
\caption{A circular L-shape domain with an initial mesh of six elements.}
\label{L_domain_curve}
\end{figure}
The six triangles in the initial mesh are parametrized using the
transfinite interpolation technique, see \cite{Demkowicz:2006:HPAFE}, p. 201, for details.
The FE error is compared again with the best approximation error
for a sequence of $h$-adaptive meshes and $p=0,1,2$, in Fig.~\ref{crates}.
The results indicate again a slight loss of stability for the elements
of lowest order.
\begin{figure}[htp]
\begin{center}
\subfigure[Adaptive $h$-refinements, $p=0$.]
{\label{crate0}\includegraphics[scale=0.375]{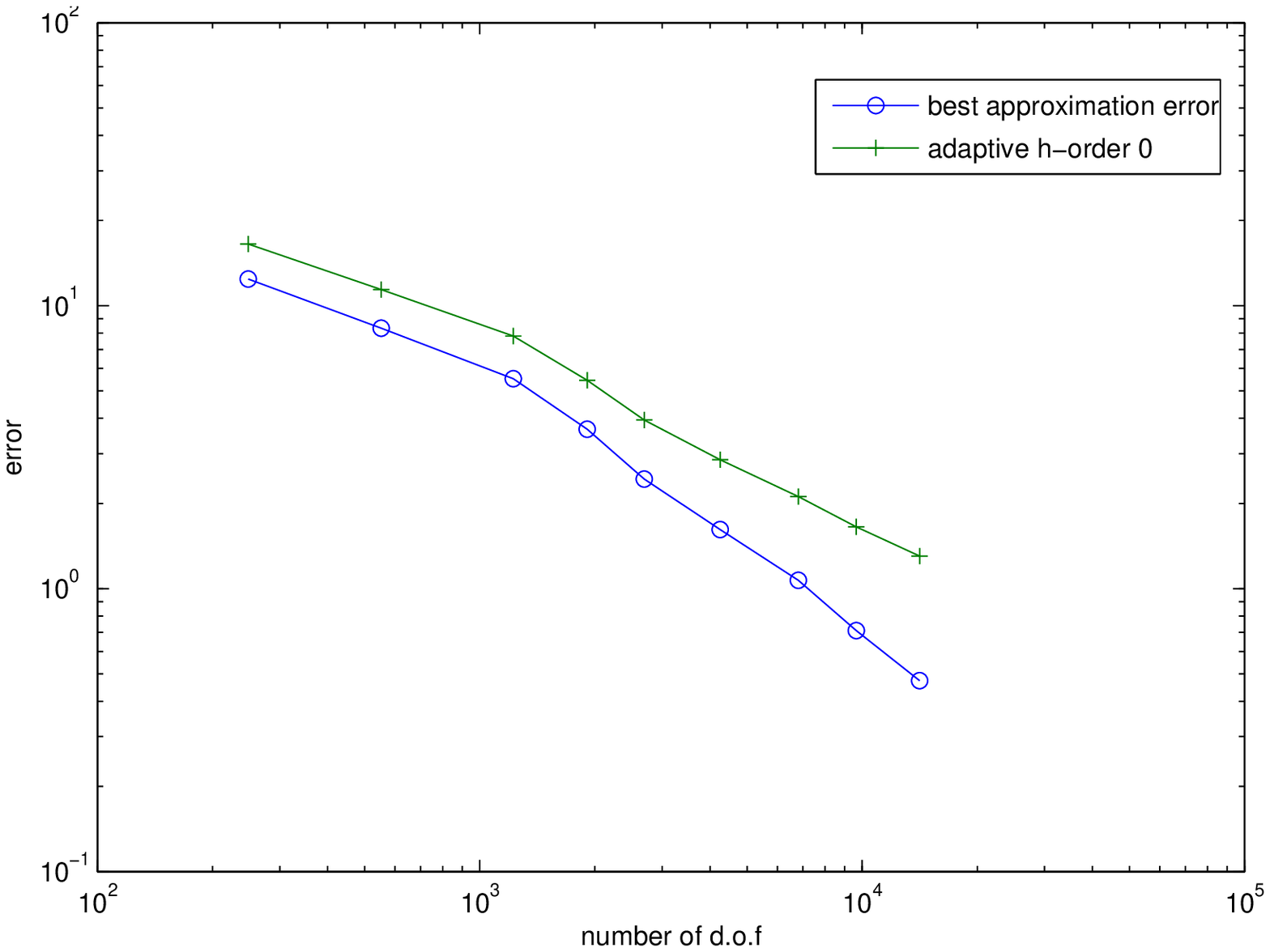}}
\subfigure[Adaptive $h$-refinements, $p=1$.]
{\label{crate1}\includegraphics[scale=0.375]{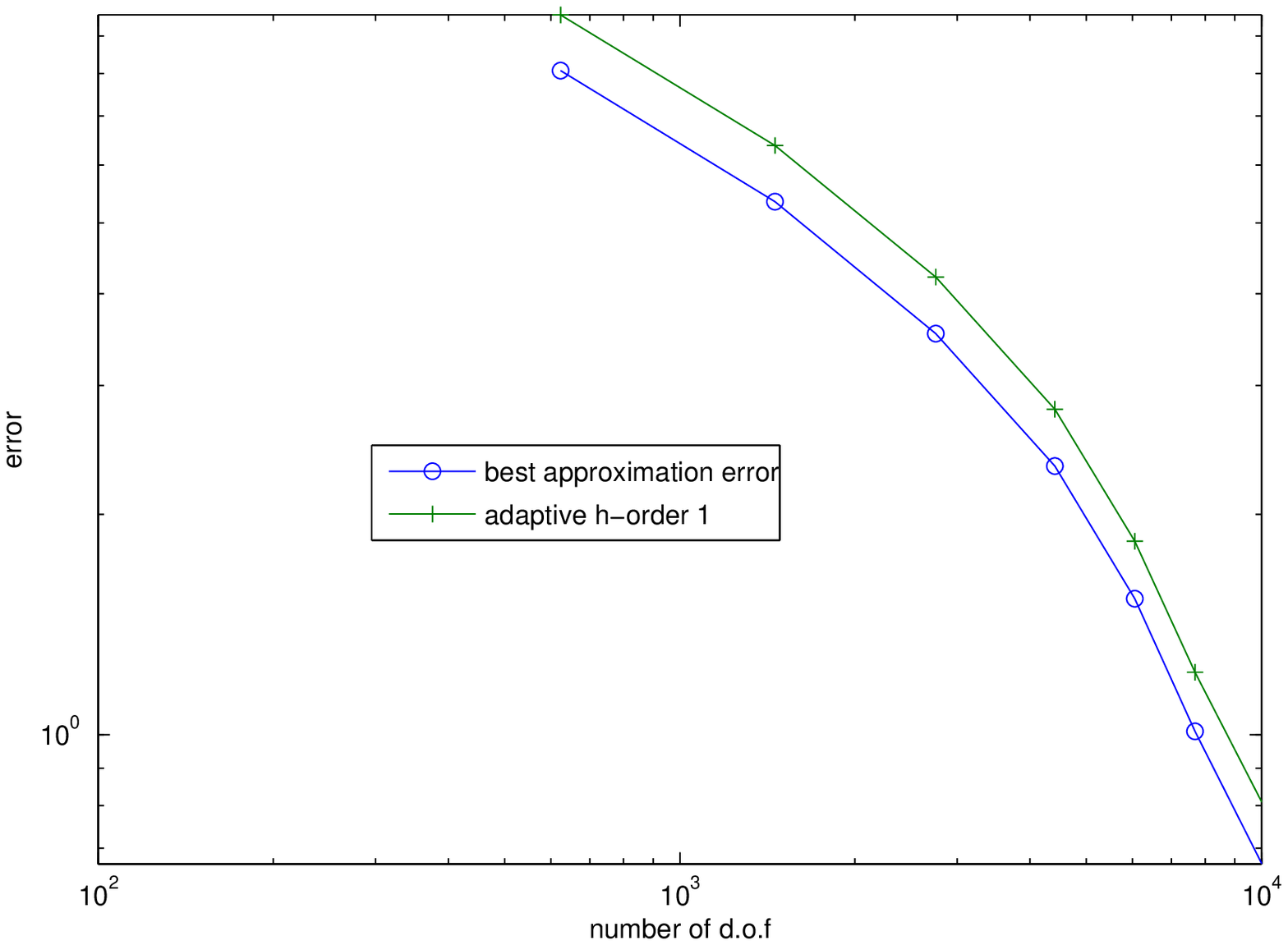}}
\subfigure[Adaptive $h$-refinements, $p=2$.]
{\label{crate2}\includegraphics[scale=0.375]{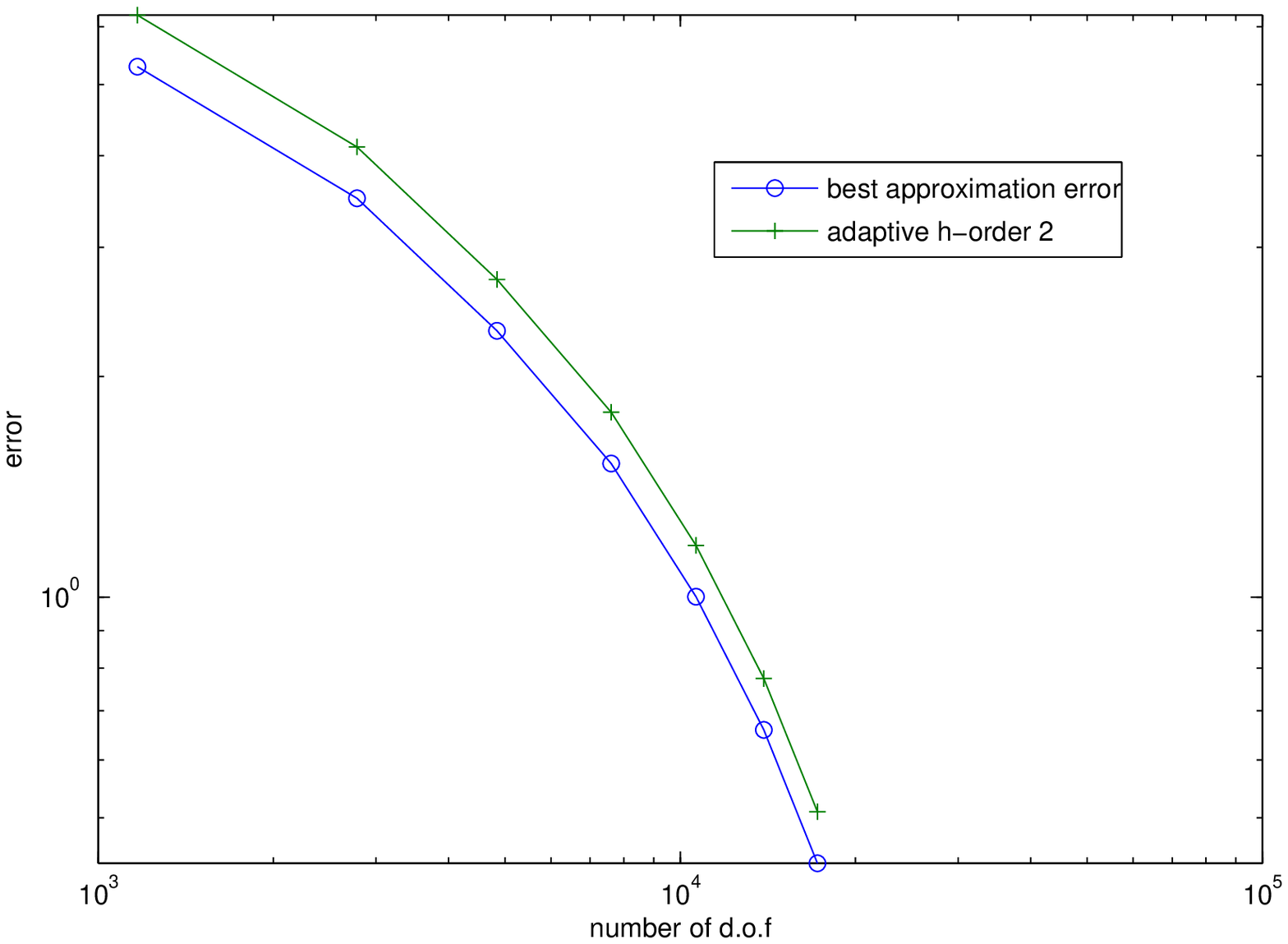}}
\subfigure[Adaptive $hp$-refinements.]
{\label{crateahp}\includegraphics[scale=0.375]{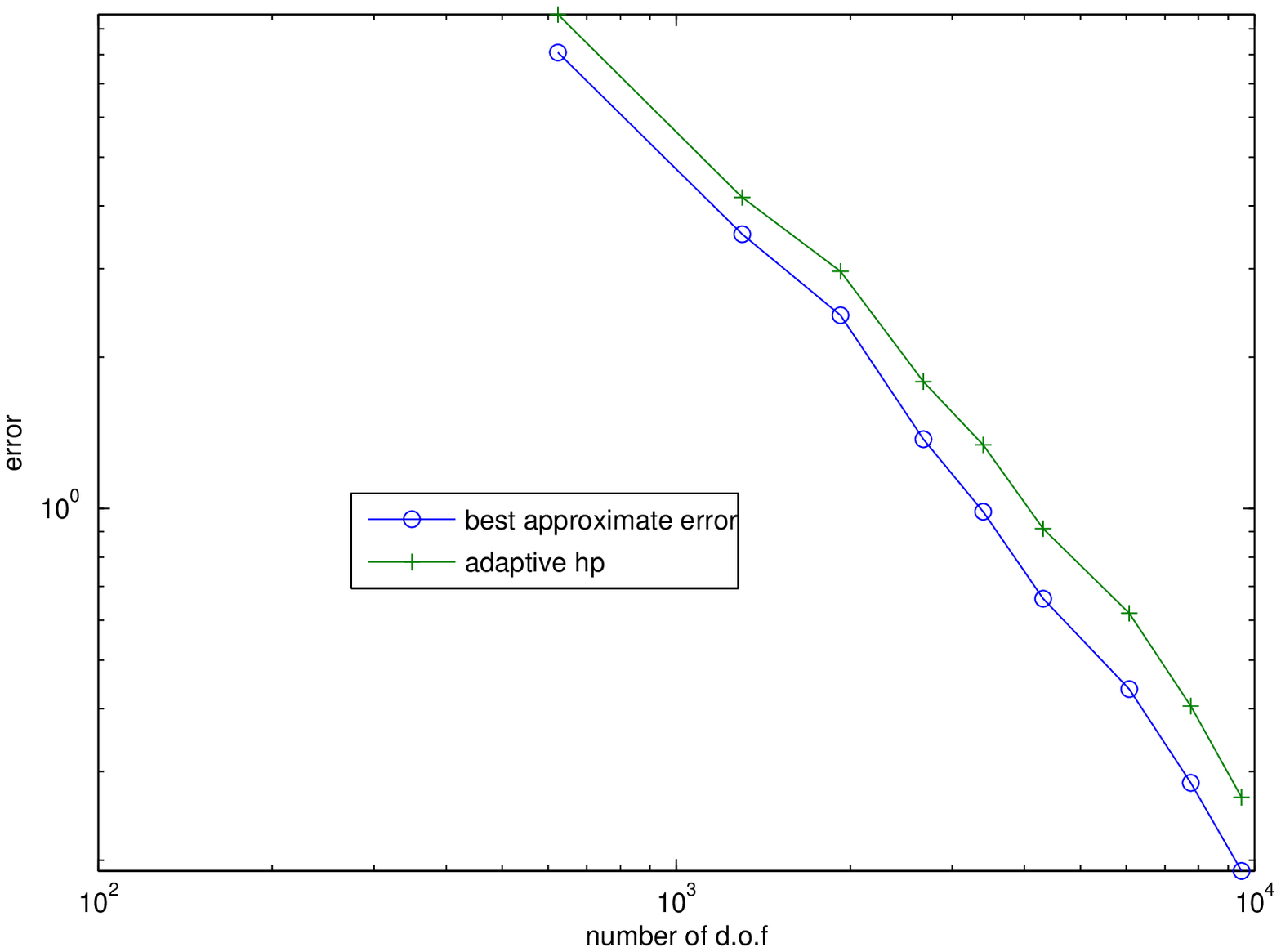}}
\end{center}
\caption{Comparison of FE error for adaptively refined curvilinear meshes
with the best approximation error.}
\label{crates}
\end{figure}
Finally, in Fig.~\ref{comp3}, we compare the convergence history for
all tested refinements.
The $hp$-adaptivity delivers again the best
results but the exponential convergence is questionable. This may indicate
that the use of standard Projection-Based interpolation operator 
in the reference domain is not optimal. 
\begin{figure}[ht]
\begin{center}
\includegraphics[scale=0.6]{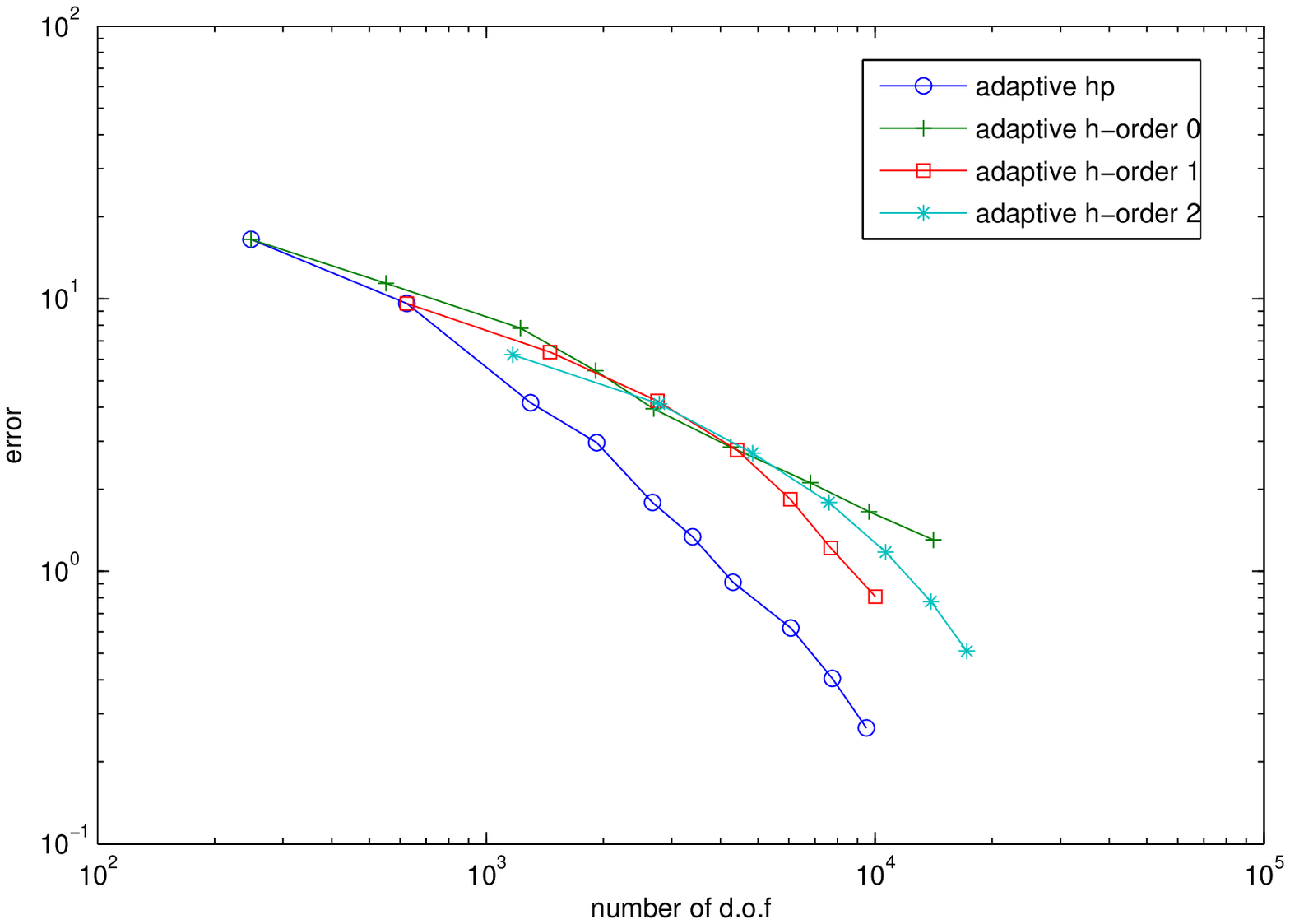}
\end{center}
\caption{Convergence history for adaptive $h$- and
$hp$-refinements on curvilinear meshes.}
\label{comp3}
\end{figure}

\section{Conclusions}

We have presented a complete $h$-stability analysis for a generalization
of Arnold-Falk-Winther elements to curvilinear meshes of variable order in two space dimensions. The stability analysis for both generalizations: variable order elements, 
and curvilinear elements proved to be rather non-trivial. 
The case of variable elements has been tackled with a novel logical construction
showing the existence of necessary interpolation operators rather than constructing
them explicitly. The presented construction departs from operators used
by Arnold, Falk and Winther and modifies the Projection-Based (PB) interpolation operators
as well.

The analysis of curvilinear
meshes for elasticity
differs considerably from that for problems involving only grad-curl-div operators. Piola 
maps transform gradients, curls and divergence in the physical domain into the
corresponding gradients, curls and divergence in the reference domain. Consequently,
problems involving the grad, curl or div operators only (e.g. Maxwell equations
or the mixed formulation for a scalar elliptic problem)
can be reformulated in the parametric domain at the expense of introducing material
anisotropies reflecting the geometric parametrizations. This is not the case for elasticity where
the strain tensor (symmetric part of the displacement gradient) in the physical
domain does not transform into the symmetric part of the displacement gradient in the
reference domain\footnote{The same problem is encountered in the case of complex
stretchings introduced by Perfectly Matched Layers, see e.g. \cite{Demkowicz:2007:HPAFE2}.}.
Consequently, the analysis for affine meshes cannot be simply reproduced for curvilinear
ones, and new interpolation operators have to be carefully drafted. 
We have managed to prove only an asymptotic stability for the curvilinear meshes.

Presented numerical results on $h$-adaptivity go beyond our analysis as we use 1-irregular
meshes with hanging nodes supported by an existing $hp$ software. 


Finally, the paper presents only a two-dimensional result. We continue working on the
3D case using different ideas and hope to present new results soon.

\appendix
\section{Mesh generation\label{app_mesh_generation}}

We assume that the domain $\Omega$ is a (curvilinear) polygon, 
and that it can be meshed
with a regular family $(\mathcal{T}_{h})_{h}$ of $C^{0}$-compatible meshes (i.e. 
$\overline{\Omega}=\bigcup_{T\in\mathcal{T}_{h}} T$, for all $h$) that satisfy
the regularity assumptions discussed in the previous section.

We will outline now shortly how one can generate such meshes in practice.
Suppose the domain $\Omega$ has been meshed with a $C^0$-compatible initial mesh
$\mathcal{T}_{\text{int}}$, 
$\overline{\Omega}=\bigcup_{i=1}^{m}T_{i}$, 
where $\{T_{i}\}_{i=1}^{m}$ are curved triangles of class $C^{1,1}$.
We denote by $\{G_{1},\cdots,G_{m}\}$ the mappings from $\hat{T}$ to 
$\{T_{1},\cdots,T_{m}\}$.  
Then $G_{i}\in C^{1,1}(\hat{T})$ and $G_{i}^{-1}\in C^{1}(T_{i})$ for any $1\leq i\leq m$.  
For examples of techniques to generate an initial mesh satisfying the assumptions above,
see \cite{Demkowicz:2006:HPAFE}.

\begin{lemma}
\label{decreasing_nonlinear_mapping}
Let $T = G_{T}(\hat{T})$ be a closed triangle in $\mathbb{R}^2$
with $G_{T}\in C^{1,1}(\hat{T})$. 
Let $\check{\sigma}>0$ be a positive constant. 
For any $\check{h}>0$, we denote by $\check{T}_{\check{h}}$ 
any triangle contained in $\hat{T}$ such that the diameter of $\check{T}_{\check{h}}$ is $\check{h}$, 
and $\check{h}/\check{\rho}\leq\check{\sigma}$ where $\check{\rho}$ is the diameter of the sphere inscribed in 
$\check{T}_{\check{h}}$. 
Let $\hat{T}\ni\hat{\mathbf{x}}\rightarrow H\hat{\mathbf{x}} =
\check{B}\hat{\mathbf{x}}+\check{b}$ 
be an affine mapping from $\hat{T}$ onto $\check{T}_{\check{h}}$.
Let $\hat{\mathbf{p}}$ be the centroid of $\hat{T}$. 
We put $B=D(G_{T}\circ H) (\hat{\mathbf{p}})$, $b=(G_{T}\circ H)(\hat{\mathbf{p}})$, 
and $\Psi(\hat{\mathbf{x}})=(G_{T}\circ H)(\hat{\mathbf{x}})-\check{B}(\hat{\mathbf{x}}
-\hat{\mathbf{p}})-b$. 

Then, we have
\begin{equation*}
(\sup_{\hat{\mathbf{x}}\in\hat{T}}\Vert D\Psi
(\hat{\mathbf{x}})\Vert) \Vert B^{-1}\Vert \leq C\check{h},
\end{equation*}
where $C$ is a constant independent of $\check{h}$, 
and $\check{T}_{\check{h}}$.
\end{lemma}

\begin{proof}
Obviously, $B=DG_{T}(H(\hat{\mathbf{p}}))\check{B}$. So 
$B^{-1}=\check{B}^{-1}(DG_{T}(H(\hat{\mathbf{p}})))^{-1}$.
We have 
\begin{equation*}
D\Psi(\hat{\mathbf{x}})=
(DG_{T}(H(\hat{\mathbf{x}}))-DG_{T}(H(\hat{\mathbf{p}})))\check{B}.
\end{equation*}
Since $G_{T}$ is a $C^{1}$-diffeomorphism from $\hat{T}$ onto $T$, we have
$\sup_{\hat{\mathbf{x}}\in\hat{T}}\Vert (DG_{T}(\hat{\mathbf{x}}))^{-1}\Vert<\infty$.
With $G_{T}\in C^{1,1}(\hat{T})$, we also have 
\begin{equation*}
\Vert DG_{T}(H(\hat{\mathbf{x}}))-DG_{T}(H(\hat{\mathbf{p}}))\Vert\leq
M \Vert H(\hat{\mathbf{x}})-H(\hat{\mathbf{p}})\Vert \quad \text{for any } 
\hat{\mathbf{x}}\in\hat{T}
\end{equation*} 
where $M$ is the Lipschitz constant for all first order derivatives of $G_{T}$ on $\hat{T}$.

With $\check{h}/\check{\rho}\leq\check{\sigma}$, we obtain that $\Vert \check{B} \Vert\cdot
\Vert \check{B}^{-1}\Vert\leq\dfrac{c}{\check{\sigma}}$, for some $c>0$.
With $\check{h}$ denoting the diameter of $\check{T}_{\check{h}}$, we have
$\Vert H(\hat{\mathbf{x}})-H(\hat{\mathbf{p}})\Vert\leq \check{h}$ for 
any $\hat{\mathbf{x}}\in\check{T}_{\check{h}}$. The definition of $\Psi$ implies then
\begin{equation*}
(\sup_{\hat{\mathbf{x}}\in\hat{T}}\Vert D\Psi
(\hat{\mathbf{x}})\Vert) \Vert B^{-1}\Vert \leq \dfrac{c}{\check{\sigma}}
M(\sup_{\hat{\mathbf{y}}\in\hat{T}}\Vert (DG_{T}(\hat{\mathbf{y}}))^{-1}\Vert)\check{h}.
\end{equation*}
Setting $C= \dfrac{c}{\check{\sigma}}M
(\sup_{\hat{\mathbf{y}}\in\hat{T}}\Vert (DG_{T}(\hat{\mathbf{y}}))^{-1}\Vert)$
finishes the proof.
\end{proof}

Let $\hat{T}_i$ be now multiple copies of the reference triangle corresponding
to the initial mesh, $T_i = G_i(\hat{T}_i),\: i=1,\ldots,m$. 
\begin{definition}
\label{regular_reference}
(Regular triangular meshes in the reference space)
A family of triangulations $(\hat{\mathcal{T}}_{i,\check{h}})_{\check{h}}$ 
of reference triangles $\hat{T}_i$ is said to be regular provided two conditions
are satisfied:
\begin{itemize}
  \item[(i)] Partitions of edges of $\hat{T}_i$ mapped into the same edge
in the physical space are identical.
  \item[(ii)] 
$\sup_{\check{h}}\sup_{\check{T}\in\hat{\mathcal{T}}_{i,\check{h}}}\check{h}/
\check{\rho}<\infty$, where $\check{h}$ and $\check{\rho}$ are the outer and inner diameters 
of $\check{T}$.
\end{itemize}
\end{definition}

Obviously, uniform refinements of reference triangles\footnote{With the same
number of divisions for each reference triangle.} are regular. A number of
adaptive refinement algorithms produces regular meshes as well. To this class
belong e.g. Rivara's algorithm (bisection by the longest edge),
Arnold's algorithm (bisection by the newest edge), the Delaunay triangulation 
(see \cite{Carey:1997:Grids}).

Using Lemma~\ref{decreasing_nonlinear_mapping} and the fact that 
$G_{i}$ is $C^{1}$-diffeomorphism  from $\hat{T}$ to $T_{i}, \: 1\leq i\leq m$, 
we easily conclude that any regular refinements in the reference space
produce curvilinear meshes that satisfy our mesh regularity assumptions.

\section{Properties of Sobolev spaces on curved and reference triangles}

\begin{lemma}
\label{mapping1}
Let $T$ be a curved triangle. For any $\omega\in H^{1}(T;\mathbb{R}^{2})$, 
we define $\hat{\omega}(\hat{\mathbf{x}})$ on $\hat{T}$ by 
$$
\omega(\mathbf{x})=\dfrac{DG_{T}(\hat{\mathbf{x}})}
{\det(DG_{T}(\hat{\mathbf{x}}))}\hat{\omega}(\hat{\mathbf{x}})
\quad \hat{\mathbf{x}}\in\hat{T}.
$$ 
Then $\hat{\omega}(\hat{\mathbf{x}})\in H(\hat{T})$. 
Divergence transforms by the classical Piola's rule:
$$
\text{div}\omega(\mathbf{x})=
(\det(DG_{T}(\hat{\mathbf{x}})))^{-1}\text{div}_{\hat{\mathbf{x}}}\hat{\omega}(\hat{\mathbf{x}})
$$
for $\hat{\mathbf{x}}\in\hat{T}$ almost everywhere.
\end{lemma}

\begin{proof}
Notice that $\hat{\omega}(\hat{\mathbf{x}})=\det(DG_{T}(\hat{\mathbf{x}}))(DG_{T}(\hat{\mathbf{x}}))^{-1}
\omega(\mathbf{x})$ for any $\hat{\mathbf{x}}\in\hat{T}$. It is straightforward to see that $\det(DG_{T}(\hat{\mathbf{x}}))(DG_{T}(\hat{\mathbf{x}}))^{-1}$ is a matrix whose entries contain first order partial derivatives of $G_{T}(\hat{x})$. 
Notice that 
$$\hat{\omega}(\hat{\mathbf{x}})=\det(DG_{T}(\hat{\mathbf{x}}))(DG_{T}(\hat{\mathbf{x}}))^{-1}
\omega(\mathbf{x})$$ 
is the standard pull back mapping from $H(\text{div},T)$ to $H(\text{div}_{\hat{\mathbf{x}}},\hat{T})$. 
So we immediately have $\text{div}\omega(\mathbf{x})=\dfrac{1}{\det(DG_{T}(\hat{\mathbf{x}}))}\text{div}_{\hat{\mathbf{x}}}\hat{\omega}(\hat{\mathbf{x}})$ 
 for $\hat{\mathbf{x}}\in\hat{T}$ almost everywhere.
Since $G_{T}$ is a $C^{1}$-diffeomorphism from $\hat{T}$ to $T$, we can conclude that 
$\hat{\omega}(\hat{\mathbf{x}})\in H(\text{div},\hat{T})$.

Let $e\in\triangle_{1}(T)$. We denote by $\zeta(s):[0,1]\rightarrow\hat{e}$, 
the local affine parameterization 
of $\hat{e}$. We have then
\begin{align*}
\Vert\hat{\omega}\Vert_{L^{2}(\hat{e})}^{2} = &
\int_{[0,1]}(\hat{\omega}(\zeta(s)))^{\top}\hat{\omega}(\zeta(s))\Vert\dot{\zeta}(s)\Vert ds \\
= & \int_{[0,1]}(\det(DG_{T}(\zeta(s))))^{2}\omega(G_{T}(\zeta(s)))^{\top}
[DG_{T}(\zeta(s))^{-\top}DG_{T}(\zeta(s))^{-1}] \\
& \qquad \omega(G_{T}(\zeta(s)))\Vert (DG_{T}(\zeta(s)))^{-1}[DG_{T}(\zeta(s))\dot{\zeta}(s)]\Vert ds. 
\end{align*}
Since $G_{T}$ is a $C^{1}$-diffeomorphism from $\hat{T}$ to $T$, we can conclude that 
$\hat{\omega}|_{\partial\hat{T}}\in L^{2}(\partial\hat{T};\mathbb{R}^{2})$. So $\hat{\omega}\in
H(\hat{T})$.
\end{proof}

\begin{lemma}
\label{H_div_transform_inequality1}
There exist $\delta>0$ and $C>0$ such that, for any $h<\delta$ and $T\in\mathcal{T}_{h}$,
\begin{equation*}
\Vert\omega\Vert_{L^{2}(T)}\leq C\Vert\hat{\omega}\Vert_{L^{2}(\hat{T})},\forall \omega\in L^{2}(T;\mathbb{R}^{2}),
\end{equation*}
where $\dfrac{DG_{T}(\hat{\mathbf{x}})\hat{\omega}(\hat{\mathbf{x}})}{\det(DG_{T}(\hat{\mathbf{x}}))}=
\omega(\mathbf{x})$ for any $\hat{\mathbf{x}}\in\hat{T}$.
If $T$ is a triangle for any $T\in\mathcal{T}_{h}$ and $h>0$, then the above inequality holds for any $h>0$.
\end{lemma}

\begin{proof}
\begin{align*}
\Vert\omega\Vert_{L^{2}(T)} = &\int_{T}\dfrac{1}{\det(DG_{T}(\hat{\mathbf{x}}))^{2}}[DG_{T}(\hat{\mathbf{x}})
\hat{\omega}(\hat{\mathbf{x}})]^{\top}DG_{T}(\hat{\mathbf{x}})\hat{\omega}(\hat{\mathbf{x}})d\mathbf{x}\\
= & \int_{\hat{T}}\det((DG_{T})^{-1})\hat{\omega}^{\top}[DG_{T}^{\top}
DG_{T}]\hat{\omega}d\hat{\mathbf{x}} \\
= & \int_{\hat{T}}\det(DG_{T}^{-1}B_{T})\det(B_{T}^{-1})\hat{\omega}^{\top}B_{T}^{\top}
[B_{T}^{-\top}DG_{T}^{\top}DG_{T}B_{T}^{-1}]B_{T}\hat{\omega}d\hat{\mathbf{x}}.
\end{align*}

Since $c_{h}\rightarrow 0$ as $h\rightarrow 0$, $\lim_{h\rightarrow 0}\sup_{T\in\mathcal{T}_{h}}
\sup_{\hat{\mathbf{x}}\in\hat{T}}\Vert B_{T}^{-\top}DG_{T}^{\top}DG_{T}B_{T}^{-1}-I\Vert=0$. 
By Lemma~\ref{geometry_property3}, we have $\lim_{h\rightarrow 0}\sup_{T\in\mathcal{T}_{h}}
\sup_{\hat{\mathbf{x}}\in\hat{T}}\vert\det(DG_{T}^{-1}B_{T})-1\vert = 0$.

Since $(\mathcal{T}_{h})_{h}$ is regular, there is a constant $\sigma>0$ 
such that $\sigma_{1}/\sigma_{2}\leq\sigma$ 
for any $T\in\mathcal{T}_{h}$, where $\sigma_1$ and $\sigma_2$ 
denote the biggest and smallest singular value of 
the corresponding matrix $B_{T}$. 
Then $\Vert B_{T}^{\top}\Vert = \Vert B_{T}\Vert = \sigma_1$ 
and $\det(B_{T})=\sigma_1\cdot\sigma_2$. 
So $\Vert B_{T}^{\top}\Vert \cdot \Vert B_{T}\Vert \det (B_{T}^{-1})=\sigma_1/\sigma_2\leq\sigma$ for 
any $h>0$ and any $T\in\mathcal{T}_{h}$.

We can conclude thus that there exist $\delta>0$ and $C>0$ such that,
for any $h<\delta$ and $T\in\mathcal{T}_{h}$,
\begin{equation*}
\Vert\omega\Vert_{L^{2}(T)}\leq C\Vert\hat{\omega}\Vert_{L^{2}(\hat{T})}
\quad \forall \omega\in L^{2}(T;\mathbb{R}^{2})
\end{equation*}
If, for all $h$, $T\in\mathcal{T}_{h}$ are (regular) triangles, 
the asymptotic argument is not necessary, and the above inequality holds for any $h>0$.
\end{proof}

\begin{lemma}
\label{H_div_transform_inequality2}
There exist $\delta>0$ and $C>0$ such that, for any $h<\delta$ and $T\in\mathcal{T}_{h}$,
\begin{equation*}
\Vert\hat{\omega}\Vert_{H(\text{div}_{\hat{\mathbf{x}}},\hat{T})}^{2}
+\Vert\hat{\omega}\Vert_{L^{2}(\partial\hat{T})}^{2}
\leq C\Vert\omega\Vert_{H^{1}(T)}^{2}
\quad\forall \omega\in H^{1}(T;\mathbb{R}^{2}),
\end{equation*}
where $\dfrac{DG_{T}(\hat{\mathbf{x}})\hat{\omega}(\hat{\mathbf{x}})}{\det(DG_{T}(\hat{\mathbf{x}}))}=
\omega(\mathbf{x})$ for any $\hat{\mathbf{x}}\in\hat{T}$.
If $T$ is a triangle for any $T\in\mathcal{T}_{h}$ and $h>0$, then the above inequality holds for any $h>0$.
\end{lemma}

\begin{proof}
\begin{align*}
\Vert\hat{\omega}\Vert_{L^{2}(\hat{T})}^{2} & = 
\int_{\hat{T}}(\det(DG_{T}(\hat{\mathbf{x}})))^{2}\omega(\mathbf{x})^{\top}
[DG_{T}(\hat{\mathbf{x}})^{-T}DG_{T}(\hat{\mathbf{x}})^{-1}]\omega(\mathbf{x})d\hat{\mathbf{x}}\\
& = \int_{T}\det(DG_{T}(\hat{\mathbf{x}}))\omega(\mathbf{x})^{\top}
[DG_{T}(\hat{\mathbf{x}})^{-T}DG_{T}(\hat{\mathbf{x}})^{-1}]\omega(\mathbf{x})d\mathbf{x}\\
& = \int_{T}\det(DG_{T}(\hat{\mathbf{x}})B_{T}^{-1})\det(B_{T})\omega(\mathbf{x})^{\top}B_{T}^{-\top} \\
& \qquad \quad [B_{T}^{\top}DG_{T}(\hat{\mathbf{x}})^{-T}DG_{T}(\hat{\mathbf{x}})^{-1}B_{T}]B_{T}^{-1}\omega(\mathbf{x})d\mathbf{x}.
\end{align*}
According to Lemma~\ref{geometry_property2}, $\lim_{h\rightarrow 0}\sup_{T\in\mathcal{T}_{h}}
\sup_{\hat{\mathbf{x}}\in\hat{T}}\Vert B_{T}^{\top}DG_{T}(\hat{\mathbf{x}})^{-T}
DG_{T}(\hat{\mathbf{x}})^{-1}B_{T}-I\Vert=0$. 
By Lemma~\ref{geometry_property3}, we have $\lim_{h\rightarrow 0}\sup_{T\in\mathcal{T}_{h}}
\sup_{\hat{\mathbf{x}}\in\hat{T}}\vert\det(DG_{T}(\hat{\mathbf{x}})B_{T}^{-1})-1\vert=0$.

Since $(\mathcal{T}_{h})_{h}$ is regular, there exists a constant $\sigma>0$ 
such that  $\sigma_{1}/\sigma_{2}\leq\sigma$ 
for any $T\in\mathcal{T}_{h}$, where $\sigma_1$ and $\sigma_2$ denote
 the biggest and smallest singular value of 
the matrix $B_{T}$. 
Then $\Vert B_{T}^{-\top}\Vert=\Vert B_{T}^{-1}\Vert = \sigma_{2}^{-1}$ and 
$\det(B_{T})=\sigma_1\cdot\sigma_{2}$. So $\Vert B_{T}^{-\top}\Vert\cdot\Vert B_{T}^{-1}\Vert\det(B_{T})
=\sigma_1/\sigma_2\leq \sigma$.
Consequently, there exist $\delta_{1}>0$ and $C_{1}>0$ such that,
 for any $h<\delta_{1}$ and $T\in\mathcal{T}_{h}$,
\begin{align*}
& \det(DG_{T}(\hat{\mathbf{x}})B_{T}^{-1})\det(B_{T})\omega(\mathbf{x})^{\top}B_{T}^{-\top}
[B_{T}^{\top}DG_{T}(\hat{\mathbf{x}})^{-T}DG_{T}(\hat{\mathbf{x}})^{-1}B_{T}]B_{T}^{-1}\omega(\mathbf{x})\\ 
& \leq C_{1}^{2} \omega(\mathbf{x})^{\top}\omega(\mathbf{x})
\end{align*}
for all $\omega\in H^{1}(T;\mathbb{R}^{2}), \mathbf{x}\in T$.
We can conclude that, for any $h<\delta_{1}$ and $T\in\mathcal{T}_{h}$,
\begin{equation*}
\Vert\hat{\omega}\Vert_{L^{2}(\hat{T})}\leq C_{1}\Vert\omega\Vert_{L^{2}(T)}
\quad \forall \omega\in H^{1}(T;\mathbb{R}^{2})
\end{equation*}
According to Lemma~\ref{mapping1}, we have
\begin{align*}
 \Vert \text{div}_{\hat{\mathbf{x}}}\hat{\omega}\Vert_{L^{2}(\hat{T})}^{2} = &
\int_{\hat{T}}(\det(DG_{T}(\hat{\mathbf{x}})))^{2}(\text{div}\omega(\mathbf{x}))^{2}d\hat{\mathbf{x}}
= \int_{T}\det(DG_{T}(\hat{\mathbf{x}}))(\text{div}\omega(\mathbf{x}))^{2}d\mathbf{x} \\
= & \int_{T}\det(DG_{T}(\hat{\mathbf{x}})B_{T}^{-1})\det(B_{T})(\text{div}\omega(\mathbf{x}))^{2}d\mathbf{x}.
\end{align*}
By Lemma~\ref{geometry_property3}, $\lim_{h\rightarrow 0}\sup_{T\in\mathcal{T}_{h}}
\sup_{\hat{\mathbf{x}}\in\hat{T}}\vert\det(DG_{T}(\hat{\mathbf{x}})B_{T}^{-1})-1\vert=0$.
Obviously, $\det(B_{T})\leq \tilde{h}_{T}^{2}$.
There must exist then $\delta_{2}>0$ and $C_{2}>0$ such that,
 for any $h<\delta_{2}$ and $T\in\mathcal{T}_{h}$,
\begin{equation*}
\det(DG_{T}(\hat{\mathbf{x}})B_{T}^{-1})\det(B_{T})(\text{div}\omega(\mathbf{x}))^{2} \leq 
C_{2}^{2}\tilde{h}_{T}^{2}(\text{div}\omega(\mathbf{x}))^{2}
\quad \forall \omega\in H^{1}(T;\mathbb{R}^{2}), \mathbf{x}\in T.
\end{equation*}
We conclude that, for any $h<\delta_{2}$ and $T\in\mathcal{T}_{h}$,
\begin{equation*}
\Vert \text{div}_{\hat{\mathbf{x}}}\hat{\omega}\Vert_{L^{2}(\hat{T})}\leq C_{2}\tilde{h}_{T}\Vert\text{div}\omega\Vert_{L^{2}(T)}
\quad
\forall \omega\in H^{1}(T;\mathbb{R}^{2}).
\end{equation*}
We take now an arbitrary $e\in\triangle_{1}(T)$. We denote by 
$\zeta(s):[0,1]\rightarrow\hat{e}$ the local affine parameterization 
of $\hat{e}$. We have then
\begin{align*}
\Vert\hat{\omega}\Vert_{L^{2}(\hat{e})}^{2} = &
\int_{[0,1]}(\hat{\omega}(\zeta(s)))^{\top}\hat{\omega}(\zeta(s))\Vert\dot{\zeta}(s)\Vert ds \\
= & \int_{[0,1]}(\det(DG_{T}(\zeta(s))))^{2}\omega(G_{T}(\zeta(s)))^{\top}
[DG_{T}(\zeta(s))^{-\top}DG_{T}(\zeta(s))^{-1}]\\
& \qquad \quad \omega(G_{T}(\zeta(s)))\Vert\dot{\zeta}(s)\Vert ds \\
= & \int_{[0,1]}(\det(DG_{T}(\zeta(s)))B_{T}^{-1})^{2}(\det(B_{T}))^{2}\omega(G_{T}(\zeta(s)))^{\top}B_{T}^{-\top} \\
& \qquad \quad [B_{T}^{\top}DG_{T}(\zeta(s))^{-\top}DG_{T}(\zeta(s))^{-1}B_{T}]B_{T}^{-1}\omega(G_{T}(\zeta(s)))
\Vert\dot{\zeta}(s)\Vert ds.
\end{align*}
According to Lemma~\ref{geometry_property2}, we have that 
$$
\lim_{h\rightarrow 0}\sup_{T\in\mathcal{T}_{h}}
\sup_{e\in T}\sup_{s\in [0,1]}\Vert B_{T}^{\top}DG_{T}(\zeta(s))^{-\top}DG_{T}(\zeta(s))^{-1}B_{T}-I\Vert=0.
$$ 
By Lemma~\ref{geometry_property3}, 
$\lim_{h\rightarrow 0}\sup_{T\in\mathcal{T}_{h}}\sup_{e\in T}\sup_{s\in [0,1]}
\vert(\det(DG_{T}(\zeta(s)))B_{T}^{-1})^{2}-1\vert=0$.

Consider again the singular values of $B_{T}$, $\sigma_1 \geq \sigma_2$. 
Then $\Vert B_{T}^{-\top}\Vert=\Vert B_{T}^{-1}\Vert = \sigma_{2}^{-1}$, and 
$\det(B_{T})=\sigma_1\cdot\sigma_{2}$. 
So $\Vert B_{T}^{-\top}\Vert\cdot\Vert B_{T}^{-1}\Vert(\det(B_{T}))^{2}
=(\sigma_1)^{2}\leq \tilde{h}_{T}^{2}$. 
Consequently, there exist $\delta_{3}>0$ and $C_{3}>0$ such that,
 for any $h<\delta_{3}$ and $T\in\mathcal{T}_{h}$,
\begin{align*}
& 
(\det(DG_{T}(\zeta(s)))B_{T}^{-1})^{2}(\det(B_{T}))^{2}\omega(G_{T}(\zeta(s)))^{\top}B_{T}^{-\top} \\
& \quad
[B_{T}^{\top}DG_{T}(\zeta(s))^{-\top}DG_{T}(\zeta(s))^{-1}B_{T}]B_{T}^{-1}\omega(G_{T}(\zeta(s))) \\
\leq & C_{3} \tilde{h}_{T}^{2}\omega(\mathbf{x})^{\top}\omega(\mathbf{x})\quad  
\forall \omega\in H^{1}(T;\mathbb{R}^{2}), e\in\triangle_{1}(T), s\in [0,1]
\end{align*}
and
\begin{align*}
\Vert\dot{\zeta}(s)\Vert = & \Vert B_{T}^{-1}[B_{T}DG_{T}(\zeta(s))^{-1}]\cdot [DG_{T}(\zeta(s))\dot{\zeta}(s)] \Vert \\
\leq & C_{3} \tilde{h}_{T}^{-1}\Vert DG_{T}(\zeta(s))\dot{\zeta}(s) \Vert
\quad \forall e\in\triangle_{1}(T), s\in [0,1]. 
\end{align*}
We can conclude that, for any $h<\delta_{3}$ and $T\in\mathcal{T}_{h}$,
\begin{align*}
\Vert\hat{\omega}\Vert_{L^{2}(\hat{e})}^{2}& \leq C_{3}^{2}\tilde{h}_{T}
\int_{[0,1]}\omega(G_{T}(\zeta(s)))^{\top}\omega(G_{T}(\zeta(s)))\Vert DG_{T}(\zeta(s))\dot{\zeta}(s)\Vert ds\\
& =C_{3}^{2}\tilde{h}_{T}\Vert\omega\Vert_{L^{2}(e)}^{2}.
\end{align*}
Obviously, $\sup_{h}\sup_{T\in\mathcal{T}_{h}}\tilde{h}_{T}<\infty$.
Since $\omega\in H^{1}(T;\mathbb{R}^{2})$, we can use the Trace Theorem to conclude that
there exist $\delta>0$ and $C>0$ such that, for any $h<\delta$ and $T\in\mathcal{T}_{h}$,
\begin{equation*}
\Vert\hat{\omega}\Vert_{H(\text{div}_{\hat{\mathbf{x}}},\hat{T})}^{2}
+\Vert\hat{\omega}\Vert_{L^{2}(\partial\hat{T})}^{2}
\leq C\Vert\omega\Vert_{H^{1}(T)}^{2}\quad \forall \omega\in H^{1}(T;\mathbb{R}^{2}).
\end{equation*}
It is easy to see that, if $T$ is a (regular) triangle for any $T\in\mathcal{T}_{h}$ 
and $h>0$, then the 
inequality above holds for any $h>0$.
\end{proof}

\begin{lemma}
\label{H_curl_transform_inequality1}
There exist $\delta>0$ and $C>0$ such that, for any $h<\delta$ and $T\in\mathcal{T}_{h}$,
\begin{equation*}
\Vert\text{curl}\omega\Vert_{L^{2}(T)}^{2} \leq C\tilde{h}_{T}^{-2}
\Vert\text{curl}_{\hat{\mathbf{x}}}\hat{\omega}\Vert_{L^{2}(\hat{T})}^{2}
\quad
\forall \omega\in H^{1}(T;\mathbb{R}^{2})
\end{equation*}
where $\dfrac{B_{T}\hat{\omega}(\hat{\mathbf{x}})}{\det(B_{T})}=
\omega(\mathbf{x})$ for any $\hat{\mathbf{x}}\in\hat{T}$.
If $T$ is a triangle for any $T\in\mathcal{T}_{h}$ and $h>0$, then the 
above inequality holds for any $h>0$.
\end{lemma}

\begin{proof}
We have,
\begin{align*}
\text{curl}\omega(\mathbf{x}) = & \dfrac{B_{T}}{\det(B_{T})}(\text{curl}\hat{\omega})(\hat{\mathbf{x}}) 
= \dfrac{1}{\det(B_{T}DG_{T}(\hat{\mathbf{x}}))}B_{T}\text{curl}_{\hat{\mathbf{x}}}\hat{\omega}(\hat{\mathbf{x}})
(DG_{T}(\hat{\mathbf{x}}))^{\top}\\
= & \dfrac{(\det(B_{T}^{-1}))^{2}}{\det(B_{T}^{-1}DG_{T}(\hat{\mathbf{x}}))}B_{T}\text{curl}_{\hat{\mathbf{x}}}
\hat{\omega}(\hat{\mathbf{x}})B_{T}^{\top}(DG_{T}(\hat{\mathbf{x}})B_{T}^{-1})^{\top}.
\end{align*}
and
\begin{align*}
\Vert\text{curl}\omega\Vert_{L^{2}(T)}^{2} = &
\int_{T}\dfrac{(\det(B_{T}^{-1}))^{4}}{(\det(B_{T}^{-1}DG_{T}(\hat{\mathbf{x}})))^{2}}(DG_{T}(\hat{\mathbf{x}})
B_{T}^{-1})B_{T}(\text{curl}_{\hat{\mathbf{x}}}\hat{\omega}(\hat{\mathbf{x}}))^{\top} \\
&\qquad [B_{T}^{\top}B_{T}]\text{curl}_{\hat{\mathbf{x}}}
\hat{\omega}(\hat{\mathbf{x}})B_{T}^{\top}(DG_{T}(\hat{\mathbf{x}})B_{T}^{-1})^{\top}d\mathbf{x}\\
= & \int_{\hat{T}}\dfrac{(\det(B_{T}^{-1}))^{3}}{(\det(B_{T}^{-1}DG_{T}(\hat{\mathbf{x}})))^{2}}(DG_{T}(\hat{\mathbf{x}})
B_{T}^{-1})B_{T}(\text{curl}_{\hat{\mathbf{x}}}\hat{\omega}(\hat{\mathbf{x}}))^{\top} \\
&\qquad [B_{T}^{\top}B_{T}]\text{curl}_{\hat{\mathbf{x}}}
\hat{\omega}(\hat{\mathbf{x}})B_{T}^{\top}(DG_{T}(\hat{\mathbf{x}})B_{T}^{-1})^{\top}d\hat{\mathbf{x}}.
\end{align*}
Since $c_{h}\rightarrow 0$ as $h\rightarrow 0$, $\lim_{h\rightarrow 0}\sup_{T\in\mathcal{T}_{h}}
\sup_{\hat{\mathbf{x}}\in\hat{T}}\Vert DG_{T}(\hat{\mathbf{x}})B_{T}^{-1}-I \Vert =0$.

\noindent By Lemma~\ref{geometry_property3}, 
$\lim_{h\rightarrow 0}\sup_{T\in\mathcal{T}_{h}}\sup_{\hat{\mathbf{x}}\in\hat{T}}
\vert(\det(B_{T}^{-1}DG_{T}(\hat{\mathbf{x}})))^{2}-1\vert=0$.

Since $(\mathcal{T}_{h})_{h}$ is regular, there exists a constant $\sigma>0$ 
such that  $\sigma_{1}/\sigma_{2}\leq\sigma$ 
for any $T\in\mathcal{T}_{h}$, with $\sigma_1 \geq \sigma_2$ denoting the singular values of 
matrix $B_{T}$. Then $\Vert B_{T}^{\top}\Vert=\Vert B_{T}\Vert = \sigma_{1}$ and 
$\det(B_{T}^{-1})=\sigma_1^{-1}\cdot\sigma_{2}^{-1}$. So $\Vert B_{T}^{\top}\Vert^{2}\cdot\Vert B_{T}\Vert^{2}
(\det(B_{T}^{-1}))^{3}
=\sigma_1/\sigma_2^{3}\leq \sigma/\sigma_{2}^{2}\leq c^{\prime}\tilde{h}_{T}^{-2}$ for some constant $c^{\prime}>0$.

Consequently, there exist $\delta_{1}>0$ and $C>0$ such that, for any $h<\delta_{1}$ and $T\in\mathcal{T}_{h}$,
\begin{align*}
& \int_{\hat{T}}\dfrac{(\det(B_{T}^{-1}))^{3}}{(\det(B_{T}^{-1}DG_{T}(\hat{\mathbf{x}})))^{2}}
(DG_{T}(\hat{\mathbf{x}})B_{T}^{-1})B_{T}(\text{curl}_{\hat{\mathbf{x}}}\hat{\omega}(\hat{\mathbf{x}}))^{\top} \\
&\qquad [B_{T}^{\top}B_{T}]\text{curl}_{\hat{\mathbf{x}}}\hat{\omega}(\hat{\mathbf{x}})B_{T}^{\top}(DG_{T}(\hat{\mathbf{x}})B_{T}^{-1})^{\top}d\hat{\mathbf{x}} \\
\leq & C\tilde{h}_{T}^{-2}\int_{\hat{T}}\Vert (\text{curl}_{\hat{\mathbf{x}}}\hat{\omega}(\hat{\mathbf{x}}))^{\top}
(\text{curl}_{\hat{\mathbf{x}}}\hat{\omega}(\hat{\mathbf{x}}))\Vert d\hat{\mathbf{x}}.
\end{align*}
Again, it is easy to see that, if $T$ is a (regular) triangle 
for any $T\in\mathcal{T}_{h}$ and $h>0$, then the 
above inequality holds for any $h>0$. This finishes the proof.
\end{proof}

\begin{lemma}
\label{H_curl_transform_inequality2}
There exist $\delta>0$ and $C>0$ such that,
 for any $h<\delta$ and $T\in\mathcal{T}_{h}$,
\begin{equation*}
\Vert\hat{\omega}\Vert_{H(\text{curl}_{\hat{\mathbf{x}}},\hat{T})}^{2}\leq C
(\Vert\omega\Vert_{L^{2}(T)}^{2}+\tilde{h}_{T}^{2}\Vert\text{curl}\omega\Vert_{L^{2}(T)}^{2})
\quad
\forall \omega\in H^{1}(T;\mathbb{R}^{2})
\end{equation*}
where $\dfrac{B_{T}\hat{\omega}(\hat{\mathbf{x}})}{\det(B_{T})}=
\omega(\mathbf{x})$ for any $\hat{\mathbf{x}}\in\hat{T}$. If $T$ is a triangle for any $T\in\mathcal{T}_{h}$ 
and $h>0$, then the above inequality holds for any $h>0$.
\end{lemma}

\begin{proof}
We have,
\begin{equation*}
\Vert\hat{\omega}\Vert_{H(\text{curl}_{\hat{\mathbf{x}}},\hat{T})}^{2} = 
\Vert\hat{\omega}\Vert_{L^{2}(\hat{T})}^{2}+\Vert\text{curl}_{\hat{\mathbf{x}}}
\hat{\omega}\Vert_{L^{2}(\hat{T})}^{2},
\end{equation*}
and
\begin{align*}
\Vert\hat{\omega}\Vert_{L^{2}(\hat{T})}^{2} = &
\int_{\hat{T}}(\det(B_{T}))^{2}\omega(\mathbf{x})^{\top}B_{T}^{-\top}B_{T}^{-1}\omega(\mathbf{x})d\hat{\mathbf{x}} \\
= & \int_{T}\det(B_{T})\det(B_{T}DG_{T}(\hat{\mathbf{x}})^{-1})\omega(\mathbf{x})^{\top}B_{T}^{-\top}B_{T}^{-1}\omega(\mathbf{x})d\mathbf{x}.
\end{align*}
By Lemma~\ref{geometry_property3}, 
$\lim_{h\rightarrow 0}\sup_{T\in\mathcal{T}_{h}}\sup_{\hat{\mathbf{x}}\in \hat{T}}
\vert \det(B_{T}DG_{T}(\hat{\mathbf{x}})^{-1})-1\vert=0$.

Since $(\mathcal{T}_{h})_{h}$ is regular, there exists a constant $\sigma>0$ 
such that $\sigma_{1}/\sigma_{2}\leq\sigma$ 
for any $T\in\mathcal{T}_{h}$, where $\sigma_1 \geq \sigma_2$ are the singular values of 
matrix $B_{T}$. Then $\Vert B_{T}^{-\top}\Vert=\Vert B_{T}^{-1}\Vert = \sigma_{2}^{-1}$ and 
$\det(B_{T})=\sigma_1\cdot\sigma_{2}$. 
So, $\Vert B_{T}^{-\top}\Vert\cdot\Vert B_{T}^{-1}\Vert\det(B_{T})
=\sigma_1/\sigma_2\leq \sigma$.
Consequently, there exist $\delta_{1}>0$ and $C_{1}>0$ such that,
 for any $h<\delta_{1}$ and $T\in\mathcal{T}_{h}$,
\begin{equation*}
\det(B_{T})\det(B_{T}DG_{T}(\hat{\mathbf{x}})^{-1})\omega(\mathbf{x})^{\top}B_{T}^{-\top}
B_{T}^{-1}\omega(\mathbf{x}) \leq C_{1}^{2}\omega(\mathbf{x})^{\top}\omega(\mathbf{x})
\end{equation*}
for any $\omega\in H^{1}(T;\mathbb{R}^{2}), \mathbf{x}\in T$.
We can conclude that, for any $h<\delta_{1}$ and $T\in\mathcal{T}_{h}$,
\begin{equation*}
\Vert\hat{\omega}\Vert_{L^{2}(\hat{T})}\leq C_{1}\Vert\omega\Vert_{L^{2}(T)},\forall \omega\in H^{1}(T;\mathbb{R}^{2}).
\end{equation*}
At the same time,
\begin{align*}
\text{curl}_{\hat{\mathbf{x}}}\hat{\omega}(\hat{\mathbf{x}})=
&\det(B_{T})B_{T}^{-1}\text{curl}_{\hat{\mathbf{x}}}
\omega(\mathbf{x})\\
=&\det(B_{T})B_{T}^{-1}\text{curl}_{\mathbf{x}}\omega(\mathbf{x})
(DG_{T}(\hat{\mathbf{x}}))^{-\top}\det(DG_{T}(\hat{\mathbf{x}})) \\
= & \det(B_{T})B_{T}^{-1}\text{curl}_{\mathbf{x}}\omega(\mathbf{x})
B_{T}^{-\top}(B_{T}^{\top}DG_{T}(\hat{\mathbf{x}})^{-\top})\det(DG_{T}(\hat{\mathbf{x}}))\\
= & \det(B_{T})B_{T}^{-1}\text{curl}_{\mathbf{x}}\omega(\mathbf{x})
B_{T}^{-\top}(DG_{T}(\hat{\mathbf{x}})^{-1}B_{T})^{\top}\det(DG_{T}(\hat{\mathbf{x}})).
\end{align*}
and
\begin{align*}
\Vert\text{curl}_{\hat{\mathbf{x}}}\hat{\omega}\Vert_{L^{2}(\hat{T})}^{2} = 
&
\int_{\hat{T}} \det(B_{T})^{2}\det(DG_{T}(\hat{\mathbf{x}}))^{2}(DG_{T}(\hat{\mathbf{x}})^{-1}B_{T})
B_{T}^{-1}(\text{curl}_{\mathbf{x}}\omega(\mathbf{x}))^{\top} \\
&\qquad B_{T}^{-\top}B_{T}^{-1}\text{curl}_{\mathbf{x}}\omega(\mathbf{x})
B_{T}^{-\top}(DG_{T}(\hat{\mathbf{x}})^{-1}B_{T})^{\top}d\hat{\mathbf{x}}\\
= & \int_{\hat{T}} \det(B_{T})^{2}\det(DG_{T}(\hat{\mathbf{x}}))(DG_{T}(\hat{\mathbf{x}})^{-1}B_{T})
B_{T}^{-1}(\text{curl}_{\mathbf{x}}\omega(\mathbf{x}))^{\top} \\
&\qquad B_{T}^{-\top}B_{T}^{-1}\text{curl}_{\mathbf{x}}\omega(\mathbf{x})
B_{T}^{-\top}(DG_{T}(\hat{\mathbf{x}})^{-1}B_{T})^{\top}d\mathbf{x} \\
= & \int_{\hat{T}} \det(B_{T})^{3}\det(B_{T}^{-1}DG_{T}(\hat{\mathbf{x}}))(DG_{T}
(\hat{\mathbf{x}})^{-1}B_{T})B_{T}^{-1}(\text{curl}_{\mathbf{x}}\omega(\mathbf{x}))^{\top} \\
&\qquad B_{T}^{-\top}B_{T}^{-1}\text{curl}_{\mathbf{x}}\omega(\mathbf{x})
B_{T}^{-\top}(DG_{T}(\hat{\mathbf{x}})^{-1}B_{T})^{\top}d\mathbf{x}.
\end{align*}
According to Lemma~\ref{geometry_property2}, $\lim_{h\rightarrow 0}\sup_{T\in\mathcal{T}_{h}}
\sup_{\hat{\mathbf{x}}\in\hat{T}}\Vert DG_{T}(\hat{\mathbf{x}})^{-1}B_{T}-I\Vert=0$. 

\noindent By Lemma~\ref{geometry_property3}, $\lim_{h\rightarrow 0}\sup_{T\in\mathcal{T}_{h}}
\sup_{\hat{\mathbf{x}}\in\hat{T}}\vert\det(B_{T}^{-1}DG_{T}(\hat{\mathbf{x}}))-1\vert=0$.

\noindent Since $(\mathcal{T}_{h})_{h}$ is regular, $\det(B_{T})^{3} \Vert B_{T}^{-\top}\Vert^{2}
\cdot\Vert B_{T}^{-1}\Vert^{2}\leq\sigma\tilde{h}_{T}$, for some constant $\sigma>0$.

There exist thus $\delta_{2}>0$ and $C_{2}>0$ such that,
for any $h<\delta_{2}$ and $T\in\mathcal{T}_{h}$,
\begin{align*}
& \det(B_{T})^{3}\det(B_{T}^{-1}DG_{T}(\hat{\mathbf{x}}))(DG_{T}
(\hat{\mathbf{x}})^{-1}B_{T})B_{T}^{-1}(\text{curl}_{\mathbf{x}}\omega(\mathbf{x}))^{\top} \\
&\qquad B_{T}^{-\top}B_{T}^{-1}\text{curl}_{\mathbf{x}}\omega(\mathbf{x})
B_{T}^{-\top}(DG_{T}(\hat{\mathbf{x}})^{-1}B_{T})^{\top}d\mathbf{x} \\
\leq & C_{2}^{2}\tilde{h}_{T}^{2}\Vert (\text{curl}_{\mathbf{x}}\omega(\mathbf{x}))^{\top}
\text{curl}_{\mathbf{x}}\omega(\mathbf{x})\Vert
\quad 
\forall \omega\in H^{1}(T;\mathbb{R}^{2}), \mathbf{x}\in T.
\end{align*}
We conclude that, for any $h<\delta_{2}$ and $T\in\mathcal{T}_{h}$,
\begin{equation*}
\Vert\text{curl}_{\hat{\mathbf{x}}}\hat{\omega}\Vert_{L^{2}(\hat{T})}\leq 
C_{2}\tilde{h}_{T}\Vert\text{curl}\omega\Vert_{L^{2}(T)},\forall \omega\in H^{1}(T;\mathbb{R}^{2}).
\end{equation*}
Again, it is easy to see that,
if  $T$ is a triangle for any $T\in\mathcal{T}_{h}$ and $h>0$, then the 
inequality above holds for any $h>0$. This ends the proof.
\end{proof}

\begin{lemma}
\label{mapping3}
Let $T$ be a curved triangle. Let $u(\mathbf{x})$ be defined on $T$ and 
$\hat{u}(\hat{\mathbf{x}})$ be defined on $\hat{T}$ and
$$
u = \hat{u} \circ G_{T}^{-1} \quad\mbox{or}\quad \hat{u} = u \circ G_T
$$
where $G_{T}$ is the element map.
 Then $u(\mathbf{x})\in L^{2}(T)$ if and only if $\hat{u}(\hat{x})\in L^{2}(\hat{T})$.
\end{lemma}
\begin{proof}
This is an immediate consequence of the fact that $G_{T}$ is a $C^{1}$-diffeomorphism.
\end{proof}

\section{Proofs of results from Section 8\label{app_proofs}}

Proof of Lemma~\ref{L2_non_standard_projection_inequality}
\begin{proof}
We put $r=\tilde{r}(T)$. 
We assume $\{\hat{\xi}_{1}(\hat{\mathbf{x}}),\cdots,\hat{\xi}(\hat{\mathbf{x}})_{l_{r}}\}$ 
is a basis for $\mathcal{P}_{\tilde{r}}\Lambda^{2}(\hat{T})$. 
Then 
$$
\Pi_{\tilde{r},T}^{2}u(\mathbf{x}(\hat{\mathbf{x}}))=\dfrac{1}{\det(DG_{T}(\hat{\mathbf{x}}))}
\sum_{i=1}^{l_{r}}\alpha_{i}\hat{\xi}_{i}(\hat{\mathbf{x}})
$$ 
and 
$$
P_{\tilde{r},T}^{2}u(\mathbf{x}(\hat{\mathbf{x}}))=\dfrac{1}
{\det(DG_{T}(\hat{\mathbf{x}}))}\sum_{i=1}^{l_{r}}\beta_{i}\hat{\xi}_{i}(\hat{\mathbf{x}})
$$
Coefficients $(\alpha_{1},\cdots,\alpha_{l_{r}})^{\top}$ 
and $(\beta_{1},\cdots,\beta_{l_{r}})^{\top}$ 
are obtained by solving the following two linear systems,
$$
A_{1}(\alpha_{1},\cdots,\alpha_{l_{r}})^{\top}=\mathbf{b}_{1}
\quad \mbox{and} \quad  
A_{2}(\beta_{1},\cdots,\beta_{l_{r}})^{\top}=\mathbf{b}_{2}
$$
with 
\begin{align*}
 (A_{1})_{ij}=\int_{T}\dfrac{\hat{\xi}_{i}(\hat{\mathbf{x}}(\mathbf{x}))
\hat{\xi}_{j}(\hat{\mathbf{x}}(\mathbf{x}))d\mathbf{x}}{\det(DG_{T}(\hat{\mathbf{x}}(\mathbf{x})))}\quad
& 
(\mathbf{b}_{1})_{j}=\int_{T}u(\mathbf{x})\hat{\xi}_{j}(\hat{\mathbf{x}}(\mathbf{x}))d\mathbf{x}\\
 (A_{2})_{ij}=\int_{T}\dfrac{\hat{\xi}_{i}(\hat{\mathbf{x}}(\mathbf{x}))
\hat{\xi}_{j}(\hat{\mathbf{x}}(\mathbf{x}))d\mathbf{x}}{(\det(DG_{T}(\hat{\mathbf{x}}(\mathbf{x}))))^{2}}
 \quad &
(\mathbf{b}_{2})_{j}=\int_{T}\dfrac{\hat{\xi}_{j}(\hat{\mathbf{x}}(\mathbf{x}))d\mathbf{x}}{\det(DG_{T}(\hat{\mathbf{x}}(\mathbf{x})))}u(\mathbf{x})
\end{align*}
for $1\leq i,j\leq l_{r}$.
By pulling back to $\hat{T}$ we obtain, for any $1\leq i,j\leq l_{r}$,
\begin{align*}
 (A_{1})_{ij}=\int_{\hat{T}}\hat{\xi}_{i}(\hat{\mathbf{x}})
\hat{\xi}_{j}(\hat{\mathbf{x}})d\hat{\mathbf{x}} \quad &
(\mathbf{b}_{1})_{j}=\det(B_{T})\int_{\hat{T}}\det(B_{T}^{-1}
DG_{T}(\hat{\mathbf{x}}))u(\mathbf{x})\hat{\xi}_{j}(\hat{\mathbf{x}}(\mathbf{x}))d\hat{\mathbf{x}}\\
  (A_{2})_{ij}=\int_{\hat{T}}\dfrac{\hat{\xi}_{i}(\hat{\mathbf{x}})\hat{\xi}_{j}(\hat{\mathbf{x}})d\hat{\mathbf{x}}}{\det(DG_{T}(\hat{\mathbf{x}}))} \quad &
(\mathbf{b}_{2})_{j}=\int_{\hat{T}}u(\mathbf{x}(\mathbf{x}))\hat{\xi}_{j}(\hat{\mathbf{x}})d\hat{\mathbf{x}}
\end{align*}

Since $\det(B_{T})$ is a non-zero constant, $(\det(B_{T})A_{2})(\beta_{1},\cdots,\beta_{l_{r}})^{\top}=\det(B_{T})\mathbf{b}_{2}$.
So we can redefine $A_{2}$ and $\mathbf{b}_{2}$ in the following way.
\begin{equation*}
(A_{2})_{ij}=\int_{\hat{T}}\dfrac{\det(B_{T})}{\det(DG_{T}(\hat{\mathbf{x}}))}\hat{\xi}_{i}(\hat{\mathbf{x}})
\hat{\xi}_{j}(\hat{\mathbf{x}})d\hat{\mathbf{x}} \quad
(\mathbf{b}_{2})_{j}=\det(B_{T})\int_{\hat{T}}u(\mathbf{x}(\hat{\mathbf{x}}))\hat{\xi}_{j}(\hat{\mathbf{x}})d\hat{\mathbf{x}}
\end{equation*} 
According to Lemma~\ref{geometry_property3}, $\lim_{h\rightarrow 0}\sup_{T\in\mathcal{T}_{h}}\Vert A_{1}-A_{2}\Vert=0$. 
And, for any $\varepsilon>0$, 
\begin{align*}
 \lim_{h\rightarrow 0}\sup_{T\in\mathcal{T}_{h}}\Vert \mathbf{b}_{1}-\mathbf{b}_{2}\Vert^{2}
\leq & (\varepsilon\det(B_{T}))^{2}\int_{\hat{T}}u^{2}(\mathbf{x}(\hat{\mathbf{x}}))d\hat{\mathbf{x}}\\
= & \varepsilon^{2}\det(B_{T})\int_{T}\det(B_{T}DG_{T}(\hat{\mathbf{x}}(\mathbf{x}))^{-1})u^{2}
(\mathbf{x})d\mathbf{x}\\
\leq & 4\varepsilon^{2}\det(B_{T})\Vert u\Vert_{L^{2}(T)}^{2}.
\end{align*}
The last inequality holds when $h$ is small enough. 
This implies that 
$$
\lim_{h\rightarrow 0}\sup_{T\in\mathcal{T}_{h}}\Vert \mathbf{b}_{1}-\mathbf{b}_{2}\Vert
\leq 2\varepsilon\sqrt{\det(B_{T})}\Vert u\Vert_{L^{2}(T)}.
$$
So, for any $\varepsilon>0$, there exists $\delta>0$ such that, for 
any $h\leq\delta$ and $T\in\mathcal{T}_{h}$, we have
\begin{equation*}
\Vert (\alpha_{1}-\beta_{1},\cdots,\alpha_{l_{r}}-\beta_{l_{r}})\Vert\leq 3\varepsilon\sqrt{\det(B_{T})}\Vert u\Vert_{L^{2}(T)}.
\end{equation*}
We have then
\begin{align*}
 \Vert\Pi_{\tilde{r},T}^{2}u-P_{\tilde{r},T}u\Vert_{L^{2}(T)}^{2} = &
\int_{T}\dfrac{1}{(\det(DG_{T}(\hat{\mathbf{x}}(\mathbf{x}))))^{2}}\sum_{i=1}^{l_{r}}\hat{\xi}_{i}^{2}(\hat{\mathbf{x}}
(\mathbf{x}))(\alpha_{i}-\beta_{i})^{2}d\mathbf{x}\\
= & \int_{\hat{T}}\dfrac{1}{\det(DG_{T}(\hat{\mathbf{x}}))}\sum_{i=1}^{l_{r}}\hat{\xi}_{i}^{2}(\hat{\mathbf{x}})(\alpha_{i}-\beta_{i})^{2}d\hat{\mathbf{x}}\\
\leq & c\varepsilon^{2}\Vert u\Vert_{L^{2}(T)}^{2}\int_{\hat{T}}\det(B_{T}DG_{T}(\hat{\mathbf{x}})^{-1})\sum_{i=1}^{l_{r}}\hat{\xi}_{i}^{2}(\hat{\mathbf{x}})d\hat{\mathbf{x}}
\end{align*}
Here $c$ is a positive constant which depends on $l_{r}$ only.
By Lemma~\ref{geometry_property3}, there exists $M>0$ such that,
for any $h$ small enough and any $T\in\mathcal{T}_{h}$,
\begin{equation*}
\int_{\hat{T}}\det(B_{T}DG_{T}(\hat{\mathbf{x}})^{-1})
\sum_{i=1}^{l_{r}}\hat{\xi}_{i}^{2}(\hat{\mathbf{x}})
d\hat{\mathbf{x}}\leq M^{2}
\end{equation*}
We can conclude that,
 for any $\varepsilon>0$, there exists $\delta>0$ such that, for any $h<\delta$ 
and $T\in\mathcal{T}_{h}$, 
\begin{equation*}
\Vert\Pi_{\tilde{r},T}^{2}u-P_{\tilde{r},T}u\Vert_{L^{2}(T)} \leq\varepsilon\Vert u\Vert_{L^{2}(T)},\forall u\in L^{2}(T)
\end{equation*}
Here $P_{\tilde{r},T}$ is the standard $L^2$-projection onto $\mathcal{P}_{\tilde{r}}\Lambda^{2}(T)$.
\end{proof}

Proof of Lemma~\ref{H_div_transform_curved}
\begin{proof}
Obviously, $\hat{\omega}(\hat{\mathbf{x}})=\det(DG_{T}(\hat{\mathbf{x}}))
(DG_{T}(\hat{\mathbf{x}}))^{-1}
\omega(G_{T}(\hat{\mathbf{x}}))$.
Using Lemma~\ref{mapping1}, we can conclude that $\hat{\omega}(\hat{\mathbf{x}})\in H(\hat{T})$. 

By pulling back to $\hat{T}$ and using 
the definition of $\mathcal{P}_{\tilde{r}+1}\Lambda^{1}(T)$, we can see that 
(\ref{interior_aux_curved}) is the same as (\ref{interior_aux_reference}), and (\ref{interior_div_curved}) is the same as 
(\ref{interior_div_reference}).
Thus, we only need to show that (\ref{H_div_edge_curved}) is the same as (\ref{H_div_edge_reference}).
Since $\mathcal{T}_{h}$ is $C^{0}$-compatible, then $G_{T}(\mathbf{\zeta}(s)) = \mathbf{x}_{e}(s)$ for 
any $s\in [0,1]$ where $\zeta: [0,1]\rightarrow\hat{e}$ is an affine local parameterization of $\hat{e}$.

\noindent We have then
\begin{align*}
& \int_{[0,1]}[\omega(\mathbf{x}_{e}(s))
\cdot\mathbf{n}(\mathbf{x}_{e}(s))]\hat{\eta}(s)\Vert \dot{\mathbf{x}_{e}}(s)\Vert ds \\
= & \int_{[0,1]}[\omega(G_{T}(\mathbf{\zeta}(s)))\cdot\mathbf{n}(G_{T}(\mathbf{\zeta}(s)))]
\hat{\eta}(s)\Vert DG_{T}(\mathbf{\zeta}(s))\dot{\mathbf{\zeta}}(s)\Vert ds \\
= & \int_{[0,1]}\left[\dfrac{DG_{T}(\zeta(s))\hat{\omega}(\zeta(s))}{\det(DG_{T}(\zeta(s)))}\cdot
\dfrac{(DG_{T}(\zeta(s)))^{-\top}\hat{\mathbf{n}}(\zeta(s))}
{\Vert (DG_{T}(\zeta(s)))^{-\top}\hat{\mathbf{n}}(\zeta(s)) \Vert}\right]\hat{\eta}(s)
\Vert DG_{T}(\mathbf{\zeta}(s))\dot{\mathbf{\zeta}}(s)\Vert ds \\
= & \int_{[0,1]}\left[\dfrac{\hat{\omega}(\zeta(s))}{\det(DG_{T}(\zeta(s)))}\cdot
\dfrac{\hat{\mathbf{n}}(\zeta(s))}
{\Vert (DG_{T}(\zeta(s)))^{-\top}\hat{\mathbf{n}}(\zeta(s)) \Vert}\right]\hat{\eta}(s)
\Vert DG_{T}(\mathbf{\zeta}(s))\dot{\mathbf{\zeta}}(s)\Vert ds
\end{align*}
Notice that ${\mathbf{\zeta}}(s)=c\hat{t}$ where $\hat{t}$ is a unit tangent vector along $\hat{e}$, 
and $c$ is a nonzero constant, $\hat{\mathbf{n}}(\zeta(s))=(\hat{t}_{2},-\hat{t}_{1})^{\top}$, 
and $(DG_{T}(\zeta(s)))^{-\top}=\dfrac{A}{\det(DG_{T}(\zeta(s)))}$ with 
\begin{equation*}
A=
\left[ 
\begin{tabular}
[c]{ll}
$(DG_{T})_{22}(\zeta(s))$ & $-(DG_{T})_{21}(\zeta(s))$\\
$-(DG_{T})_{12}(\zeta(s))$ & $(DG_{T})_{11}(\zeta(s))$
\end{tabular}
\right]
\end{equation*}
Therefore, we have
\begin{align*}
& \int_{[0,1]}\left[\dfrac{\hat{\omega}(\zeta(s))}{\det(DG_{T}(\zeta(s)))}\cdot
\dfrac{\hat{\mathbf{n}}(\zeta(s))}
{\Vert (DG_{T}(\zeta(s)))^{-\top}\hat{\mathbf{n}}(\zeta(s)) \Vert}\right]\hat{\eta}(s)
\Vert DG_{T}(\mathbf{\zeta}(s))\dot{\mathbf{\zeta}}(s)\Vert ds \\
= &  c\int_{[0,1]}\left[\hat{\omega}(\zeta(s))\cdot\hat{\mathbf{n}}(\zeta(s))
\right]\hat{\eta}(s) ds 
\end{align*}
We conclude that (\ref{H_div_edge_curved}) is equivalent with (\ref{H_div_edge_reference}). 
This finishes the proof.
\end{proof}

Proof of Lemma~\ref{commuting_diagram1_curved}
\begin{proof}
By Lemma~\ref{mapping1}, we have 
$\hat{\omega}(\hat{\mathbf{x}})\in H(\hat{T})$, 
and 
$$
\text{div}\omega(\mathbf{x}(\hat{\mathbf{x}}))=
\dfrac{1}{\det(DG_{T}(\hat{\mathbf{x}}))}\text{div}_{\hat{\mathbf{x}}}\hat{\omega}(\hat{\mathbf{x}})
$$ 
for $\hat{\mathbf{x}}\in\hat{T}$ almost everywhere, provided we define 
$\hat{\omega}(\hat{\mathbf{x}})$ on $\hat{T}$ by 
$$
\omega(\mathbf{x}(\hat{\mathbf{x}}))=
\dfrac{DG_{T}(\hat{\mathbf{x}})}
{\det(DG_{T}(\hat{\mathbf{x}}))}\hat{\omega}(\hat{\mathbf{x}})
$$ 
for any $\hat{\mathbf{x}}\in\hat{T}$.

Using Definition~\ref{L2_non_orthogonal}, Definition~\ref{L2_non_orthogonal_reference}, 
Lemma~\ref{L2_transform}, 
Lemma~\ref{H_div_transform_curved}, and Lemma~$10$ in \cite{QD:2009:MMEW}, it is easy to see that 
\begin{align*}
& \Pi_{\tilde{r},T}^{2}\text{div}\omega(\mathbf{x}(\hat{\mathbf{x}}))=\dfrac{1}{\det(DG_{T}(\hat{\mathbf{x}}))}\Pi_{\tilde{r},\hat{T}}^{2}\text{div}_{\hat{\mathbf{x}}}\hat{\omega}(\hat{\mathbf{x}})=\dfrac{1}{\det(DG_{T}(\hat{\mathbf{x}}))}\text{div}_{\hat{\mathbf{x}}}\Pi_{\tilde{r}+1,\hat{T}}^{1}\hat{\omega}(\hat{\mathbf{x}})\\
& \text{div}\Pi_{\tilde{r}+1,T}^{1}\omega(\mathbf{x}(\hat{\mathbf{x}}))=\text{div}\left[\dfrac{DG_{T}(\hat{\mathbf{x}})}
{\det(DG_{T}(\hat{\mathbf{x}}))}\Pi_{\tilde{r}+1,\hat{T}}^{1}\hat{\omega}(\hat{\mathbf{x}})\right] =
\dfrac{1}{\det(DG_{T}(\hat{\mathbf{x}}))}\text{div}_{\hat{\mathbf{x}}}\Pi_{\tilde{r}+1,\hat{T}}^{1}\hat{\omega}(\hat{\mathbf{x}})
\end{align*}
We have thus $\Pi_{\tilde{r},T}^{2}\text{div}\omega = \text{div}\Pi_{\tilde{r}+1,T}^{1}\omega$.
\end{proof}

Proof of Lemma~\ref{H_div_inequality1_curved}
\begin{proof}
According to Lemma~\ref{H_div_transform_inequality1} and Lemma~\ref{H_div_transform_inequality2}, 
there exist $\delta>0$ and $C_{1}>0$ such that, for any $h<\delta$ and $T\in\mathcal{T}_{h}$,
\begin{equation*}
\Vert\Pi_{\tilde{r},T}^{1}\omega\Vert_{L^{2}(T)}\leq C_{1}\Vert\Pi_{\tilde{r},\hat{T}}^{1}\hat{\omega}\Vert_{L^{2}(\hat{T})}\quad
\forall \omega\in L^{2}(T;\mathbb{R}^{2})
\end{equation*}
\begin{equation*}
\Vert\hat{\omega}\Vert_{H(\text{div}_{\hat{\mathbf{x}}},T)}^{2}
+\Vert\hat{\omega}\Vert_{L^{2}(\partial\hat{T})}^{2}
\leq C_{1}\Vert\omega\Vert_{H^{1}(T)}\quad
\forall \omega\in H^{1}(T;\mathbb{R}^{2})
\end{equation*}
By Lemma~\ref{H_div_transform_curved}, $\hat{\omega}\in H^{1}(\hat{T};\mathbb{R}^{2})$ for 
any $\omega\in H^{1}(T;\mathbb{R}^{2})$.

The definition of operator $\Pi_{\tilde{r},\hat{T}}$ implies that there exists a constant $C_{2}>0$ such that
\begin{equation*}
\int_{\hat{T}}(\Pi_{\tilde{r},\hat{T}}^{1}\hat{\omega}(\hat{\mathbf{x}}))^{\top}\Pi_{\tilde{r},\hat{T}}^{1}\hat{\omega}(\hat{\mathbf{x}})d\hat{\mathbf{x}}\leq C_{2}(\Vert\hat{\omega}\Vert_{H(\text{div}_{\hat{\mathbf{x}}},\hat{T})}^{2}
+\Vert \hat{\omega}\Vert_{L^{2}(\partial(\hat{T}))}^{2}),\forall \hat{\omega}(\hat{\mathbf{x}})\in [H^{1}(\hat{T})]^{2}.
\end{equation*}
It is easy to see that for affine meshes
the above inequality holds for any $h>0$. This finishes the proof.
\end{proof}

Proof of Theorem~\ref{W_inequality_curved}
\begin{proof}
For any $h>0$ and any $T\in\mathcal{T}_{h}$,  we define a linear isomorphism $A_{T}$ from 
$H^{1}(\hat{T};\mathbb{R}^{2})$ to $H^{1}(T;\mathbb{R}^{2})$ by 
$(A_{T}\hat{\omega})(\mathbf{x}(\hat{\mathbf{x}}))=\dfrac{B_{T}\hat{\omega}(\hat{\mathbf{x}})}{\det(B_{T})}$.
It is easy to see that $A_{T}$ is a linear isomorphism from 
$\mathcal{P}_{\tilde{r}+2}\Lambda^{0}(\hat{T};\mathbb{R}^{2})$ to 
$\mathcal{P}_{\tilde{r}+2}\Lambda^{0}(T;\mathbb{R}^{2})$.

We define an operator $E_{T}:H^{1}(\hat{T};\mathbb{R}^{2})\longrightarrow \mathcal{P}_{\tilde{r}+2}
\Lambda^{0}(\hat{T};\mathbb{R}^{2})$ by $W_{T}(A_{T}\hat{\omega})=A_{T}(E_{T}\hat{\omega})$.
Obviously, $W_{T}$ is well-defined if and only if $E_{T}$ is well-defined.
We denote by $\{\hat{\xi}_{1},\cdots,\hat{\xi}_{l_{\tilde{r}}}\}$ 
a particular basis of $\mathcal{P}_{\tilde{r}+2}\Lambda^{0}
(\hat{T};\mathbb{R}^{2})$.

According to Lemma~\ref{H_curl_transform_inequality1} and Lemma~\ref{H_curl_transform_inequality2}, 
it is sufficient to show that there exist $\delta>0$ and $C_{1}>0$ such that,
for any $h<\delta$ and 
$T\in\mathcal{T}_{h}$, $E_{T}$ is well-defined,
and $\Vert (z_1,\cdots,z_{l_{\tilde{r}}})\Vert\leq 
C_{1}\Vert \hat{\omega}\Vert_{H^{1}(\hat{T})}$ for any $\hat{\omega}$.
Here $\sum_{k=1}^{l_{\tilde{r}}}z_{k}\hat{\xi}_{k} = E_{T}\hat{\omega}$.

According to the definition of $W_{T}$, $E_{T}$ can be defined by relations
\begin{equation}
\int_{T}\text{div}(A_{T}E_{T}\hat{\omega}-A_{T}\hat{\omega})(\mathbf{x})
\hat{\psi}(\hat{\mathbf{x}}(\mathbf{x}))d\mathbf{x}=0\quad
\forall \hat{\psi}
\in\mathcal{P}_{\tilde{r}(T)}(\hat{T})/\mathbb{R}
\label{E_interior_div}
\end{equation}
\begin{equation}
\int_{T}((A_{T}E_{T}\hat{\omega})(\mathbf{x})-(A_{T}\hat{\omega})(\mathbf{x}))^{\top}
DG_{T}(\hat{\mathbf{x}}(\mathbf{x}))^{-\top}
\hat{\mathbf{h}}_{i}(\hat{\mathbf{x}}(\mathbf{x}),t_{\tilde{r}(\hat{T})}) d\mathbf{x}=0\quad
1\leq i\leq k_{\tilde{r}}
\label{E_interior_aux}
\end{equation}
\begin{equation}
\int_{[0,1]}[(A_{T}E_{T}\hat{\omega}-A_{T}\hat{\omega})(\mathbf{x}_{e}(s))
\cdot\mathbf{n}(\mathbf{x}_{e}(s))]\hat{\eta}(s)\Vert \dot{\mathbf{x}_{e}}(s)\Vert ds = 0\quad
\forall \hat{\eta}\in\mathcal{P}_{\tilde{r}(e)}([0,1]), \forall e\in\triangle_{1}(T)
\label{E_edge_normal}
\end{equation}
\begin{equation}
\int_{[0,1]}[(A_{T}E_{T}\hat{\omega}-A_{T}\hat{\omega})(\mathbf{x}_{e}(s))
\cdot\mathbf{t}(\mathbf{x}_{e}(s))]\hat{\eta}(s)\Vert \dot{\mathbf{x}_{e}}(s)\Vert ds = 0\quad
\forall \hat{\eta}\in\mathcal{P}_{\tilde{r}(e)}([0,1]), \forall e\in\triangle_{1}(T)
\label{E_edge_tangent}
\end{equation}
\begin{equation}
E_{T}\hat{\omega} = 0 \text{ at all vertices of }\hat{T}
\label{E_vertex}
\end{equation}
Denote:
$$
\left[ \begin{array}{cc} b_{11} & b_{12}\\b_{21} & b_{22}\end{array} \right]=B_{T},
J=\det(B_{T}), 
\left[\begin{array}{c} \hat{u}_{1} \\ \hat{u}_{2} \end{array}\right]=E_{T}\hat{\omega},
\left[ \begin{array}{c} \hat{w}_{1}\\ \hat{w}_{2} \end{array} \right]=\hat{\omega}, 
\hat{u}_{i,j}=\dfrac{\partial \hat{u}_{i}}{\partial\hat{x}_{j}},
\hat{w}_{i,j}=\dfrac{\partial \hat{w}_{i}}{\partial\hat{x}_{j}}.
$$
By pulling back to $\hat{T}$,  $E_{T}$ can be defined by relations
\begin{align}
\label{W_interior_div_reference}
& \int_{\hat{T}}J^{-1}[(b_{11}(DG_{T})_{22}-b_{21}(DG_{T})_{12})\hat{u}_{1,1}
+(b_{12}(DG_{T})_{22}-b_{22}(DG_{T})_{12})\hat{u}_{2,1}\\ \nonumber
& +(b_{21}(DG_{T})_{11}-b_{11}(DG_{T})_{21})\hat{u}_{1,2}
+(b_{22}(DG_{T})_{11}-b_{12}(DG_{T})_{21})\hat{u}_{2,2}] 
\hat{\psi}(\hat{\mathbf{x}})d\hat{\mathbf{x}}\\ \nonumber
= & \int_{\hat{T}}J^{-1}[(b_{11}(DG_{T})_{22}-b_{21}(DG_{T})_{12})\hat{w}_{1,1}
+(b_{12}(DG_{T})_{22}-b_{22}(DG_{T})_{12})\hat{w}_{2,1}\\ \nonumber
& +(b_{21}(DG_{T})_{11}-b_{11}(DG_{T})_{21})\hat{w}_{1,2}
+(b_{22}(DG_{T})_{11}-b_{12}(DG_{T})_{21})\hat{w}_{2,2}] 
\hat{\psi}(\hat{\mathbf{x}})d\hat{\mathbf{x}},\\ \nonumber
& \forall \hat{\psi}
\in\mathcal{P}_{\tilde{r}(T)}(\hat{T})/\mathbb{R}\nonumber
\end{align}
\begin{align}
\label{W_interior_aux_reference}
& \int_{\hat{T}}(E_{T}\hat{\omega}(\hat{\mathbf{x}}))^{\top}
B_{T}^{\top}DG_{T}(\hat{\mathbf{x}})^{-\top}
\hat{\mathbf{h}}_{i}(\hat{\mathbf{x}},t) \det(B_{T}^{-1}DG_{T}(\hat{\mathbf{x}}))d\hat{\mathbf{x}} \\ \nonumber
= & \int_{\hat{T}}(\hat{\omega}(\hat{\mathbf{x}}))^{\top}
B_{T}^{\top}DG_{T}(\hat{\mathbf{x}})^{-\top}
\hat{\mathbf{h}}_{i}(\hat{\mathbf{x}},t) \det(B_{T}^{-1}DG_{T}(\hat{\mathbf{x}}))d\hat{\mathbf{x}} \quad
1\leq i\leq k_{\tilde{r}} \nonumber
\end{align}
\begin{align}
\label{W_edge_normal_reference}
& \int_{[0,1]}[B_{T}E_{T}\hat{\omega}(\zeta(s))\cdot(DG_{T}(\zeta(s))^{-\top}
\hat{\mathbf{n}}(\zeta(s)))]\hat{\eta}(s)\det(B_{T}^{-1}DG_{T}(\zeta(s)) ds \\ \nonumber
= & \int_{[0,1]}[B_{T}\hat{\omega}(\zeta(s))\cdot(DG_{T}(\zeta(s))^{-\top}\hat{\mathbf{n}}(\zeta(s)))]
\hat{\eta}(s)\det(B_{T}^{-1}DG_{T}(\zeta(s)) ds \\ \nonumber
& \forall \hat{\eta}\in\mathcal{P}_{\tilde{r}(e)}([0,1]), \forall e\in\triangle_{1}(T)
\end{align}
\begin{align}
\label{W_edge_tangent_reference}
& \int_{[0,1]}[B_{T}E_{T}\hat{\omega}(\zeta(s))
\cdot(DG_{T}(\zeta(s))\dot{\zeta}(s))]\hat{\eta}(s) \det(B_{T}^{-1}) ds \\ \nonumber
= & \int_{[0,1]}[B_{T}\hat{\omega}(\zeta(s))
\cdot(DG_{T}(\zeta(s))\dot{\zeta}(s))]\hat{\eta}(s) \det(B_{T}^{-1}) ds\quad
\forall \hat{\eta}\in\mathcal{P}_{\tilde{r}(e)}([0,1]), \forall e\in\triangle_{1}(T)
\end{align}
\begin{equation}
E_{T}\hat{\omega} = 0 \text{ at all vertices of }\hat{T}
\label{W_vertex_reference}
\end{equation}
It is easy to see that (\ref{W_interior_div_reference}) comes from (\ref{E_interior_div}), 
(\ref{W_interior_aux_reference}) comes from (\ref{E_interior_aux}), (\ref{W_edge_tangent_reference}) 
comes from (\ref{E_edge_tangent}), and (\ref{W_vertex_reference}) comes from (\ref{E_vertex}). 
And (\ref{W_edge_normal_reference}) can be got from (\ref{E_edge_normal}) by using the fact that 
$\Vert\dot{\mathbf{x}_{e}}(s)\Vert=\Vert DG_{T}(\zeta(s))\hat{\mathbf{t}}\Vert=
c\Vert DG_{T}(\zeta(s))^{-\top}\hat{\mathbf{n}}\Vert$ for some non-zero constant $c$, 
which comes from direct calculation.

Notice that vector $\dot{\zeta}(s)$ is constant tangent vector along each edge of $\hat{T}$.
Set 
$$
a=\hat{\mathbf{n}}^{\top}B_{T}^{\top}B_{T}\dot{\zeta}(s)\det(B_{T}^{-1}), \quad
b=\dfrac{\det(B_{T})\Vert\dot{\zeta}(s)\Vert}{\dot{\zeta}(s)^{\top}B_{T}^{\top}B_{T}\dot{\zeta}(s)}
$$
Obviously, $b\neq 0$.

Perform now the operation: $b\times [(\ref{W_edge_tangent_reference})-a\times (\ref{W_edge_normal_reference})]$.
We have,
\begin{align}
\label{W_edge_tangent_modified_reference}
\int_{[0,1]}[E_{T}\hat{\omega}(\zeta(s))\cdot F_{T}(s)]\hat{\eta}(s)ds 
& = \int_{[0,1]}[\hat{\omega}(\zeta(s))\cdot F_{T}(s)]\hat{\eta}(s)ds\\
& \forall \hat{\eta}\in\mathcal{P}_{\tilde{r}(e)}([0,1]), \forall e\in\triangle_{1}(T)\nonumber
\end{align}
where 
\begin{align*}
F_{T}(s) & = \det(B_{T}^{-1})[B_{T}^{\top}(DG_{T}(\zeta(s)))\dot{\zeta}(s) \\
& -\det(B_{T}^{-1}DG_{T}(\zeta(s)))(\hat{\mathbf{n}}^{\top}
B_{T}^{T}B_{T}\dot{\zeta}(s))B_{T}^{\top}(DG_{T}(\zeta(s)))^{-\top}\hat{\mathbf{n}})] 
 \dfrac{\det(B_{T})\Vert\dot{\zeta}(s)\Vert}{\dot{\zeta}(s)^{\top}B_{T}^{\top}B_{T}\dot{\zeta}(s)}
\end{align*}
Then the definition of operator $E_{T}$ can be rewritten by using conditions 
(\ref{W_interior_div_reference}),(\ref{W_interior_aux_reference}),
(\ref{W_edge_normal_reference}),
(\ref{W_edge_tangent_modified_reference}), and~(\ref{W_vertex_reference}).

Using the fact that $\hat{\mathbf{n}}\bot\dot{\zeta}(s)$, 
Lemmas~\ref{geometry_property1},\ref{geometry_property2},\ref{geometry_property3}, 
and the assumption that $(\mathcal{T}_{h})_{h}$ is regular, we obtain
\begin{equation*}
\lim_{h\rightarrow 0}\sup_{T\in\mathcal{T}_{h}}\sup_{e\in \triangle_{1}(T)}\sup_{s\in [0,1]}
\Vert F_{T}(s)\cdot\dfrac{\dot{\zeta}(s)}{\Vert\dot{\zeta}(s)\Vert}-1\Vert=
\lim_{h\rightarrow 0}\sup_{T\in\mathcal{T}_{h}}\sup_{e\in \triangle_{1}(T)}\sup_{s\in [0,1]}
\Vert F_{T}(s)\cdot\hat{\mathbf{n}}\Vert=0
\end{equation*} 
Consequently,
$$
\lim_{h\rightarrow 0}\sup_{T\in\mathcal{T}_{h}}\sup_{e\in \triangle_{1}(T)}\sup_{s\in [0,1]}
\Vert F_{T}(s) - \hat{\mathbf{t}}\Vert=0
$$
We denote now by $E(T,\tilde{r})$ the matrix corresponding to the left-hand side of 
conditions~(\ref{W_interior_div_reference}),(\ref{W_interior_aux_reference}),(\ref{W_edge_normal_reference}),
(\ref{W_edge_tangent_modified_reference}),(\ref{W_vertex_reference}), a particular basis of 
$\mathcal{P}_{\tilde{r}+2}\Lambda^{0}(\hat{T};\mathbb{R}^{2})$ (the solution space), some
basis of 
$\mathcal{P}_{\tilde{r}(T)}(\hat{T})/\mathbb{R}$, and some basis of $\mathcal{P}_{\tilde{r}(e)}([0,1])$ for each 
$e\in\triangle_{1}(T)$. 
We denote by $\{\hat{\xi}_{1},\cdots,\hat{\xi}_{l_{\tilde{r}}}\}$ a basis for 
$\mathcal{P}_{\tilde{r}+2}\Lambda^{0}(\hat{T};\mathbb{R}^{2})$. 
Finally, we denote by $C(\tilde{r})$ the matrix corresponding to the left-hand side 
of conditions (\ref{C_interior_div}),
(\ref{C_interior_aux}),(\ref{C_edge_normal}), and~(\ref{C_edge_tangent}), 
and the same bases as above.

Using the fact that 
$$
\lim_{h\rightarrow 0}\sup_{T\in\mathcal{T}_{h}}\sup_{e\in \triangle_{1}(T)}\sup_{s\in [0,1]}
\Vert F_{T}(s) - \hat{\mathbf{t}}\Vert=0
$$ 
and Lemmas~\ref{geometry_property1},\ref{geometry_property2},
\ref{geometry_property3}, we conclude that, for any $t\in [0,1]$, 
$$
\lim_{h\rightarrow 0}\sup_{T\in\mathcal{T}_{h}}\Vert E(T,\tilde{r}) - C(\tilde{r})\Vert=0$$
Then, for any given $t\in [0,1]$, and any given $\tilde{r}$ with 
non-singular $C(\tilde{r})$, the matrix $E(T,\tilde{r})$ 
is non-singular for any $T\in\mathcal{T}_{h}$ when $h>0$ small enough. 
Notice that the right-hand sides of conditions~(\ref{W_interior_div_reference}),
(\ref{W_interior_aux_reference}),(\ref{W_edge_normal_reference}),(\ref{W_edge_tangent_modified_reference}), and~(\ref{W_vertex_reference}) 
are continuous linear functionals of $\hat{\omega}\in H^{1}(\hat{T};\mathbb{R}^{2})$.
We can conclude thus that the operator $W_{T,t}$ is well-defined 
for any $T\in\mathcal{T}_{h}$ with small enough $h$. 
 
Since for any $t\in [0,1]$, $\lim_{h\rightarrow 0}\sup_{T\in\mathcal{T}_{h}}\Vert E(T,\tilde{r}) - C(\tilde{r})\Vert=0$,
and the matrix $C(\tilde{r})$ depends only on $\tilde{r}$,
we can conclude that there exists $C_{1}>0$ such that when $h>0$ small enough, then
$\Vert (z_1,\cdots,z_{l_{\tilde{r}}})\Vert \leq C_{1}\Vert \hat{\omega}\Vert_{H^{1}(\hat{T})}$ for 
any $T\in\mathcal{T}_{h}$. Here $\sum_{k=1}^{l_{\tilde{r}}}z_{k}\hat{\xi}_{k} = E_{T}\hat{\omega}$.

Finally, it is easy to see that the operator $W_{T}$ will be well-defined and the inequality in the statement of this theorem 
holds for any $h>0$ for affine meshes.
This finishes the proof.
\end{proof}

\bibliographystyle{siam}
\bibliography{refer1}

\end{document}